\documentclass[11pt]{article}
\usepackage{caption}
\usepackage{url}
\usepackage{verbatim}
\usepackage{enumitem}
\usepackage{amsmath}
\usepackage[usenames,dvipsnames]{color}
\usepackage{amsthm}
\usepackage{amsfonts}
\usepackage{graphicx}
\usepackage{nicefrac}
\usepackage{url}
\usepackage{etoolbox}
\usepackage{cmbright}
\usepackage[]{algorithm2e}
\usepackage{graphicx, amssymb, subcaption,graphics}
\def\captionfont{\setb@se{11pt}\protect\footnotesize}
\def\captionfont{\protect\footnotesize}
\usepackage{tikz}
\usetikzlibrary{arrows}
\usepackage{verbatim}
\newcommand{\iprd}[2]{\left( #1 , #2 \right)}

\newcommand{\vertiii}[1]{{\left\vert\kern-0.25ex\left\vert\kern-0.25ex\left\vert #1 \right\vert\kern-0.25ex\right\vert\kern-0.25ex\right\vert}}

\newcommand{\msfT}{\mathsf{T}}

\def\norm#1#2{\left\| #1 \right\|_{#2}}

\newcommand{\mD}{\mathfrak D}
\newcommand{\md}{\mathfrak d}

\newcommand{\x}{\chi}

\newcommand{\hf}{\frac{1}{2}}

\def\0{\mbox{\boldmath $0$}}

\newcommand{\nrm}[1]{\left\| #1 \right\|}
\newcommand{\ciptwo}[2]{\left( #1 , #2 \right)}

\newcommand{\eipns}[2]{\left[ #1 , #2 \right]_{\rm ns}}
\newcommand{\eipew}[2]{\left[ #1 , #2 \right]_{\rm ew}}
\newcommand{\viptwo}[2]{\left\langle #1, #2 \right\rangle}
\newcommand{\nabh}{\nabla_{\! h}}

\newcommand{\monenrm}[1]{\left\| #1 \right\|_{-1}}
\newcommand{\moneinn}[1]{\left( #1 \right)_{-1}}

\newtheorem{thm}{Theorem}[section]
\newtheorem{prop}[thm]{Proposition}
\newtheorem{cor}[thm]{Corollary}

\newtheorem{rmk}[thm]{Remark}
\newtheorem{lem}[thm]{Lemma}

\begin{document}

\title{Preconditioned Steepest Descent Methods for some Nonlinear Elliptic Equations Involving  p-Laplacian Terms}

\author{
Wenqiang Feng\thanks{Department of Mathematics, The University of Tennessee, Knoxville, TN 37996 (wfeng1@vols.utk.edu)}
\and
Abner J. Salgado\thanks{Department of Mathematics, The University of Tennessee, Knoxville, TN 37996 (asalgad1@utk.edu)}
\and
Cheng Wang\thanks{Department of Mathematics, The University of Massachusetts, North Dartmouth, MA  02747 (cwang1@umassd.edu)}	
\and
Steven M. Wise\thanks{Corresponding author: Department of Mathematics, The University of Tennessee, Knoxville, TN 37996 (swise1@utk.edu)}
}

\maketitle
\numberwithin{equation}{section}

	\begin{abstract}
We describe and analyze preconditioned steepest descent (PSD) solvers for fourth and sixth-order nonlinear elliptic equations that include p-Laplacian terms on periodic domains in 2 and 3 dimensions. The highest and lowest order terms of the equations are constant-coefficient, positive linear operators, which suggests a natural preconditioning strategy. Such nonlinear elliptic equations often arise from time discretization of parabolic equations that model various biological and physical phenomena, in particular, liquid crystals, thin film epitaxial growth and phase transformations. The analyses of the schemes involve the characterization of the strictly convex energies associated with the equations. We first give a general framework for PSD in generic Hilbert spaces. Based on certain reasonable assumptions of the linear pre-conditioner, a geometric convergence rate is shown for the nonlinear PSD iteration.  We then apply the general the theory to the fourth and sixth-order problems of interest, making use of Sobolev embedding and regularity results to confirm the appropriateness of our  pre-conditioners for the regularized p-Lapacian problems. Our results include a sharper theoretical convergence result for p-Laplacian systems compared to what may be found in existing works. We demonstrate rigorously how to apply the theory in the finite dimensional setting using finite difference discretization methods. Numerical simulations for some important physical application problems -- including thin film epitaxy with slope selection and the square phase field crystal model -- are carried out to verify the efficiency of the scheme. 
	\end{abstract}


	\section{Introduction}
Let $\Omega\subset \mathbb{R}^d$, $d=2,3$, be a rectangular domain. In this work we are interested in efficient solution techniques for fourth and sixth-order nonlinear elliptic equations that have p-Laplacian terms. The fourth-order problem reads as follows: given $f$ $\Omega$-periodic, find $u$ $\Omega$-periodic such that
	\begin{equation}
	\label{eqn:4ths}
u - s\nabla\cdot(|\nabla u|^{p-2}\nabla u)+s\varepsilon^2 \Delta^2 u  = f , 
	\end{equation}
where $0 < \varepsilon \le 1$ and $s$ is a positive parameter. The sixth-order problem is as follows: given $f, g$ $\Omega$-periodic, find $u,w$ $\Omega$-periodic such that
	\begin{subequations}
	\begin{align}
u -\Delta w = & \  g ,
	\label{eq:6th-mixed-a}
	\\
s \lambda  u- s\nabla\cdot\left(|\nabla u|^{p-2}\nabla u \right) +s\varepsilon^2 \Delta^2 u -w = & \ f,
	\label{eq:6th-mixed-b}
	\end{align}
	\end{subequations}
where $0 < \varepsilon \le 1$,  $s>0$, and  $\lambda\ge0$ are parameters. The highest order positive diffusion term, parameterized by $\varepsilon$, is often referred to as the surface diffusion, following the thin film applications described below.

We will refer to problems \eqref{eqn:4ths} and  \eqref{eq:6th-mixed-a} -- \eqref{eq:6th-mixed-b} as \emph{regularized p-Laplacian problems}. However, this is primarily for ease of reference. The highest order order surface diffusion term, though parameterized by the ``small" coefficient $\varepsilon>0$, must be  present for the related physical models to make sense and is not an artificial regularization. In other words, we will not consider and are not concerned with the singular limit $\varepsilon\searrow 0$.

These model equations arise most commonly from the time discretization for certain time-dependent physical  models.  For example, consider the thin epitaxial film model with slope selection 
  \begin{eqnarray*}
    \partial_t u =  \nabla \cdot \left( \left| \nabla u \right|^2 \nabla u \right) -  \Delta u - \varepsilon^2 \Delta^2 u, 
  \end{eqnarray*}
in \cite{li2004epitaxial, shen2012second, wang2010unconditionally,xu2006stability}. The 4-Laplacian, in combination with the other terms, gives energetic preference to facets with unit slope,  a continuum-level model of the Ehrlich-Schwoebel kinetic barrier. The highest order term models a small amount of surface diffusion, which smooths out the facets somewhat. In the square Swift-Hohenberg (SH) equation 
  \begin{eqnarray*} 
    \partial_t u = -(1+\Delta)^2 u- \beta  u+ \eta u^3-u^5 +\alpha\left( \left| \nabla u \right|^2 \nabla u \right),  \quad \alpha >0, \quad  \beta ,\eta\in \mathbb{R},
  \end{eqnarray*}
studied in~\cite{cross1993pattern, hoyle1995steady, medina2014formation, lloyd2008localized}, and the square phase field crystal (SPFC) equation
	\begin{equation*}
\partial_t u = \Delta \left(  \gamma_0  u+ \gamma_1\Delta u + \varepsilon^2 \Delta^2 u - \nabla \cdot \left( \left| \nabla u \right|^2 \nabla u \right) \right), \quad \gamma_0 \in \mathbb{R}, \quad \gamma_1 >0,
	\end{equation*}
studied in~\cite{elder04,golovin03,medina2014formation,lloyd2008localized}, the 4-Laplacian term gives preference to square-symmetry patterns. In general, such localized structures play important roles in biological, chemical, and physical processes~\cite{hoyle2006pattern}. 

For these time-dependent problems, convex splitting schemes have been proposed and analyzed in \cite{shen2012second, wang2010unconditionally} to obtain unconditional unique solvability and unconditional energy stability.  The convex splittings scheme for the thin film model is~\cite{wang2010unconditionally}
	\[
u^m - s\nabla\cdot(|\nabla u^m|^2\nabla u^m)+s\varepsilon^2 \Delta^2 u^m  =  u^{m-1}-s\Delta u^{m-1},
	\]
where $s>0$ is the time step size, and the superscripts indicate the time discretizations. The convex splitting scheme for the SPFC model -- which can be inferred from the general principles in~\cite{wang2010unconditionally,wise09a} --  is precisely 
	\begin{align*}
u^m -\Delta w^m = & \  u^{m-1} ,
	\\
s \gamma_0  u^m- s\nabla\cdot\left(|\nabla u^m|^2\nabla u^m \right) +s\varepsilon^2 \Delta^2 u^m -w^m = & \ -s\gamma_1\Delta u^{m-1},
	\end{align*}
assuming $\gamma_0,\gamma_1 \ge 0$. These schemes are nonlinear and require one to deal with the p-Laplacian term at the implicit time level. We remark that there are also second-order-in-time convex splitting schemes for such nonlinear parabolic equations, as described in~\cite{shen2012second}, which have similar nonlinear structures.  In any case, solving nonlinear elliptic equations with the p-Laplacian term is challenging, because of its highly nonlinear nature. In \cite{shen2012second, wang2010unconditionally}, the authors used a  nonlinear conjugate gradient algorithm to solve the nonlinear system at each implicit time step. Such naive gradient methods are guaranteed to converge due to the global convexity of the equations, but are not necessarily efficient.  

Several works develop and analyze numerical schemes for nonlinear elliptic equations involving the  p-Laplacian operator. The works~\cite{barrett94, bermejo2000multigrid, huang2007preconditioned, liu2001quasi, tai2002global, zhou2013steepest, zhou2005preconditioned} are based on  finite element approximations in space. Recently, the vanishing moment method for the p-Laplacian was proposed in \cite{feng2008vanishing}. In that method, the highest order term is purely artificial, whereas, for the models above, the surface diffusion term is small, but non-vanishing. A hybridizable discontinuous Galerkin method for the p-Laplacian was proposed in \cite{cockburn2016hybridizable}. Of these works, \cite{huang2007preconditioned, zhou2013steepest, zhou2005preconditioned} are primarily focused on efficient solvers for the elliptic equations with p-Laplacian terms, rather than, say, error estimates.

The main goal of this paper is to design a general framework of preconditioned steepest descent (PSD) methods for certain nonlinear elliptic equations with p-Laplacian terms. The main idea is to use a linearized version of the nonlinear operator as a pre-conditioner, or in other words, as a metric for choosing the search direction. We propose and analyze the preconditioned steepest descent methods for both the fourth- and sixth-order p-Laplacian problems mentioned above. Herein we present numerical simulations for the 6-Laplacian thin film epitaxy and the $H^{-1}$ gradient flow SPFC model by using the proposed method.  While we restrict our focus to the p-Laplacian problems herein, the search direction framework is general and can be applied to other nonlinear equations, such as the Cahn-Hilliard (CH) equation \cite{cahn58,cheng2016refine,pego89,shen10a}, functionalized Cahn-Hilliard (FCH) Equation \cite{Christlieb14high,doelman2014meander,feng2016energy}, for example. 

The convergence analyses of the nonlinear iteration algorithms we propose for the p-Laplacian equations are quite challenging, due to the highly nonlinear nature of the problems. However, we are able to recast the equations as equivalent minimization problems involving strictly convex functionals in generic Hilbert spaces. Once this is done, we are able to characterize the properties of general pre-conditioners that will result in geometric convergence rates. This general approach is applicable to both the 4th and 6th order equations at the space-continuous level, as well as the approximation of these problems in finite dimensions using finite differences as we show. Though we do not explore it here, we remark that the theory is extendible to the pseudo-spectral, spectral-Galerkin, and mixed finite element settings as well, using the appropriate discrete Gagliardo-Nirenburg inequalities. To our knowledge, the only related theoretical results available in the existing literature are to be found in \cite{huang2007preconditioned}, in which finite element PSD solvers were  designed and analyzed. Specifically, it was proved in \cite{huang2007preconditioned} that their method converges with the rate $O (k^{-\beta})$, where $k$ is the iteration index and $\beta = \frac{p}{p-2} >0$. In this article, we provide a theoretical analysis with a geometric convergence rate $O (\alpha^k)$, with $0<\alpha < 1$, for the finite difference PSD solver applied to the regularized p-Laplacian problems.

For such nonlinear analyses, the essential difficulty has always been associated with the subtle fact that the numerical solution has to be bounded uniformly in certain functional norms, so that a bound for the iteration error could be established. For the p-Laplacian problems, typically a uniform $W^{1,p}$ bound of the numerical solution is available at each iteration stage, and such a bound may be used to derive an $O (k^{-\beta})$ convergence rate for the PSD iteration.  However, for the regularized p-Laplacian problems, one observes that a linear operator with higher-order diffusion may be utilized so that a uniform $H^2$ bound of the numerical solution may be obtained. Specifically, the existence of the surface diffusion term $\varepsilon^2 \Delta^2 u$ enables us to derive a geometric convergence rate $O (\alpha^k)$ for the PSD iteration, which gives a sharper theoretical result than the existing one in \cite{huang2007preconditioned}.  Our strategy comes at a cost that we point out at the offset: a linear, positive, constant-coefficient operator of order 4 or 6 must be inverted to obtain the search direction. But, since we are interested in applications involving coarsening processes over periodic domains, the FFT can be utilized to make this process efficient.

The remainder of the paper is organized as follows. In section \ref{sec:psd}, we present a general preconditioned steepest descent (PSD) method for nonlinear equations in generic Hilbert spaces, and provide the convergence rate estimates for the PSD method. The  application of the general theory to the fourth-order regularized p-Laplacian problem is formulated in section \ref{sec:4th}.  The PSD scheme for the sixth-order regularized p-Laplacian problem is outlined in section \ref{sec:6th}. Subsequently, in section \ref{sec:fdm}, we introduce a two-dimensional finite difference discretization and provide the fully discrete convergence analysis. Applications to thin film epitaxy with slope selection and the SPFC model and the numerical results are presented in section \ref{sec:num}. The concluding remarks are offered in section \ref{sec:sum}. In the Appendix, we give the proof of a few discrete Sobolev inequalities.

	\section{Preconditioned Steepest Descent Methods}\label{sec:psd}	
	
	\subsection{The Classical Setting: Linear SPD Systems in Finite Dimensions}
	
Before we get to the general case, let us quickly review the convergence theory for preconditioned steepest decent methods for solving the linear system $\mathsf{A}{\bf u} = {\bf f}$, where $\mathsf{A}\in\mathbb{R}^{m\times m}_{\rm sym}$ is positive definite. This is closely related to the preconditioned conjugate gradient (PCG) method, though may be less familiar to the reader. Solving $\mathsf{A}{\bf u} = {\bf f}$ is, of course, equivalent to minimizing the quadratic energy $E[{\bf v}] : = \frac{1}{2}{\bf v}^T\mathsf{A}{\bf v} - {\bf v}^T{\bf f}$.  Suppose that $\mathsf{L}\in\mathbb{R}^{m\times m}_{\rm sym}$ is also positive definite. Here $\mathsf{A}$ is the \emph{stiffness matrix} and $\mathsf{L}$ is the \emph{pre-conditioner}. The idea is that $\mathsf{L}\approx \mathsf{A}$, but the former is ``easier to invert."  The preconditioned steepest decent algorithm for approximating the solution to  $\mathsf{A}{\bf u} ={\bf f}$ is given in Algorithm~\ref{alg-1} \cite{axelsson1996iterative, knyazev2007steepest}.
	
	\begin{algorithm}[H]
	\KwData
{${\bf u}_0,{\bf f}\in\mathbb{R}^m$}

$\mathbf{r}_0 := \ \mathbf{f} - \mathsf{A}{\bf u}_0$\;
$\mathbf{d}_0 :=  \mathsf{L}^{-1} \mathbf{r}_0$\;

	\For
{$k=0,\ \cdots \,,k_{\rm max}-1$}
{$\alpha_k :=  ({\bf d}_k^T {\bf r}_k)/ ({\bf d}_k^T \mathsf{A}{\bf d}_k)$\;
$\mathbf{u}_{k+1} :=  \mathbf{u}_k + \alpha_k \mathbf{d}_k$\;
$\mathbf{r}_{k+1} := {\bf f} -  \mathsf{A}{\bf u}_{k+1}$\;
	\eIf
{$\nrm{\mathbf{r}_{k+1}}<{\rm tol}$ {\bf or} $k = k_{\rm max}-1$}
{${\bf u}_\star := {\bf u}_{k+1}$\; 
exit {\bf for} loop\;}
{$\mathbf{d}_{k+1} :=  \mathsf{L}^{-1} \mathbf{r}_{k+1}$\;}}
	\KwResult
{${\bf u}_\star$}
	\caption{Preconditioned Steepest Descent}
	\label{alg-1}
	\end{algorithm}

Here ${\bf d}_k\in \mathbb{R}^m$ is called the \emph{search direction} and ${\bf r}_k\in \mathbb{R}^m$ is called the  \emph{residual}. We observe that
	\[
\alpha_k = \operatorname*{argmin}_{\alpha\in\mathbb{R}} E[{\bf u}_k + \alpha {\bf d}_k] = \operatorname*{argzero}_{\alpha\in\mathbb{R}} \delta E[{\bf u}_k + \alpha {\bf d}_k]({\bf d}_k) = \frac{{\bf d}_k^T {\bf r}_k}{ {\bf d}_k^T \mathsf{A}{\bf d}_k}.
	\]
We have the classical convergence result:
	\[
\nrm{{\bf u}-{\bf u}_k}_{\mathsf{A}} \le \left(\frac{\kappa -1 }{\kappa+1}\right)^k\nrm{{\bf u}-{\bf u}_0}_{\mathsf{A}},
	\]
where $\kappa := \frac{\lambda_m}{\lambda_1}$, and $\lambda_m$ is the largest eigenvalue of $\mathsf{L}^{-1}\mathsf{A}$, and $\lambda_1$ is the smallest \cite{axelsson1996iterative,knyazev2007steepest, saad2003iterative}.

	\subsection{Non-Quadratic Energy Functionals in Generic Hilbert Spaces}
	
Here we review the general theory for preconditioned steepest descent in a generic Hilbert space  \cite{atkinson2005theoretical,ciarlet89, ekeland1976convex,knyazev2007steepest}.  Suppose that $H$ is a (real) Hilbert space with the inner product $\iprd{\, \cdot\, }{\, \cdot\, }_H$ and induced norm $\nrm{\, \cdot\, }_H$. We consider an energy functional $E[\, \cdot \, ]:H\to\mathbb{R}$ with the following properties:
\begin{description}[align=left,labelwidth=1cm]
	\item[(E1)]
$E$ is twice Fr\'{e}chet differentiable for all points $\nu\in H$. For each fixed $\nu\in H$, $\delta E[\nu](\, \cdot \,):H\to\mathbb{R}$ is the continuous linear functional equal to the first Fr\'{e}chet derivative at $\nu$, and, for each fixed $\nu\in H$, $\delta^2E[\nu](\, \cdot \, ,\, \cdot \,):H\times H\to\mathbb{R}$ is the continuous bilinear operator equal to the second Fr\'{e}chet derivative at $\nu$.

	\item[(E2)]
For every $\nu\in H$, 
	\begin{equation}
0  \le \delta^2 E[\nu](\xi,\xi), \quad \forall \ \xi\in H , 
	\label{assmp-L6-a} 
	\end{equation}	
and
	\begin{equation}
0 < \delta^2 E[\nu](\xi,\xi), \quad \forall \ \xi\in H\setminus\left\{0\right\} .
	\label{assmp-L6-b} 
	\end{equation}	
This implies the strict convexity of $E$.

	\item[(E3)]
$E$ is coercive with respect to the norm on $H$, \emph{i.e.}, there exist  constants $C_1>0$, $C_2\ge 0$ such that
	\[
C_1 \nrm{\nu}_H^2 \le E[\nu]+ C_2 , \quad \forall \ \nu\in H.
	\]
\end{description}
If $E$ satisfies (E1) -- (E3), it follows \cite{ciarlet89} that there is is a unique element $u\in H$ with the property that
	\[
E[u] \le E[\nu], \quad \forall \ \nu \in H, \qquad \mbox{with} \qquad E[u] < E[\nu], \quad \mbox{for} \ \nu\ne u,
	\]
and this \emph{minimizer} further satisfies
	\[
\delta E[u](\xi) = 0, \quad \forall \ \xi \in H.
	\]

We wish to construct, via preconditioned steepest descent (PSD), a sequence that converges to the unique minimizer. By $H'$ we denote the continuous dual of $H$. When it is convenient, we use the symbol $\langle \, \cdot\,  ,  \, \cdot\,\rangle_H: H'\times H\to \mathbb{R}$ to denote the dual pairing between $H'$ and $H$.  Consider a linear operator ${\mathcal L} :H \to H'$. This operator ${\mathcal L}$, which we call the \emph{pre-conditioner} induces a bilinear form on $H$:
	\[
(\nu,\xi)_{\mathcal L} : = \langle{ \mathcal L}[\nu],\xi \rangle_{H} = {\mathcal L}[\nu](\xi), \quad \forall \ \nu ,\xi\in H.
	\]
We assume that ${\mathcal L}$ satisfies the following properties:
	\begin{description}[align=left,labelwidth=1cm]
	\item[(L1)]
$( \, \cdot\,  , \, \cdot\, )_{\mathcal L}: H\times H\to \mathbb{R}$ is symmetric, \emph{i.e.},
	\[
(\nu,\xi)_{\mathcal L} = (\xi,\nu)_{\mathcal L}, \quad \forall \ \nu,\xi\in H;
	\]
	\item[(L2)]
$( \, \cdot\,  , \, \cdot\, )_{\mathcal L}$ is continuous with respect to the standard topology of $H$, \emph{i.e.}, there is some $C_3>0$ such that 
	\[
\left|(\nu,\xi)_{\mathcal L}\right| \le  C_3\nrm{\nu}_H\nrm{\xi}_H, \quad \forall \ \nu,\xi \in H;
	\]
	\item[(L3)]
$( \, \cdot\,  , \, \cdot\, )_{\mathcal L}$ is coercive with respect to $H$, \emph{i.e.}, there is some $C_4>0$ such that
	\[
C_4\nrm{\nu}_H^2 \le  (\nu,\nu)_{\mathcal L} , \quad \forall \ \nu \in H.
	\]
	\end{description}
It follows that $( \, \cdot\,  , \, \cdot\, )_{\mathcal L}: H\times H\to \mathbb{R}$ is an inner product on $H$, equivalent to the primary inner product $\iprd{\, \cdot\, }{\, \cdot\, }_H$. The induced norm, $\nrm{\nu}_{\mathcal L} := \sqrt{(\nu ,\nu)_{\mathcal L}}$, is equivalent to the primary norm. By the Riesz Representation Theorem, if $f\in H'$, then there exists a unique $u_f\in H$ such that
	\[
(u_f,\xi)_{\mathcal L}  = f[\xi] = \langle f, \xi\rangle_{H}, \quad \forall \xi\in H,
	\]
with
	\[
\nrm{u_f}_{\mathcal L} = \nrm{f}_{{\mathcal L}^{-1}} := \sup_{0\ne \xi \in H}\frac{f[\xi]}{\nrm{\xi}_{\mathcal L}},
	\]
where the second norm is the ${\mathcal L}$-induced operator norm.

Suppose that $u^k\in H$ is given. We define the following \emph{search direction} problem: find $d^k \in H$ such that
	\begin{equation}
\iprd{d^k}{\xi}_{{\mathcal L}} = - \delta E\left[u^k\right](\xi), \quad \forall \xi\in H.
	\label{search-direction-pb}
	\end{equation}
We call $d^k$ the $k^{\rm th}$ \emph{search direction}. In operator form, we write ${\mathcal L}[d^k] = -\delta E[u^k]$ in $H'$. The functional $-\delta E\left[u^k\right]$ is called the \emph{residual} of $u^k$. By the Riesz Representation Theorem, we discover that
	\begin{equation}
- \delta E\left[u^k\right](d^k) = \norm{d^k}{{\mathcal L}}^2 = \norm{\delta E\left[u^k\right]}{{\mathcal L}^{-1}}^2.
	\label{RRT-1}
	\end{equation}
We then define the next iterate $u^{k+1}$ as
	\begin{equation}
u^{k+1} := u^k + \alpha_k d^k,
	\label{eq:psd4}
	\end{equation}
where $\alpha_k\in\mathbb{R}$ is the unique solution to
	\begin{equation}
\alpha_k := \operatorname*{argmin}_{\alpha\in\mathbb{R}} E[u^k + \alpha d^k] = \operatorname*{argzero}_{\alpha\in\mathbb{R}}\delta E[u^k + \alpha d^k](d^k) .
	\label{eqn-search}
	\end{equation}
Therefore, we have the \emph{fundamental orthogonality relation}
	\begin{equation}
\delta E[u^k + \alpha_k d^k](d^k) = \delta E[u^{k+1}](d^k) = 0.
	\label{eqn-orthogonal}
	\end{equation}
It follows that the sequence $\left\{u^k\right\}_{k=0}^\infty\subset H$ generated by the preconditioned steepest descent algorithm converges to the unique minimizer $u\in H$. We now wish to estimate the convergence rate.

	\subsection{Estimates of the Convergence Rate for the PSD Method}
	
We summarize some standard results.
	\begin{prop}
Suppose that $E$ satisfies (E1) -- (E3). It follows that, for any $\nu,\xi\in H$,
	\begin{equation}
\delta E[\nu](\xi -\nu) \le E[\xi] - E[\nu] \le \delta E[\xi](\xi -\nu),
	\label{est-convex}
	\end{equation}
and, consequently,
	\[
0 \le 	\left(\delta E[\xi]-\delta E[\nu]\right)(\xi -\nu).
	\]
	\end{prop}

	\begin{prop}
	\label{prop-energy-decrease}
Suppose that $E$ satisfies (E1) -- (E3). Let $\left\{u^k\right\}_{k=0}^\infty \subset H$ be computed via \eqref{eq:psd4}. Then, for every $k\geq0$ we have $E[u^{k+1}] \le E[u^k]$. Furthermore, $\alpha_k > 0$, as long as $u^k\ne u$. 
	\end{prop}	
	\begin{proof}
Using the orthogonality relation \eqref{eqn-orthogonal} and the convexity inequality \eqref{est-convex}, we find
	\[
E[u^{k+1}] - E[u^k] \le  \delta E[u^{k+1}](u^{k+1}-u^k) = \alpha_k\delta E[u^{k+1}](d^k) = 0.
	\]	
Now, suppose $d^k\ne 0$. Then, by Taylor's theorem, \eqref{RRT-1}, and \eqref{assmp-L6-b},
	\[
E[u^{k+1}] = E[u^k] - \alpha_k \nrm{d^k}_{\mathcal L}^{2} +\frac{\alpha_k^2}{2} \delta^2 E[\theta^k](d^k,d^k) > E[u^k] - \alpha_k \nrm{d^k}_{\mathcal L}^{2}.
	\]
Equivalently, we get 
	\[
\alpha_k \nrm{d^k}_{\mathcal L}^{2} > 	E[u^k] -E[u^{k+1}] \ge 0,
	\]
which implies that $\alpha_k > 0$.
	\end{proof}
	
	\begin{prop}
Suppose that $E$ satisfies (E1) -- (E3) and $u\in H$ is the unique minimizer of $E$. Then, for any $\xi \in H$,
	\[
0 \le E[\xi] - E[u] \le \left(\delta E[\xi]-\delta E[u]\right)(\xi-u) = \delta E[\xi](\xi-u),
	\]
and, consequently,
	\begin{equation}
0  \le E[u^k] - E[u] \le  \left(\delta E[u^k]-\delta E[u]\right)(u^k-u) = \delta E[u^k](u^k-u).
	\label{est-convex-2}
	\end{equation}
	\end{prop}
	\begin{proof}
This follows immediately from \eqref{est-convex}, because $\delta E[u](\xi) = 0$, for all $\xi\in H$.
	\end{proof}
	
Now, we make the following further assumptions about the pre-conditioner ${\mathcal L}$ with respect to the derivatives of the energy $E$:
	\begin{description}[align=left,labelwidth=1cm]
	\item[(L4)]
There is a constant $C_5>0$ such that
	\begin{equation}
C_5 \norm{\xi-\nu}{\mathcal{L}}^2\le \left(\delta E[\xi] -\delta E[\nu]\right)(\xi -\nu)  , 
	\label{assmp-L4} 
	\end{equation}
for all $\nu,\xi\in H$.
	\item[(L5)]
Suppose $B:= \left\{\nu \in  H  \ \middle| \ E[\nu]\le E_0\right\}$ is non-empty. (This is the the case if, for example, one chooses $E_0 = E[0]$.) There is a constant $C_6=C_6(E_0)>0$ such that, for all $\nu\in B$, and any $\xi\in H$,
	\begin{equation}
\left| \delta^2 E[\nu](\xi,\xi) \right| \leq  C_6 \norm{\xi}{\mathcal{L}}^2 .
	\label{assmp-L5} 
	\end{equation}
	\end{description}
	
	\begin{rmk}
We note that, practically speaking, (L5) is harder of the last two conditions to enforce. In some sense, if the norm induced by ${\mathcal L}$ is not ``strong" enough, then there does not exist $C_6>0$  so that (L5) is satisfied.
	\end{rmk}

	\begin{thm}
	\label{thm-generic-convergence}
Suppose that assumptions (E1) -- (E3) and (L1) -- (L5) are valid. Let $\left\{u^k\right\}_{k=0}^\infty \subset H$  be the sequence generated by \eqref{eq:psd4}. Then 
	\begin{equation}
	\label{eqn:contraction}
0\le E [u^k] - E[u] \leq (C_7)^k (E [u^0] - E[u]) ,
	\end{equation}
where 
	\begin{equation}
0< C_7 := 1 - \frac{C_5}{2C_6} < 1.
	\end{equation}
	\end{thm}
	\begin{proof}
Consider the function $g(\alpha) : = E[u^k+\alpha d^k] - E[u^k]$, $\alpha\in\mathbb{R}$. Then $g(0) = 0$, and $g$ has a global minimum at $\alpha_k >0$. By coercivity and continuity of $E$, there is a $\beta_k$, $\alpha_k < \beta_k <\infty$, such that $g(\beta_k) = 0$, and, for all $\alpha\in[0,\beta_k]$, 
	\[
E[u^k+\alpha d^k] \le E[u^k] \le E[u^0] =: E_0.
	\]
By Taylor's theorem, there is a $\gamma = \gamma(u^k,d^k,\alpha)\in (0,1)$, such that 
	\[
E[u^k+\alpha d^k]-E[u^k]  = \alpha\delta E[u^k](d^k) +\frac{\alpha^2}{2} \delta^2 E[\theta^k](d^k,d^k),
	\]
where $\theta^k := u^k + (1-\gamma)\alpha d^k$. By convexity of $E$,
	\[
E[\theta^k] \le \gamma E[u^k] + (1-\gamma) E[u^k+\alpha d^k] \le E[u^k] \le E[u^0] = E_0 . 
	\]
Using estimate \eqref{assmp-L5} -- with the set $B$ defined with respect to $E_0 = E[u^0]$ -- and norm equality \eqref{RRT-1}, we get, for all $\alpha\in [0,\beta_k]$,
	\begin{align}
g(\alpha) = E[u^k+\alpha d^k]-E[u^k]  \leq & \ \alpha\delta E[u^k](d^k) +\frac{\alpha^2}{2} C_6 \norm{d^k}{\mathcal{L}}^2 
	\nonumber
	\\
= & \ \big(-\alpha+\frac{\alpha^2}{2}C_6 \big)\norm{ \delta E[u^k]}{\mathcal{L}^{-1}}^2 =: f(\alpha).
	\label{g-less-than-f}
	\end{align}
Now, the function $f(\alpha)$ is quadratic, $f(0) = 0$, $f(\beta_k) \ge g(\beta_k) =0$, and $f'(0) < 0$. See Figure \ref{fig:fplot}. Thus $f$ has a minimum in $(0,\beta_k)$. 
	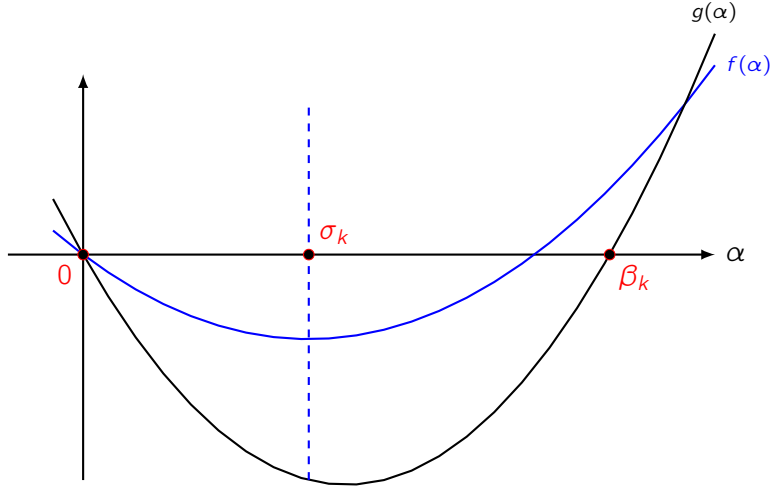
\begin{figure}[!htp]
	\begin{center}
	\begin{tikzpicture}[>=latex,scale=2.0,domain=-0.2:4.2]
    \draw[thick,->] (-0.5,0) -- (4.2,0) node[right] {$\alpha$};
    \draw[thick,->] (0,-1.5) -- (0,1.2) node[above] {};
    \draw[thick,blue]   plot (\x,{0.25*(\x-3.0)*\x})   node[right] {$_{f(\alpha)}$};
    \draw[thick]   plot (\x,{0.5*(\x-3.5)*\x}) node[above] {$_{g(\alpha)}$};
    \draw[red,fill=black] (3.5,0) node[below right] {$\beta_k$} circle (1pt);
    \draw[dashed,blue,thick] (1.5,-1.5)--(1.5,1);
     \draw[red,fill=black] (1.5,0) node[above right] {$\sigma_k$} circle (1pt);
     \draw[red,fill=black] (0.0,0) node[below left] {0} circle (1pt);
	\end{tikzpicture}
	\end{center}
\caption{The functions $g(\alpha) = E[u^k+\alpha d^k] - E[u^k]$ and $f(\alpha)=\big(-\alpha+\frac{\alpha^2}{2}C_6 \big)\norm{ \delta E[u^k]}{\mathcal{L}^{-1}}^2$ from \eqref{g-less-than-f}. The function $g$, which is strictly convex, is dominated by the function $f$, which is quadratic, on the interval $[0,\beta_k]$.}	
	\label{fig:fplot}
	\end{figure}
In fact, the minimum  is achieved at $0< \sigma_k:=\frac{1}{C_6} < \beta_k$ . Then we have 
	\[
E[u^k+\alpha_kd^k]-E[u^k] \le g(\sigma_k) = E[u^k+\sigma_kd^k]-E[u^k] \leq -\frac{1}{2C_6} \norm{ \delta E[u^k]}{\mathcal{L}^{-1}}^2 = f(\sigma_k),
	\]
or, equivalently,
	\[
E[u^k] - E[u^{k+1}] \geq \frac{1}{2C_6} \norm{ \delta E[u^k]}{\mathcal{L}^{-1}}^2.
	\]
Now, using estimates \eqref{est-convex-2} and  \eqref{assmp-L4} we obtain 
	\[
0\le E[u^k]-E[u] \le \frac{1}{C_5}\nrm{\delta E[u^k]}_{{\mathcal L}^{-1}}^2.
	\]
Combining the last two estimates, we get the result
	\[
0\le E[u^k]-E[u] \le \frac{2C_6}{C_5}\left(E[u^k] - E[u^{k+1}] \right),
	\]
or, equivalently,
	\[
0\le E[u^{k+1}]-E[u] \le \left(\frac{2C_6}{C_5}-1\right) \left(E[u^k] - E[u^{k+1}] \right).
	\]
Since $E[u^{k+1}] > E[u]$, as long as $u^{k+1}\ne u$, and $E[u^k] \ge E[u^{k+1}]$, this last inequality implies that 
	\[
0 <  \frac{C_5}{2C_6} < 1.
	\]
A little more manipulation reveals the equivalent inequality
	\[
0\le E[u^{k+1}]-E[u] \le \left(1-\frac{C_5}{2C_6}\right) \left(E[u^k] - E[u] \right),
	\]
and the result follows.
	\end{proof}
	
If the following property holds, we get a simple corollary of the last theorem.
	
	\begin{description}[align=left,labelwidth=1cm]
	\item[(L6)]
There is a constant $C_8>0$, such that, for every $\nu,\xi \in H$, 
	\begin{equation}
C_8 \nrm{\xi}_{\mathcal L}^2 \le |\delta^2 E[\nu](\xi,\xi)|.
	\label{assmp-L6} 
	\end{equation}	
This implies the strong convexity of $E$ and is, therefore, stronger that (E2).
	\end{description}

	\begin{cor}
	\label{cor:phicov4c}
Suppose that assumptions (E1) -- (E3) and (L1) -- (L6) are valid. Let $\left\{u^k\right\}_{k=0}^\infty \subset H$  be the sequence generated by \eqref{eq:psd4}, and define $e^k := u-u^k$. Then 
	\begin{equation} 
\nrm{e^k}_{\mathcal L}^2 \le  (C_7)^k \frac{E [u^0] - E[u]}{C_8} .
	\label{error-phi-4c}   
	\end{equation}

	\end{cor}

	\begin{proof}
By Taylor's theorem and estimate \eqref{assmp-L6}, we have 
	\begin{eqnarray} 
E[u^{k}]-E[u] &=& \delta E[u](e^k) +\frac{1}{2} \delta^2 E[\theta^k](e^k,e^k)
	\nonumber 
	\\
& =&  \frac{1}{2} \delta^2 E[\theta^k](e^k,e^k) \ge C_8 \nrm{e^k}_{\mathcal L}^{2} ,
	\label{error-phi-1} 
	\end{eqnarray}  
where $\theta^k$ is in the line segment from $u^{k}$ to $u$. The result follows from \eqref{eqn:contraction}.
	\end{proof} 	
	
	\section{Nonlinear Elliptic Equations on Periodic Domains}\label{sec:4th}
	
	\subsection{Notation for Periodic Sobolev Spaces}
	\label{sub:notation}

For the remainder of paper $\Omega \subset \mathbb{R}^d$ with $d=2,3$ is a rectangular domain. In what follows, if $d=2$ we assume $p \in[2,\infty)$; whereas if $d=3$ we suppose $p \in [2,6]$. Most of the physically relevant cases correspond to $p$ being an even integer, however, all of our arguments hold for any value of $p$ in the indicated ranges. The Sobolev spaces of periodic functions are defined as follows: for $q \in [1,\infty]$,  we set 
	\[
W^{k,q}_{\rm per}(\Omega) := \left\{u\in W^{k,q}_{\rm loc}(\mathbb{R}^d) \ \middle| \ u \ \mbox{is} \ \Omega-\mbox{periodic}\right\},
	\]
where $k \in \mathbb{N}$ is the differentiability index. Observe that $W^{0,q}_{\rm per}(\Omega) =: L^q_{\rm per}(\Omega) = L^q(\Omega)$.  We denote the norm of $W^{k,q}_{\rm per}(\Omega)$ by $\|\cdot \|_{W^{k,q}}$, or just $\|\cdot \|_{L^q}$ when $k=0$.   In the case $q=2$ and $k=0$, we denote by $( \cdot , \cdot )$ and $\| \cdot \|$ the inner product and norm, respectively. We set $H^k_{\rm per}(\Omega) = W^{k,2}_{\rm per}(\Omega)$ and immediately remark that, given the range of $p$, we have $H^2_{\rm per}(\Omega)\hookrightarrow W^{1,p}_{\rm per}(\Omega)$. For $k \in \mathbb{N}$, the continuous dual of $H^k_{\rm per}(\Omega)$ is denoted by $H^{-k}_{\rm per}(\Omega)$ and their pairing by $\langle\cdot,\cdot\rangle_k$. We set $\langle \cdot,\cdot \rangle := \langle \cdot,\cdot \rangle_1$.

If $L_0^2(\Omega)$ denotes the set of functions in $L^2(\Omega)$ with zero mean, we define
	\begin{equation*}
\mathring{H}_{\rm per}^1(\Omega) := H_{\rm per}^1(\Omega)\cap L_0^2(\Omega), \quad \mathring{H}_{\rm per}^{-1}(\Omega) :=\left\{v\in H_{\rm per}^{-1}(\Omega) \ \middle| \ \langle v , 1 \rangle = 0  \right\}.
	\end{equation*}
We define a linear operator $\msfT : \mathring{H}_{\rm per}^{-1}(\Omega) \rightarrow \mathring{H}_{\rm per}^1(\Omega)$ via the following variational problem: given $\zeta\in \mathring{H}_{\rm per}^{-1}(\Omega)$, $\msfT[\zeta]\in \mathring{H}_{\rm per}^1(\Omega)$ solves
	\[
\iprd{\nabla \msfT[\zeta]}{\nabla\chi} = \langle \zeta, \chi\rangle , \qquad \forall \ \chi\in \mathring{H}_{\rm per}^1(\Omega).
	\]
From the Riesz representation theorem it immediately follows that $\msfT$ is well-defined. We define the inner product
	\[
\left(\zeta,\xi\right)_{\mathring{H}_{\rm per}^{-1}} :=\iprd{\nabla \msfT[\zeta]}{\nabla\msfT[\xi]} =\langle\zeta,\msfT[\xi]\rangle  = \langle \xi, \msfT[\zeta]\rangle ,  \quad \forall  \ \zeta, \xi \in \mathring{H}_{\rm per}^{-1}(\Omega) .
	\]
The induced norm is denoted $\| \cdot \|_{\mathring{H}^{-1}_{\rm per}}$.  The following facts can be easily established~\cite{diegel2015analysis}:

	\begin{lem}
	\label{lem-negative-norm}
On $\mathring{H}^{-1}_{\rm per}(\Omega)$ the norm $\|\cdot\|_{\mathring{H}_{\rm per}^{-1}}$ equals the operator norm: for all $\zeta\in \mathring{H}^{-1}_{\rm per}(\Omega)$,
	\[
  \norm{\zeta}{\mathring{H}^{-1}_{\rm per}} = \sup_{0\ne \chi\in\mathring{H}^1_{\rm per}(\Omega)} \frac{\langle \zeta , \chi\rangle}{\norm{\nabla\chi}{}} . 
	\]
Consequently, we have $\left|\langle \zeta , \chi\rangle\right| \le \norm{\zeta}{\mathring{H}^{-1}_{\rm per}} \nrm{\nabla\chi}$, for all $\chi\in H_{\rm per}^1(\Omega)$ and $\zeta\in\mathring{H}_{\rm per}^{-1}(\Omega)$.  Furthermore, for all $\zeta\in L_0^2(\Omega)$, we have the Poincar\'e type inequality: $\norm{\zeta}{\mathring{H}^{-1}_{\rm per}} \le C  \norm{\zeta}{}$, for some $C>0$. 
\end{lem}

	\subsection{A Fourth-Order Regularized p-Laplacian Problem}

We consider the following weak formulation of \eqref{eqn:4ths}: given $f\in L^2_{\rm per}(\Omega)$, find $u \in H^2_{\rm per}(\Omega)$ such that

	\begin{equation}
\iprd{u}{\xi} +s\iprd{|\nabla u|^{p-2}\nabla u}{\nabla \xi} +s\varepsilon^2 \iprd{\Delta u}{\Delta \xi}  = \iprd{f}{\xi} , \quad  \forall \ \xi \in H^2_{\rm per}(\Omega),
	\label{weak-a} 
	\end{equation}
where $0 < \varepsilon \le 1$ and $s>0$ are parameters. Equation \eqref{weak-a} is mass conservative in the following sense: $\iprd{u-f}{1} = 0$.  One can show that the solution of the weak formulation is a minimizer of the following energy: for any $\nu \in H^2_{\rm per}(\Omega)$,
	\begin{equation}
	\label{eqn:eng4c}
E[\nu] :=  \frac{1}{2} \| \nu-f \|^2 +\frac{s}{p} \norm{\nabla \nu}{L^p}^p +\frac{s\varepsilon^2}{2} \| \Delta \nu \|^2 .
	\end{equation}
It is not difficult to show that $E$ satisfies (E1) -- (E3). The first derivative of $E$ at a point $\nu$ may be calculated as follows: for any $\xi\in H^2_{\rm per}(\Omega)$, 
	\[
\left. d_\tau E[\nu+\tau \xi]\right|_{\tau=0} = \delta E[\nu](\xi) = \iprd{\nu-f}{\xi} + s\iprd{|\nabla \nu|^{p-2}\nabla \nu}{\nabla \xi} + s\varepsilon^2\iprd{\Delta \nu}{\Delta \xi} .
	\]
Thus, our original problem is equivalent to the following: find $u\in H^2_{\rm per}(\Omega)$, such that, for all $\xi\in H^2_{\rm per}(\Omega)$, $\delta E[u](\xi) = 0$, which is equivalent to \eqref{weak-a}. This problem has a unique solution, which is, in turn, the unique minimizer of the energy \eqref{eqn:eng4c}:
	\[
u := \operatorname*{argmin}_{\nu\in H^2_{\rm per}(\Omega)} E[\nu] .
	\]

The following estimate holds: for all $\nu,\xi \in H^2_{\rm per}(\Omega)$,
	\[
\left|\delta E[\nu](\xi)\right| \le \nrm{\nu-f} \cdot \nrm{\xi} + s\norm{\nabla \nu}{L^p}^{p-1} \norm{\nabla \xi}{L^p} +s\varepsilon^2\nrm{\Delta \nu} \cdot \nrm{\Delta \xi} .
	\]
The second variation is a continuous  bilinear operator. Given a fixed  $\nu\in H^2_{\rm per}(\Omega)$, the action of the second variation on the  arbitrary pair $(\xi,\eta)\in H^2_{\rm per}(\Omega)\times H^2_{\rm per}(\Omega)$ is given by
	\begin{align*}
\delta^2 E[\nu](\xi,\eta) = & \ \iprd{\xi}{\eta} + s\iprd{|\nabla \nu|^{p-2} \nabla \xi}{\nabla \eta} 
	\nonumber
	\\
& +(p-2)s\iprd{|\nabla \nu|^{p-4}\nabla \nu \cdot \nabla \xi}{\nabla \nu \cdot \nabla \eta} +s\varepsilon^2\iprd{\Delta \xi}{\Delta \eta},
	\end{align*}
and we have the bound
	\begin{align}
\left|\delta^2 E[\nu](\xi,\eta)\right| \le & \ \nrm{\xi} \cdot \nrm{\eta} + s\norm{\nabla \nu}{L^p}^{p-2}\norm{\nabla \xi}{L^p}\norm{\nabla \eta}{L^p}
	\nonumber
	\\
& \ +(p-2)s\norm{\nabla \nu}{L^p}^{p-2}\norm{\nabla \xi}{L^p}\norm{\nabla \eta}{L^p} +s\varepsilon^2\nrm{\Delta \xi } \cdot \nrm{\Delta \eta} .
	\label{eq:2ndvarboundE4}
	\end{align}
	
For this problem we define the pre-conditioner ${\mathcal L} :H^2_{\rm per}(\Omega) \to H^{-2}_{\rm per}(\Omega)$ via
	\[
\langle{\mathcal L}[\nu],\xi \rangle := \iprd{\nu}{\xi} +s\iprd{\nabla \nu}{\nabla \xi} + s\varepsilon^2 \iprd{\Delta \nu }{\Delta \xi } , \quad \forall \ \xi \in H^2_{\rm per}(\Omega).
	\]
Clearly, this is a positive, symmetric operator, and it satisfies assumptions (L1) -- (L3), and one can see the similarities with the nonlinear operator in \eqref{weak-a}. We now proceed to find the positive constants for which $C_5, C_6, C_8$ assumptions (L4) -- (L6) are satisfied.
	
	\begin{rmk}
We could also consider the possibility of changing the metric in the descent direction calculation by, for example, defining the linear operator ${\mathcal L}_k :H^2_{\rm per}(\Omega) \to H^{-2}_{\rm per}(\Omega)$ via
	\[
\langle {\mathcal L}_k[\nu], \xi \rangle := \iprd{\nu}{\xi} +s\iprd{\left|\nabla u^k\right|^{p-2}\nabla \nu}{\nabla \xi} +s\varepsilon^2 \iprd{\Delta \nu}{\Delta \xi} , \quad \forall \ \xi \in H^2_{\rm per}(\Omega).
	\]	
This is similar to the idea in~\cite{knyazev2007steepest}. The \emph{search direction} is then found as follows: find $d^k \in H^2_{\rm per}(\Omega)$ such that
	\[
\langle {\mathcal L}_k[d^k], \xi\rangle  = - \delta E\left[u^k\right](\xi), \quad \forall \ \xi \in H^2_{\rm per}(\Omega).
	\]
Our theory does not cover this case, and we will not consider it further here. We plan to examine this in a future work.
	\end{rmk}

	\begin{lem}
	\label{lem:lemma3}
Suppose that $p\in [2,\infty)$ when $d=2$, and  $p \in[2,6]$, if $d=3$. For any $\xi \in H^2_{\rm per}(\Omega)$, we have 
	\begin{equation}
\norm{\nabla \xi}{L^p} \le C_9 \left\{
	\begin{array}{llll}
 \nrm{\xi}^{\frac1p} \cdot \nrm{\Delta \xi}^{\frac{p-1}{p}}, & \mbox{if} &  d=2, & p\in [2,\infty),
	\\
\nrm{\xi}^{\frac{3}{2p} -\frac{1}{4}} \cdot \nrm{\Delta \xi}^{\frac{5}{4}-\frac{3}{2p}}, & \mbox{if} &  d=3, & p \in [2,6],
	\end{array}
\right.
	\label{W1p-estimate}
	\end{equation}
for some $C_9 = C_9(d,p)>0$. 
	\end{lem}
	\begin{proof}
This follows from the Gagliardo-Nirenberg interpolation inequality and elliptic regularity.
	\end{proof}

	\begin{lem}
	\label{lem:lemma4}
For any $\nu,\xi \in H^2_{\rm per}(\Omega)$,  
	\begin{equation}
C_5 \norm{\xi-\nu}{\mathcal{L}}^2 \le \left(\delta E[\xi] -\delta E[\nu]\right)(\xi -\nu)  , 
	\label{lem 4-0}
	\end{equation}
where $C_5 = \min \left( \frac12, \varepsilon s^{-\frac12} \right)$. Let $E_0$ be given, such that $B:= \left\{\nu \in  H_{\rm per}^2(\Omega) \ \middle| \ E[\nu]\le E_0\right\}$ is non-empty. For any $\nu\in B$ and any $\xi\in H^2_{\rm per}(\Omega)$,
	\begin{equation}
\left| \delta^2 E[\nu](\xi,\xi) \right| \leq  C_6 \norm{\xi}{\mathcal{L}}^2 ,
	\label{lem 5-0-2} 
	\end{equation}
where
	\begin{equation} 
C_6 = \left\{
	\begin{array}{llll}
1 +\frac{1}{p}\left(p-1\right)^{\frac{2p-1}{p}} \varepsilon^{\frac{-2(p-1)}{p}}s^{\frac{1}{p}} C_9^2  C_{10}^{p-2} & \mbox{for} & p\in[2,\infty), & d= 2,
	\\
1 + (p-1)\left(\frac{4p}{6-p}\right)^{\frac{p-6}{4p}}\left(\frac{4p}{5p-6}\right)^{\frac{6-5p}{4p}}s^{\frac{6-p}{4p}}\varepsilon^{\frac{6-5p}{2p}}C_9^2  C_{10}^{p-2} & \mbox{for} & p\in[2,6), & d = 3,
	\\
1 +\left(p-1\right) \varepsilon^{-2}C_9^2  C_{10}^{p-2} & \mbox{for} & p=6, & d = 3,
	\end{array}
\right.
	\label{lem 5-0-3}  
	\end{equation}
and $C_{10} = (p E_0)^{\frac{1}{p}}$. We can take $C_8 = C_5$ to satisfy  estimate \eqref{assmp-L6} of assumption (L6).
	\end{lem}

	\begin{proof}
Clearly
	\begin{align*} 
\left(\delta E[\xi]-\delta E[\nu]\right)(\xi -\nu) = & \ \norm{ \xi - \nu}{}^2 + s\varepsilon^2  \norm{\Delta (\xi-\nu)}{}^2
	\nonumber
	\\
& \ + s \left( | \nabla \xi |^{p-2} \nabla \xi - | \nabla \nu |^{p-2} \nabla \nu , \nabla (\xi -\nu) \right) . 
	\end{align*} 
In addition, the following estimate is available: 
	\begin{equation} 
\left( | \nabla \xi |^{p-2} \nabla \xi - | \nabla \nu |^{p-2} \nabla \nu , \nabla(\xi -\nu) \right) \ge \frac1{2^{p-2}} \norm{ \nabla (\xi -\nu) }{L^p}^p \ge 0 ,  \quad \mbox{for $p\ge 2$} . 
	\label{lem 4-2} 
	\end{equation}
The simple interpolation inequality 
	\[
\| \nabla \xi \|^2 \leq \| \xi\| \cdot \| \Delta \xi \|, \quad \forall \xi \in H^2_{\rm per}(\Omega),
	\]
in conjunction with Young's inequality yields
	\begin{equation*} 
\frac12 \nrm{ \xi - \nu }^2 + \frac{s\varepsilon^2}2 \nrm{\Delta(\xi -\nu)}^2   \ge  s^{\frac12}\varepsilon  \nrm{ \xi-\nu } \cdot \nrm{\Delta (\xi -\nu)} \ge  s^{\frac12}\varepsilon \nrm{ \nabla (\xi -\nu) }^2 .
	\end{equation*}
As a consequence, we get 
	\begin{align*} 
\left(\delta E[\xi]-\delta E[\nu]\right)(\xi -\nu) \ge & \ \nrm{\xi-\nu}^2 + s\varepsilon^2 \nrm{\Delta (\xi -\nu)}^2
	\nonumber 
	\\
\ge & \ \frac12 \nrm{\xi-\nu}^2 + \frac12s \varepsilon^2  \nrm{\Delta (\xi -\nu)}^2   + s^{\frac12}\varepsilon  \nrm{ \nabla (\xi -\nu) }^2 ,
	\end{align*}
and we conclude that estimate (\ref{lem 4-0}) is valid by choosing $C_5 = \min ( \frac12, \varepsilon s^{-\frac12} )$.   
	
Next we derive \eqref{lem 5-0-2}. Suppose $\nu\in B$. From \eqref{eq:2ndvarboundE4} we have
	\begin{equation}
\left|\delta^2 E[\nu](\xi,\xi)\right| \le  \nrm{\xi}^2 +(p-1)s\norm{\nabla \nu}{L^p}^{p-2}\norm{\nabla \xi}{L^p}^2 +s\varepsilon^2\nrm{\Delta \xi}^2.
	\label{lem 5-1}
	\end{equation}	 
Now, since $\nu\in B$, 
	\[
\norm{\nabla \nu}{L^p} \le (p E_0)^{\frac{1}{p}} =: C_{10}.
	\]
Suppose that $d=2$. An application of the Sobolev inequality (\ref{W1p-estimate}) in Lemma \ref{lem:lemma3} indicates that 
	\begin{align*}
p^{\frac{1}{p}} \left(\frac{p}{p-1}\right)^{\frac{p-1}{p}} \varepsilon^{\frac{2(p-1)}{p}} s^{\frac{(p-1)}{p}} C_9^{-2}  \norm{\nabla \xi}{L^p}^2  \le & \   p^{\frac{1}{p}}\norm{\xi}{}^{\frac{2}{p}} \cdot \left(\frac{p}{p-1}\right)^{\frac{p-1}{p}} \varepsilon^{\frac{2(p-1)}{p}} s^{\frac{p-1}{p}} \nrm{\Delta \xi}^{\frac{2(p-1)}{p}} 
	\nonumber
	\\
\le & \ \nrm{\xi}^2 + s\varepsilon^2  \nrm{\Delta \xi}^2  ,
	\end{align*}  
where Young's inequality is applied in the second step. It follows that, 
	\begin{equation}
(p-1)s\norm{\nabla \nu}{L^p}^{p-2} \norm{\nabla \xi}{L^p}^2 \le   \frac{1}{p}\left(p-1\right)^{\frac{2p-1}{p}} \varepsilon^{\frac{-2(p-1)}{p}}s^{\frac{1}{p}} C_9^2  C_{10}^{p-2}  \left(\nrm{\xi}^2 + s\varepsilon^2  \nrm{\Delta \xi}^2\right).
	\label{lem 5-3-b}
	\end{equation} 
Substituting \eqref{lem 5-3-b} in \eqref{lem 5-1} yields 
	\begin{equation*}
\left| \delta^2 E[\nu](\xi,\xi) \right|  \le \left( 1 + \frac{1}{p}\left(p-1\right)^{\frac{2p-1}{p}} \varepsilon^{\frac{-2(p-1)}{p}}s^{\frac{1}{p}} C_9^2  C_{10}^{p-2} \right) ( \nrm{\xi}^2 + s \varepsilon^2  \nrm{\Delta \xi}^2 ) . 	
	\end{equation*}
We conclude that estimate (\ref{lem 5-0-2}) is valid by choosing 
	\[
C_6 = 1 +\frac{1}{p}\left(p-1\right)^{\frac{2p-1}{p}} \varepsilon^{\frac{-2(p-1)}{p}}s^{\frac{1}{p}} C_9^2  C_{10}^{p-2} .
	\]
Note that both $C_9$ and $C_{10}$ are $\varepsilon$ and $s$ independent. 
Following the similar arguments, for $p\in[2,6),\ d = 3$, we get
	\[
C_6 = 1 + (p-1)\left(\frac{4p}{6-p}\right)^{\frac{p-6}{4p}}\left(\frac{4p}{5p-6}\right)^{\frac{6-5p}{4p}}s^{\frac{6-p}{4p}}\varepsilon^{\frac{6-5p}{2p}}C_9^2  C_{10}^{p-2} .
	\]
For the case $p=6,\ d = 3$, the Sobolev inequality \eqref{W1p-estimate} degenerates to $\norm{\nabla \xi}{L^p} \le C_9  \nrm{\Delta \xi}$, for any $\xi \in H^2_{\rm per}(\Omega)$. Hence, we have 
	\begin{equation*}
\nrm{\xi}^2 +s\varepsilon^2\nrm{\Delta \xi}^2 \geq  s\varepsilon^2\nrm{\Delta \xi}^2\geq  s\varepsilon^2C_9^{-2}\norm{\nabla \xi}{L^p}^2,
	\end{equation*} 
and 
	\begin{equation*}
\left| \delta^2 E[\nu](\xi,\xi) \right|  \le \left( 1 + \left(p-1\right) \varepsilon^{-2}C_9^2  C_{10}^{p-2} \right) ( \nrm{\xi}^2 + \varepsilon^2  \nrm{\Delta \xi}^2 ) . 	 
	\end{equation*}
Therefore,  estimate (\ref{lem 5-0-2}) is valid by choosing 
	\[
C_6 = 1 +\left(p-1\right) \varepsilon^{-2}C_9^2  C_{10}^{p-2}.
	\]	

That we can take $C_8 = C_5$ is the result of a simple calculation that we omit for the sake of brevity. The proof is complete. 
	\end{proof}
	

	\section{A Sixth-Order Regularized p-Laplacian Problem}\label{sec:6th}

We now study problem \eqref{eq:6th-mixed-a} -- \eqref{eq:6th-mixed-b}. A weak formulation is given as follows: for $f,g\in L^2_{\rm per}(\Omega)$, find $u\in H^2_{\rm per}(\Omega)$ and $w\in H_{\rm per}^1(\Omega)$ such that 
	\begin{subequations}
	\begin{eqnarray}
(u,\chi) + \iprd{\nabla w}{\nabla \chi}{} = & \  \iprd{g}{\chi} , \quad \forall \ \chi \in H^1_{\rm per}(\Omega),
	\label{eq:6th-mixed-weak a}
	\\
s\lambda \iprd{u}{\xi}{} + s\iprd{|\nabla u|^{p-2}\nabla u}{\nabla \xi}{} +s \varepsilon^2 \iprd{\Delta u}{\Delta \xi}{} -\iprd{w}{\xi}{} = & \ \iprd{f}{\xi}{}, \quad \forall \ \xi \in H^2_{\rm per}(\Omega),
	\label{eq:6th-mixed-weak b}
	\end{eqnarray}
	\end{subequations}
where $\lambda \ge 0$, and $\varepsilon\in (0,1]$. This problem is mass-conservative, in the sense that $(u-g,1) = 0$, and $(w-s \lambda g+ f,1)=0$, and it can be recast as a minimization problem with an energy that involves the $\mathring{H}_{\rm per}^{-1}$ norm. In particular, for any $\nu \in \mathring{H}_{\rm per}^2(\Omega)$ we define
	\begin{align}
E[\nu] = & \ \frac{1}{2}\iprd{\nu-g+\bar{g}}{\msfT[\nu-g+\bar{g}]} + \frac{\lambda s}{2}\nrm{\nu+\bar{g}}^2 - \iprd{\nu}{f}{}+ \frac{s}{p} \norm{\nabla\nu}{L^p}^p +\frac{s\varepsilon^2}{2}\nrm{\Delta \nu}^2
	\nonumber
	\\
= & \ \frac{1}{2}\nrm{\nu-g+\bar{g}}_{\mathring{H}_{\rm per}^{-1}}^2 + \frac{\lambda s}{2}\nrm{\nu+\bar{g}}^2 - \iprd{\nu}{f}{}+ \frac{s}{p} \norm{\nabla \nu}{L^p}^p +\frac{s\varepsilon^2}{2}\nrm{\Delta \nu}^2,
	\label{sixth-order-energy}
	\end{align}
where by $\bar{g}$ we denote the average of $g$ over $\Omega$.
Observe that $\nu-g+\bar{g}\in \mathring{H}_{\rm per}^{-1}$, which is required for this energy to make sense. It is straightforward to show that $E$ satisfies (E1) -- (E3), with respect to the Hilbert space $H = \mathring{H}_{\rm per}^2(\Omega)$. The first variation of $E$ is given as follows: for any $\xi \in \mathring{H}^2_{\rm per}(\Omega)$, 
	\begin{align*}
\left. d_\tau E[\nu+\tau \xi]\right|_{\tau=0} = \delta E[\nu](\xi) = & \  \iprd{\msfT[\nu-g+\bar{g}]}{\xi} +\lambda s\iprd{\nu+\bar{g}}{\xi}-\iprd{f}{\xi} 
	\nonumber
	\\
& +  s\iprd{|\nabla\nu|^{p-2}\nabla\nu}{\nabla \xi} +s\varepsilon^2\iprd{\Delta \nu}{\Delta \xi} .
	\end{align*}
The unique minimizer of $E$ -- let us call it $u_\star \in \mathring{H}_{\rm per}^2(\Omega)$ for the moment -- satisfies $\delta E[u_\star](\xi) = 0$, for all $\xi\in \mathring{H}^2_{\rm per}(\Omega)$. By the definition of the $\msfT$ operator, there is a unique element $w_\star\in \mathring{H}_{\rm per}^1(\Omega)$ such that
	\[
w_\star : = -  \msfT[u_\star-g+\bar{g}] .
	\]
Therefore, we have, for all $\xi\in \mathring{H}^2_{\rm per}(\Omega)$,
	\[
s \lambda \iprd{ u_\star+\bar{g}}{\xi}  +  s \iprd{|\nabla u_\star|^{p-2}\nabla u_\star}{\nabla \xi} + s\varepsilon^2\iprd{\Delta u_\star}{\Delta \xi} -\iprd{w_\star}{\xi} = \iprd{f}{\xi} .
	\]
Setting $u := u_\star+\bar{g}$ and $w := w_\star +s\lambda\bar{g} - \bar{f}$ and using the fact that $\xi$ is of zero mean, we have 
	\[
s \lambda \iprd{ u}{\xi}  +  s \iprd{|\nabla u|^{p-2}\nabla u}{\nabla \xi} + s\varepsilon^2\iprd{\Delta u}{\Delta \xi} -\iprd{w}{\xi} = \iprd{f}{\xi} , \quad \forall\ \xi\in \mathring{H}_{\rm per}^2.
	\]
Using the definition of the $\msfT$ operator again, we conclude that $w_\star \in \mathring{H}_{\rm per}^1(\Omega)$ satisfies
	\[
\iprd{\nabla w_\star}{\nabla \chi}{} = -\iprd{u_\star-g+\bar{g}}{\chi}{} , 
	\]
for all $\chi\in \mathring{H}_{\rm per}^1(\Omega)$, which implies that
	\[
\iprd{\nabla w}{\nabla \chi}{} = -\iprd{u-g}{\chi}{} .
	\]
It follows that solving \eqref{eq:6th-mixed-weak a} -- \eqref{eq:6th-mixed-weak b} is equivalent to minimizing the coercive, strictly convex energy \eqref{sixth-order-energy}, after the appropriate affine change of variables.

The second variation of $E$ is a continuous bilinear operator. Given a fixed  $\nu \in \mathring{H}^2_{\rm per}(\Omega)$, the action of the second variation on the  arbitrary pair $(\xi,\eta)\in \mathring{H}^2_{\rm per}(\Omega)\times \mathring{H}^2_{\rm per}(\Omega)$ becomes 
	\begin{align*}
\delta^2 E[\nu](\xi,\eta) = & \ \iprd{\xi}{\msfT[\eta]} +  \lambda s   \iprd{\xi}{\eta}+ s\iprd{|\nabla \nu|^{p-2} \nabla \xi}{\nabla \eta}
	\nonumber
	\\
& +(p-2)s\iprd{|\nabla \nu|^{p-4}\nabla \nu \cdot \nabla \xi}{\nabla \nu \cdot \nabla \eta} +s\varepsilon^2\iprd{\Delta \xi}{\Delta \eta}.
\label{eqn:62nddev}
	\end{align*}

Similar to the estimate in the  fourth-order case~\eqref{eq:2ndvarboundE4}, we have the bound 
	\begin{align*}
\left|\delta^2 E[\nu](\xi,\eta)\right| \le & \ \nrm{\xi}_{\mathring{H}_{\rm per}^{-1}}\nrm{\eta}_{\mathring{H}_{\rm per}^{-1}} + \lambda s \nrm{\xi} \cdot \nrm{\eta} + s\norm{\nabla \nu}{L^p}^{p-2}\norm{\nabla \xi}{L^p}\norm{\nabla \eta}{L^p}
	\nonumber
	\\
& \ +(p-2)s\norm{\nabla \nu}{L^p}^{p-2}\norm{\nabla \xi}{L^p}\norm{\nabla \eta}{L^p} +s\varepsilon^2\nrm{\Delta \xi} \cdot \nrm{\Delta \eta} ,
	\end{align*}
which implies that 
	\begin{equation}
\left|\delta^2 E[\nu](\xi,\xi)\right| \le  \norm{\xi}{\mathring{H}^{-1}_{\rm per}}^2+ s\lambda \nrm{\xi}^2 +(p-1)s\norm{\nabla \nu}{L^p}^{p-2}\norm{\nabla \xi}{L^p}^2 +s\varepsilon^2\nrm{\Delta \xi}^2,
	\label{eq:2ndbound}
	\end{equation}
for all $\nu,\xi \in \mathring{H}^2_{\rm per}(\Omega)$.

For the sixth order problem, we define the pre-conditioner ${\mathcal L} :\mathring{H}^2_{\rm per}(\Omega) \to \mathring{H}^{-2}_{\rm per}(\Omega)$ via
	\begin{equation}
\langle {\mathcal L}[\nu],\xi\rangle := s\lambda \iprd{\nu}{\xi}+\iprd{\nu}{\xi}_{\mathring{H}^{-1}_{\rm per}} +s\iprd{\nabla \nu}{\nabla \xi} +s\varepsilon^2 \iprd{\Delta \nu}{\Delta \xi} , \quad \forall \ \xi \in \mathring{H}^2_{\rm per}(\Omega).
	\label{eq:defofL6}
	\end{equation}
This operator satisfies (L1) -- (L3).  To show that it satisfies (L3) -- (L6), we need some technical results.
	\begin{lem}
For every $\xi \in \mathring{H}_{\rm per}^2(\Omega)$ we have
	\begin{equation}
	\label{eq:intperiodic-2}
  \nrm{\xi} \le \nrm{\xi}_{\mathring{H}^{-1}_{\rm per}}^{\frac23}\nrm{\Delta \xi}^{\frac13},
	\end{equation}
and
	\begin{equation}
	\label{eq:intperiodic}
\nrm{\nabla \xi} \le \nrm{\xi}_{\mathring{H}^{-1}_{\rm per}}^{\frac13}\nrm{\Delta \xi}^{\frac23}.
	\end{equation}
	\end{lem}
	\begin{proof}
Using integration by parts we get 
	\begin{equation}
\| \nabla \xi \|^2 =  - ( \xi, \Delta \xi ) \le \| \xi \| \cdot \| \Delta \xi \| . 
	\label{lem 4-43-1}
	\end{equation}
The definition of the $\mathring{H}^{-1}_{\rm per}(\Omega)$ norm implies that 
	\begin{equation}  
\nrm{\xi}^2 = \iprd{\xi}{\xi}{} \le  \|\xi\|_{\mathring{H}^{-1}_{\rm per}}  \| \nabla \xi \| .
	\label{lem 4-43-2}
	\end{equation}
Therefore, a combination of (\ref{lem 4-43-1}) and (\ref{lem 4-43-2}) leads to 
	\[
\| \nabla \xi \| \le \| \xi \|^{\frac{1}{2}} \cdot \| \Delta \xi \|^{\frac{1}{2}} 
\le \| \xi \|_{\mathring{H}^{-1}_{\rm per}}^{\frac{1}{4}}  \| \nabla \xi \|^{\frac{1}{4}} \cdot \| \Delta \xi \|^{\frac{1}{2}}   , 
	\]
so that
	\[
\| \nabla \xi \|^{\frac{3}{4}} \le \|\xi \|_{\mathring{H}^{-1}_{\rm per}}^{\frac{1}{4}} \| \Delta \xi \|^{\frac{1}{2}} .
	\]
which yields the second inequality. The first may be proved in a similar way.
	\end{proof}

Similar to before, the Gagliardo-Nirenberg inequality, together with elliptic regularity, yield the following interpolation result.

	\begin{lem}
	\label{lem:lemma3-six}
Suppose that $p\in [2,\infty)$ when $d=2$, and   $p \in[2,6]$, if $d=3$. 
For any $\xi\in \mathring{H}^2_{\rm per}(\Omega)$, we have 
	\begin{equation}
\norm{\nabla \xi}{L^p} \le C_9 \left\{
	\begin{array}{llll}
\norm{\xi}{\mathring{H}^{-1}_{\rm per}}^{\frac2{3p}} \nrm{\Delta \xi}^{1-\frac2{3p}},&\mbox{if} & d=2,  & p\in [2,\infty),
	\\
\nrm{\xi}_{\mathring{H}^{-1}_{\rm per}}^{\frac{1}{p}-\frac{1}{6}}  \nrm{\Delta \xi}^{\frac{7}{6}-\frac{1}{p}}, & \mbox{if} &  d=3, & p \in [2, 6],
	\end{array}
\right.
	\label{W1p-estimate-six}
	\end{equation}
for some $C_9 = C_9(d,p)>0$. 
	\end{lem}

We can now find the coefficients $C_5$, $C_6$, and $C_8$, which establish properties (L4) -- (L6) and therefore guarantee the geometric convergence of the PSD method for the sixth-order case.

	\begin{lem}
	\label{lem:lemma6}
For any $\nu,\xi \in \mathring{H}^2_{\rm per}(\Omega)$,  we have
	\begin{equation}
C_5 \norm{\xi-\nu}{\mathcal{L}}^2 \le \left(\delta E[\xi] -\delta E[\nu]\right)(\xi -\nu)  , 
	\label{lem 6-0}
	\end{equation}
where $C_5 = \min \left( \frac13, \varepsilon^{\frac43} s^{-\frac13} \right)$. Let $E_0$ be given such that $B:= \left\{\xi \in  \mathring{H}_{\rm per}^2(\Omega) \ \middle| \ E[\xi]\le E_0\right\}$ is non-empty. For any $\nu\in B$ and any $\xi\in \mathring{H}^2_{\rm per}(\Omega)$, the following estimate is valid: 
	\begin{equation}
\left| \delta^2 E[\nu](\xi,\xi) \right| \leq  C_6 \norm{\xi}{\mathcal{L}}^2 ,
	\label{lem 7-0-2} 
	\end{equation}
where 
	\begin{equation} 
C_6 = \left\{
	\begin{array}{lll}
1+(p-1)\left(\frac{3p}{2}\right)^{-\frac{2}{3p}} \left(\frac{3p}{3p-2}\right)^{\frac{2-3p}{3p}} \varepsilon^{\frac{4-6p}{3p}} s^{\frac{2}{3p}} C_9^2C_{10}^{p-2}, & \mbox{for} & p\in[2,\infty),\ d= 2,
	\\
1+ (p-1)\left(\frac{6p}{6-p}\right)^{\frac{p-6}{6p}} \left(\frac{6p}{7p-6}\right)^{\frac{6-7p}{6p}} \varepsilon^{\frac{6-7p}{3p}} s^{\frac{6-p}{6p}}  C_9^2C_{10}^{p-2}, & \mbox{for} & p\in[2,6),\ d = 3,
	\\
1+(p-1)\varepsilon^{-2}C_9^2C_{10}^{p-2}, & \mbox{for} & p=6,\ d = 3,
	\end{array}
\right.
	\label{lem 7-0-3}  
	\end{equation}
and $C_{10} = (p E_0)^{\frac{1}{p}}$. We can take $C_8 = C_5$  to satisfy  estimate \eqref{assmp-L6} of assumption (L6).
	\end{lem}
	
	\begin{proof}
The proof is similar to that of Lemma \ref{lem:lemma4}. Using \eqref{lem 4-2} again, we have
	\begin{align*} 
\left(\delta E[\xi]-\delta E[\nu]\right)(\xi-\nu) = & \ s\lambda \nrm{\xi-\nu}^2 +\norm{\xi-\nu}{\mathring{H}^{-1}_{\rm per}}^2+ s\varepsilon^2  \nrm{\Delta (\xi-\nu)}^2
	\nonumber
	\\
& \ + s \left( | \nabla \xi |^{p-2} \nabla \xi - | \nabla \nu |^{p-2} \nabla \nu , \nabla (\xi-\nu) \right) . 
	\nonumber
	\\
\ge & \ s\lambda \nrm{\xi-\nu}^2 +\norm{\xi-\nu}{\mathring{H}^{-1}_{\rm per}}^2+ s\varepsilon^2  \nrm{\Delta (\xi-\nu)}^2
	\nonumber 
	\\
\ge & \ s\lambda \nrm{\xi-\nu}^2 +\frac23 \norm{\xi-\nu}{\mathring{H}^{-1}_{\rm per}}^2 + \frac13s \varepsilon^2  \nrm{\Delta (\xi-\nu)}^2 
	\\
	\nonumber 
& \ + s^{\frac23}\varepsilon^{\frac{4}{3}}  \nrm{ \nabla (\xi-\nu) }^2 ,
	\end{align*}
where the last step is a consequence of the interpolation inequality \eqref{eq:intperiodic}:
	\begin{equation*} 
\frac13 \norm{\xi-\nu}{\mathring{H}^{-1}_{\rm per}}^2 + \frac{2}{3}s\varepsilon^2 \nrm{\Delta(\xi-\nu)}^2   \ge  s^{\frac23}\varepsilon^{\frac{4}{3}}  \norm{\xi-\nu}{\mathring{H}^{-1}_{\rm per}}^{\frac23} \nrm{\Delta (\xi-\nu)}^{\frac43} \ge   s^{\frac23}\varepsilon^{\frac{4}{3}}  \nrm{ \nabla (\xi-\nu) }^2 .
	\end{equation*}
We conclude that estimate \eqref{lem 6-0} holds by choosing $C_5 = \min ( \frac13, \varepsilon^{\frac43} s^{-\frac13})$.   

Next we derive \eqref{lem 7-0-2}. Inequality \eqref{eq:2ndbound} yields
	\begin{equation*}
\left|\delta^2 E[\nu](\xi,\xi)\right| \le  s\lambda \nrm{\xi}^2+ \norm{\xi}{\mathring{H}^{-1}_{\rm per}}^2 +(p-1)s\norm{\nabla \nu}{L^p}^{p-2}\norm{\nabla \xi}{L^p}^2 +s\varepsilon^2\nrm{\Delta \xi}^2.
	\label{lem 7-1}
	\end{equation*}	 
Since $\nu\in B$, $\norm{\nabla \nu}{L^p} \le (E(u^0))^{\frac{1}{p}} =: C_{10}$. Suppose that $d=2$. An application of the Sobolev inequality (\ref{W1p-estimate-six}) from Lemma \ref{lem:lemma3-six} indicates that, for every $\xi\in\mathring{H}_{\rm per}^2(\Omega)$, 
	\begin{align*}
 \left(\frac{3p}{2}\right)^{\frac{2}{3p}}& \left(\frac{3p}{3p-2}\right)^{\frac{3p-2}{3p}} \varepsilon^{\frac{6p-4}{3p}} s^{\frac{3p-2}{3p}} C_9^{-2}  \norm{\nabla \xi}{L^p}^2 	\nonumber
\\
\le & \  \left(\frac{3p}{2}\right)^{\frac{2}{3p}}\norm{\xi}{\mathring{H}^{-1}_{\rm per}}^{\frac{4}{3p}} \left(\frac{3p}{3p-2}\right)^{\frac{3p-2}{3p}} \varepsilon^{\frac{6p-4}{3p}} s^{\frac{3p-2}{3p}} \nrm{\Delta \xi}^{\frac{6p-4}{3p}} \nonumber
\\
\le & \ \norm{\xi}{\mathring{H}^{-1}_{\rm per}}^2 + s\varepsilon^2  \nrm{\Delta \xi}^2  ,
	\end{align*}  
where, in the last step, we applied Young's inequality. It follows that, 
	\begin{align*}
& \hspace{-0.5in} (p-1)s\norm{\nabla \nu}{L^p}^{p-2} \norm{\nabla \xi}{L^p}^2 
	\\
\le & \  (p-1)\left(\frac{3p}{2}\right)^{-\frac{2}{3p}} \left(\frac{3p}{3p-2}\right)^{\frac{2-3p}{3p}} \varepsilon^{\frac{4-6p}{3p}} s^{\frac{2}{3p}}  C_9^2C_{10}^{p-2}\left(\norm{\xi}{\mathring{H}^{-1}_{\rm per}}^2 + s\varepsilon^2  \nrm{\Delta \xi}^2\right). 
	\end{align*} 
As a result, estimate \eqref{lem 7-0-2} is valid by choosing 
	\[
C_6 = 1+(p-1)\left(\frac{3p}{2}\right)^{-\frac{2}{3p}} \left(\frac{3p}{3p-2}\right)^{\frac{2-3p}{3p}} \varepsilon^{\frac{4-6p}{3p}} s^{\frac{2}{3p}} C_9^2C_{10}^{p-2}.
	\]
Similarly, For $p\in[2,6),\ d = 3$, we have
	\begin{align*}
 \left(\frac{6p}{6-p}\right)^{\frac{6-p}{6p}}& \left(\frac{6p}{7p-6}\right)^{\frac{7p-6}{6p}} \varepsilon^{\frac{7p-6}{3p}} s^{\frac{7p-6}{6p}} C_9^{-2}  \norm{\nabla \xi}{L^p}^2
 	\nonumber
	\\
\le & \  \left(\frac{6p}{6-p}\right)^{\frac{6-p}{6p}}\norm{\xi}{\mathring{H}^{-1}_{\rm per}}^{\frac{6-p}{3p}}  \left(\frac{6p}{7p-6}\right)^{\frac{7p-6}{6p}} \varepsilon^{\frac{7p-6}{3p}} s^{\frac{7p-6}{6p}} \nrm{\Delta \xi}^{\frac{7p-6}{3p}}
	\nonumber
	\\
\le & \ \norm{\xi}{\mathring{H}^{-1}_{\rm per}}^2 + s\varepsilon^2  \nrm{\Delta \xi}^2 .
	\end{align*}  
As a result, estimate \eqref{lem 7-0-2} is valid by choosing 
	\[
C_6 = 1+ (p-1)\left(\frac{6p}{6-p}\right)^{\frac{p-6}{6p}} \left(\frac{6p}{7p-6}\right)^{\frac{6-7p}{6p}} \varepsilon^{\frac{6-7p}{3p}} s^{\frac{6-p}{6p}}  C_9^2C_{10}^{p-2}.
	\]

For the case $p=6,\ d = 3$, the Sobolev inequality \eqref{W1p-estimate-six} degenerates, as before. But it is straightforward to show that  estimate (\ref{lem 7-0-2}) is valid upon choosing 
	\[
C_6 = 1+(p-1)\varepsilon^{-2}C_9^2C_{10}^{p-2}.
	\]	
As before, we omit the simple argument that one may take $C_8 = C_5$ to satisfy (L6). The proof is complete.
	\end{proof}

	\section{Finite Difference Spatial Discretization in 2D}
	\label{sec:fdm}

	\subsection{Notation}
	\label{sec:notation}
In this subsection we define the discrete spatial difference operators, function space, inner products and norms, following the notation used in \cite{hu09, shen2012second, wang2010unconditionally, wang2011energy, wise09a}.  Let $\Omega = (0,L_x)\times(0,L_y)$, where, for simplicity, we assume $L_x =L_y =: L > 0$. We write $L = m\cdot h$, where $m$ is a positive integer. The parameter $h = \frac{L}{m}$ is called the mesh or grid spacing. We define the following two uniform, infinite grids with grid spacing $h>0$:
	\[
E := \{ x_{i+\hf} \ |\ i\in {\mathbb{Z}}\}, \quad C := \{ x_i \ |\ i\in {\mathbb{Z}}\},
	\]
where $x_i = x(i) := (i-\hf)\cdot h$. Consider the following 2D discrete periodic function spaces: 
	\begin{eqnarray*}
{\mathcal V}_{\rm per} &:=& \left\{\nu: E\times E\rightarrow {\mathbb{R}}\ \middle| \ \nu_{i+\frac12,j+\frac12}= \nu_{i+\frac12+\alpha m,j+\frac12+\beta m}, \ \forall \, i,j,\alpha,\beta\in \mathbb{Z}  \right\},
	\\
{\mathcal C}_{\rm per} &:=& \left\{\nu: C\times C
\rightarrow {\mathbb{R}}\ \middle| \ \nu_{i,j} = \nu_{i+\alpha m,j+\beta m}, \ \forall \, i,j,\alpha,\beta\in \mathbb{Z} \right\},
	\\
{\mathcal E}^{\rm ew}_{\rm per} &:=& \left\{\nu: E\times C\rightarrow {\mathbb{R}}\ \middle| \ \nu_{i+\frac12,j}= \nu_{i+\frac12+\alpha m,j+\beta m}, \ \forall \,  i,j,\alpha,\beta\in \mathbb{Z}  \right\},
	\\
{\mathcal E}^{\rm ns}_{\rm per} &:=& \left\{\nu: C \times E\rightarrow {\mathbb{ R}}\ \middle| \ \nu_{i,j+\frac12}= \nu_{i+\alpha m,j+\frac12+\beta m}, \ \forall \, i,j,\alpha,\beta\in \mathbb{Z}  \right\}.
	\end{eqnarray*}		
The functions of ${\mathcal V}_{\rm per}$ are called {\emph{vertex centered functions}}; those of ${\mathcal C}_{\rm per}$ are called {\emph{cell centered functions}}. The functions of ${\mathcal E}^{\rm ew}_{\rm per}$ are called {\emph{east-west edge-centered functions}}, and the functions of ${\mathcal E}^{\rm ns}_{\rm per}$ are called {\emph{north-south edge-centered functions}}.  We also define the mean zero space 
	\[
\mathring{\mathcal C}_{\rm per}:=\left\{\nu\in {\mathcal C}_{\rm per} \ \middle| \  \frac{h^2}{| \Omega|} \sum_{i,j=1}^m \nu_{i,j} =: \overline{\nu}  = 0\right\} .
	\]

We now define the important difference and average operators on the spaces:  
	\begin{eqnarray*}
&& A_x \nu_{i+\hf,\Box} := \frac{1}{2}\left(\nu_{i+1,\Box} + \nu_{i,\Box} \right), \quad D_x \nu_{i+\hf,\Box} := \frac{1}{h}\left(\nu_{i+1,\Box} - \nu_{i,\Box} \right),\\
&& A_y \nu_{\Box,i+\hf} := \frac{1}{2}\left(\nu_{\Box,i+1} + \nu_{\Box,i} \right), \quad D_y \nu_{\Box,i+\hf} := \frac{1}{h}\left(\nu_{\Box,i+1} - \nu_{\Box,i} \right) , 
	\end{eqnarray*}
with $A_x,\, D_x: {\mathcal C}_{\rm per}\rightarrow{\mathcal E}_{\rm per}^{\rm ew}$ if $\Box$ is an integer, and $A_x,\, D_x: {\mathcal E}^{\rm ns}_{\rm per}\rightarrow{\mathcal V}_{\rm per}$ if $\Box$ is a half-integer, with $A_y,\, D_y: {\mathcal C}_{\rm per}\rightarrow{\mathcal E}_{\rm per}^{\rm ns}$ if $\Box$ is an integer, and $A_y,\, D_y: {\mathcal E}^{\rm ew}_{\rm per}\rightarrow{\mathcal V}_{\rm per}$ if $\Box$ is a half-integer. Likewise,
	\begin{eqnarray*}
&&a_x \nu_{i,\Box} := \frac{1}{2}\left(\nu_{i+\hf,\Box} + \nu_{i-\hf,\Box} \right),	 \quad d_x \nu_{i,\Box} := \frac{1}{h}\left(\nu_{i+\hf,\Box} - \nu_{i-\hf,\Box} \right),
	\\
&&a_y \nu_{\Box,j} := \frac{1}{2}\left(\nu_{\Box,j+\hf} + \nu_{\Box,j-\hf} \right),	 \quad d_y \nu_{\Box,j} := \frac{1}{h}\left(\nu_{\Box,j+\hf} - \nu_{\Box,j-\hf} \right),
	\end{eqnarray*}
with $a_x,\, d_x : {\mathcal E}_{\rm per}^{\rm ew}\rightarrow{\mathcal C}_{\rm per}$ if $\Box$ is an integer, and $a_x,\ d_x: {\mathcal V}_{\rm per}\rightarrow{\mathcal E}^{\rm ns}_{\rm per}$ if $\Box$ is a half-integer; and with $a_y,\, d_y : {\mathcal E}_{\rm per}^{\rm ns}\rightarrow{\mathcal C}_{\rm per}$ if $\Box$ is an integer, and $a_y,\ d_y: {\mathcal V}_{\rm per}\rightarrow{\mathcal E}^{\rm ew}_{\rm per}$ if $\Box$ is a half-integer.
    
Define the 2D center-to-vertex derivatives $\mD_x,\, \mD_y : {\mathcal C}_{\rm per}\rightarrow{\mathcal V}_{\rm per}$ component-wise as
	\begin{eqnarray*}
\mD_x \nu_{i+\hf,j+\hf} &:=& A_y(D_x\nu)_{i+\hf,j+\hf} = D_x(A_y\nu)_{i+\hf,j+\hf}
	\nonumber
	\\
&=& \frac{1}{2h}\left(\nu_{i+1,j+1}-\nu_{i,j+1}+\nu_{i+1,j}-\nu_{i,j} \right) ,
	\\
\mD_y\nu_{i+\hf,j+\hf} &:=& A_x(D_y\nu)_{i+\hf,j+\hf} = D_y(A_x\nu)_{i+\hf,j+\hf} 
	\nonumber
	\\
&=& \frac{1}{2h}\left(\nu_{i+1,j+1}-\nu_{i+1,j}+\nu_{i,j+1}-\nu_{i,j} \right) . 
	\end{eqnarray*}		

The utility of these definitions is that the differences $\mD_x$ and $\mD_y$ are collocated on the grid, unlike the case for $D_x$, $D_y$. Define the 2D vertex-to-center derivatives $\md_x,\, \md_y  : {\mathcal V}_{\rm per}\rightarrow{\mathcal C}_{\rm per}$  component-wise as
	\begin{eqnarray*}
\md_x \nu_{i,j} &:=& a_y(d_x \nu)_{i,j} = d_x(a_y \nu)_{i,j} 
	\nonumber
	\\
&=& \frac{1}{2h}\left(\nu_{i+\hf,j+\hf}- \nu_{i-\hf,j+\hf}+\nu_{i+\hf,j-\hf}-\nu_{i-\hf,j-\hf}\right) ,
	\\
\md_y \nu_{i,j} &:=& a_x(d_y \nu)_{i,j} = d_y(a_x \nu)_{i,j}
	\nonumber
	\\
&=& \frac{1}{2h}\left(\nu_{i+\hf,j+\hf}- \nu_{i+\hf,j-\hf}+\nu_{i-\hf,j+\hf}-\nu_{i-\hf,j-\hf}\right).
	\end{eqnarray*}

Now the discrete gradient operator, $\nabla^v_h$: ${\mathcal C}_{\rm per}\rightarrow {\mathcal V}_{\rm per}\times {\mathcal V}_{\rm per}$, is defined as 
	\[
\nabla^v_h \nu_{i+\hf,j+\hf} := (\mD_x \nu_{i+\hf,j+\hf}, \mD_y \nu_{i+\hf,j+\hf}). 
	\]	
The standard 2D discrete Laplacian, $\Delta_h : {\mathcal C}_{\rm per}\rightarrow{\mathcal C}_{\rm per}$, is given by 
	\[
\Delta_h \nu_{i,j} := d_x(D_x \nu)_{i,j} + d_y(D_y \nu)_{i,j} = \frac{1}{h^2}\left( \nu_{i+1,j}+\nu_{i-1,j}+\nu_{i,j+1}+\nu_{i,j-1} - 4\nu_{i,j}\right).
	\]
The 2D vertex-to-center average, $\mathcal{A} : {\mathcal V}_{\rm per}\rightarrow{\mathcal C}_{\rm per}$, is defined  to be 
	\[
\mathcal{A} \nu_{i,j} := \frac{1}{4}\left( \nu_{i+1,j}+\nu_{i-1,j}+\nu_{i,j+1}+\nu_{i,j-1}\right).
	\]		
The 2D \emph{skew} Laplacian, $\Delta^v_h : {\mathcal C}_{\rm per}\rightarrow{\mathcal C}_{\rm per}$, is defined as
	\begin{eqnarray*}
\Delta^v_{h}\nu_{i,j} &=& \md_x(\mD_x\nu)_{i,j} + \md_y( \mD_y\nu)_{i,j} 
	\nonumber
	\\
&=& \frac{1}{2h^2}\left(\nu_{i+1,j+1}+\nu_{i-1,j+1}+\nu_{i+1,j-1}+\nu_{i-1,j-1} - 4\nu_{i,j}\right) .
	\end{eqnarray*}		
The 2D discrete p-Laplacian operator is defined as  
	\begin{eqnarray*}
\nabla_h^v \cdot \left( \left| \nabla_h^v \nu\right|^{p-2} \nabla_h^v \nu \right)_{ij} := \md_x(r\, \mD_x\nu )_{i,j}+ \md_y(r\, \mD_y\nu )_{i,j},
	\end{eqnarray*}	
with 
	\[
r_{i+\frac{1}{2},j+\frac{1}{2}}:=\left[(\mD_x u )_{i+\frac{1}{2},j+\frac{1}{2}}^2+(\mD_y u )_{i+\frac{1}{2},j+\frac{1}{2}}^2\right]^{\frac{p-2}{2}}.
	\]	
Clearly, for $p=2$, $\Delta^v_{h}\nu = \nabla_h^v \cdot \left( \left| \nabla_h^v \nu\right|^{p-2} \nabla_h^v \nu \right)$.

Now we are ready to define the following grid inner products:  
	\begin{eqnarray*}
\ciptwo{\nu}{\xi}_2 &:=& h^2\sum_{i=1}^m\sum_{j=1}^n \nu_{i,j}\psi_{i,j},\quad \nu,\, \xi\in {\mathcal C}_{\rm per},
\\
\viptwo{\nu}{\xi} &:=& \ciptwo{\mathcal{A}(\nu\xi)}{1}_2 ,\quad \nu,\, \xi\in{\mathcal V}_{\rm per},
\\
\eipew{\nu}{\xi} &:=& \ciptwo{A_x(\nu\xi)}{1}_2 ,\quad \nu,\, \xi\in{\mathcal E}^{\rm ew}_{\rm per},
\\
\eipns{\nu}{\xi} &:=& \ciptwo{A_y(\nu\xi)}{1}_2 ,\quad \nu,\, \xi\in{\mathcal E}^{\rm ns}_{\rm per}.
	\end{eqnarray*}	
Suppose that $\zeta\in\mathring{\mathcal C}_{\rm per}$, then there is a unique solution $\msfT_h[\zeta]\in\mathring{\mathcal C}_{\rm per}$ such that $-\Delta_h \msfT_h[\zeta] = \zeta$. We often write, in this case, $\msfT_h[\zeta] = -\Delta^{-1}_h \zeta$. The discrete analogue of the $\mathring{H}^{-1}_{\rm per}$ inner product is defined as 
	\[
\moneinn{\zeta,\xi} := \iprd{\zeta}{\msfT_h[\xi]}_2 = \iprd{\msfT_h[\zeta]}{\xi}_2,\quad \zeta,\, \xi\in\mathring{\mathcal C}_{\rm per}.
	\]
where summation-by-parts \cite{shen2012second,wise09a} guarantees the symmetry and the second equality.

We now define the following norms for cell-centered functions.  If $\nu\in\mathring{\mathcal C}_{\rm per}$, then $\monenrm{\nu}^2 = \moneinn{\nu,\nu}$. If $\nu\in {\mathcal C}_{\rm per}$, then $\nrm{\nu}_2^2 := \ciptwo{\nu}{\nu}_{2}$; $\nrm{\nu}_p^p := \ciptwo{|\nu|^p}{1}_{2}$ ($1\le p< \infty$), and $\nrm{\nu}_\infty := \max_{1\le i\le m \atop 1\le j\le n}\left|\nu_{i,j}\right|$.
Similarly, we define the gradient norms: for $\nu\in{\mathcal C}_{\rm per}$,
	\[
\nrm{\nabh^v\nu}_p^p := \langle |\nabla_h^v\nu|^p, 1\rangle, \quad |\nabh^v\nu|^p:=[(\mD_x\nu)^2 +(\mD_y\nu)^2]^{\frac{p}{2}} = \left[\nabla_h^v\nu\cdot\nabla_h^v\nu  \right]^{\frac{p}{2}}  \in \mathcal{V}_{\rm per}, \quad 2\le p < \infty, 
	\]
and
	\[
\nrm{ \nabla_h \nu}_2^2 : = \eipew{D_x\nu}{D_x\nu} + \eipns{D_y\nu}{D_y\nu} .
	\]
	
	\subsection{Discrete Sobolev Inequalities}
	\begin{lem}
	\label{lem:dis-lemma3}
Suppose that $p\in [2,\infty)$, $d=2$, we have 
	\begin{equation*}
	\label{dis-sob-inq}
\norm{\nabla_h^v \xi}{p} \le C_9 \left\{
	\begin{array}{llll}
\norm{\xi}{2}^{\frac1p} \cdot \norm{\Delta_h \xi}{2}^{\frac{p-1}{p}}, & \forall &  \xi \in \mathcal{C}_{\rm per},
\\
\norm{\xi}{-1}^{\frac{2}{3p}} \cdot \norm{\Delta_h \xi}{2}^{1-\frac{2}{3p}},  &\forall & \xi\in  \mathring{\mathcal{C}}_{\rm per},
	\end{array}
	\right.
	\end{equation*}
for some $C_9 = C_9(p)>0$. 
	\end{lem}
The proof for $p=4, d=2$ can be found in the appendix. Following the similar arguments, the other cases can be proved.

	\begin{rmk}
Though we have focused on the case $d=2$ in this section, we can also define our operators and norms, in particular $\nabla_h^v \xi$ and $\norm{\nabla_h^v \xi}{p}$, in three space dimensions. Then for $p \in[2,6]$, we expect
	\begin{equation*}
\norm{\nabla_h^v \xi}{p} \le C_9 \left\{
	\begin{array}{llll}
\norm{\xi}{2}^{\frac{3}{2p}-\frac{1}{4}}   \norm{\Delta_h \xi}{2}^{\frac{5}{4}-\frac{3}{2p}}, & \forall &  \xi \in \mathcal{C}_{\rm per},
	\\
\nrm{\xi}_{-1}^{\frac{1}{p}-\frac{1}{6}}   \nrm{\Delta_h \xi}_2^{\frac{7}{6}-\frac{1}{p}}, & \forall &  \xi\in  \mathring{\mathcal{C}}_{\rm per},
	\end{array}
\right.
	\label{dis-W1p-estimate-six}
	\end{equation*}
for some $C_9 = C_9(d=3,p)>0$. 
\end{rmk}

	\subsection{Convergence for the Discretized Fourth-Order Problem}
	
The discrete version of \eqref{eqn:4ths} can be expressed as follows: given $f \in \mathcal{C}_{\rm per}$, find $u\in \mathcal{C}_{\rm per}$ such that
	\begin{equation}
u -s \nabla_h^v \cdot \left( \left| \nabla_h^v u \right|^{p-2} \nabla_h^v u \right) + s\varepsilon^2 \Delta_h^2 u = f.
	\label{discrete-4th-order}
	\end{equation}
This represents a second-order approximation of the solution of~\eqref{eqn:4ths}. As in the space continuous case, we formulate an equivalent minimization problem. Using the definitions from subsection \ref{sec:notation}, we have the following discrete energy: given $f \in \mathcal{C}_{\rm per}$, for any $\nu\in \mathcal{C}_{\rm per}$, define
	\begin{equation}
	\label{eqn:diseng4}
E_h[\nu] :=  \frac{1}{2} \|\nu-f \|_2^2 +\frac{s}{p} \norm{\nabla_h^v \nu}{p}^p +\frac{s\varepsilon^2}{2} \| \Delta_h \nu \|_2^2 .
	\end{equation} 
This (discrete) energy satisfies (E1) -- (E3).  The discrete variational derivative at $\nu\in\mathcal{C}_{\rm per}$ is 
	\begin{align*}
\delta E_h[\nu](\xi):= & \ d_\tau E_h[\nu+\tau \xi]|_{\tau=0} 
	\nonumber
	\\
= & \ (\nu-f,\xi)_2 +s\langle |\nabla_h^v \nu|^{p-2}\mD_x \nu,\mD_x \xi \rangle +s\langle |\nabla_h^v \nu|^{p-2}\mD_y \nu,\mD_y \xi \rangle +s\varepsilon^2 (\Delta_h \nu, \Delta_h \xi)_2
	\nonumber
	\\
= & \ (\nu-f,\xi)_2 +s\langle |\nabla_h^v \nu|^{p-2}\nabla_h^v \nu, \nabla_h^v \xi \rangle +s\varepsilon^2 (\Delta_h \nu, \Delta_h \xi)_2
	\nonumber
	\\
= & \ \left(\nu-f - s\nabla_h^v \cdot \left( \left| \nabla_h^v \nu\right|^{p-2} \nabla_h^v \nu \right) + s\varepsilon^2\Delta_h^2 \nu, \xi \right)_2 ,
	\end{align*}
for all $\xi \in \mathcal{C}_{\rm per}$, where we have used summation-by-parts \cite{shen2012second,wise09a} to obtain the last equality.  Given a fixed  $\nu\in {\mathcal C}_{\rm per}$, the action of the second variation on the  arbitrary pair $(\xi,\eta)\in {\mathcal C}_{\rm per}\times {\mathcal C}_{\rm per}$ is given by
	\begin{align*}
\delta^2 E_h[\nu](\xi,\eta) = & \ \iprd{\xi}{\eta}_2 + s\langle |\nabla_h^v \nu|^{p-2} \nabla_h^v \xi, \nabla_h^v \eta\rangle 
	\nonumber
	\\
& +(p-2)s\langle |\nabla_h^v \nu|^{p-4}\nabla_h^v \nu \cdot \nabla_h^v \xi , \nabla_h^v \nu \cdot \nabla_h^v \eta\rangle +s\varepsilon^2\iprd{\Delta_h \xi}{\Delta_h \eta}_2.
	\end{align*}
We have the bound:
	\begin{align}
\left|\delta^2 E_h[\nu](\xi,\eta)\right| \le & \ \nrm{\xi}_2   \nrm{\eta}_2 + s\norm{\nabla_h^v \nu}{p}^{p-2}\norm{\nabla_h^v \xi}{p}\norm{\nabla_h^v \eta}{p}
	\nonumber
	\\
& \ +(p-2)s\norm{\nabla_h^v \nu}{p}^{p-2}\norm{\nabla_h^v \xi}{p}\norm{\nabla_h^v \eta}{p} +s\varepsilon^2\nrm{\Delta_h \xi}_2   \nrm{\Delta_h \eta}_2 .
	\label{eq:2ndvarboundE4-h}
	\end{align}

For this problem, we define the pre-conditioner via
	\begin{align*}
(\nu,\xi)_{\mathcal{L}_h} = \mathcal{L}_h[\nu](\xi):= & \ (\nu,\xi)_2 +s\eipew{D_x \nu}{D_x \xi} +s\eipns{D_y \nu}{D_y \xi} +s\varepsilon^2 (\Delta_h \nu, \Delta_h \xi)_2
	\nonumber
	\\
= & \ (\nu  -s \Delta_h \nu   +s\varepsilon^2 \Delta_h^2 \nu, \xi)_2,
	\end{align*}
for all $\nu,\xi \in {\mathcal C}_{\rm per}$, where we have used summation-by-parts to establish the second equality. In other words,
	\[
\mathcal{L}_h[\nu] = \nu  -s \Delta_h \nu   +s\varepsilon^2 \Delta_h^2 \nu.
	\]
One will notice the similarity of the pre-conditioner with the nonlinear operator in~\eqref{discrete-4th-order}. The induced norm is
	\[
\nrm{\nu}_{{\mathcal L}_h}^2 := (\nu,\nu)_{\mathcal{L}_h} =  \nrm{\nu}_2^2 +s \nrm{\nabla_h \nu}_2^2 + s\varepsilon^2 \nrm{\Delta_h \nu}_2,
	\]
defined for every $\nu\in {\mathcal C}_{\rm per}$.

Mimicking the proofs in the continuous case, using summation-by-parts in place of integration-by-parts, and Lemma~\ref{lem:dis-lemma3}, we get the following result, whose proof is omitted:

	\begin{lem}
	\label{lem:dis-lemma1}
For any $\nu,\xi \in \mathcal{C}_{\rm per}$,  
	\begin{equation}
C_5 \norm{\xi-\nu}{\mathcal{L}_h}^2 \le \left(\delta E_h[\xi] -\delta E_h[\nu]\right)(\xi -\nu)  , 
	\end{equation}
where $C_5 = \min \left( \frac12, \varepsilon s^{-\frac12} \right)$. Let $E_0$ be given, such that $B:= \left\{\nu \in  \mathcal{C}_{\rm per} \ \middle| \ E_h[\nu]\le E_0\right\}$ is non-empty. For any $\nu\in B$ and any $\xi\in \mathcal{C}_{\rm per}$, we have 
	\begin{equation}
\left| \delta^2 E_h[\nu](\xi,\xi) \right| \leq  C_6 \norm{\xi}{\mathcal{L}_h}^2 ,
	\end{equation}
where
	\begin{equation} 
C_6 = 1 +\frac{1}{p}\left(p-1\right)^{\frac{2p-1}{p}} \varepsilon^{\frac{-2(p-1)}{p}}s^{\frac{1}{p}} C_9^2  C_{10}^{p-2} ,
	\end{equation}
and $C_{10} = (p E_0)^{\frac{1}{p}}$. We can take $C_8 = C_5$ to satisfy  estimate \eqref{assmp-L6} of assumption (L6).
	\end{lem}			
				
\subsection{Convergence for the Discretized Sixth-Order Problem}
	\label{subsec-sixth-discrete}

The (second-order accurate) discrete version of \eqref{eq:6th-mixed-a} -- \eqref{eq:6th-mixed-b} can be expressed as follows: given $f, g \in \mathcal{C}_{\rm per}$, find $u ,w \in \mathcal{C}_{\rm per}$ such that
	\begin{align*}
u - \Delta_h w = & \ g,
	\\
s\lambda u -s \nabla_h^v \cdot \left( \left| \nabla_h^v u \right|^{p-2} \nabla_h^v u \right) + s\varepsilon^2 \Delta_h^2 u - w = & \ f.
	\end{align*}
As before, it is convenient to switch to the mean-zero version: find $u_\star ,w_\star \in \mathring{\mathcal C}_{\rm per}$ such that
	\begin{align*}
u_\star - \Delta_h w_\star = & \ g - \overline{g},
	\\
s\lambda u_\star -s \nabla_h^v \cdot \left( \left| \nabla_h^v u_\star \right|^{p-2} \nabla_h^v u_\star \right) + s\varepsilon^2 \Delta_h^2 u_\star - w_\star = & \ f -\overline{f} .
	\end{align*}

Similar to fourth-order regularized p-Laplacian problem, we define the following discrete energy: for every $\nu\in \mathring{\mathcal{C}}_{\rm per}$
	\begin{equation*}
	\label{eqn:diseng6}
E_h[\nu] :=  \frac{1}{2}\nrm{\nu-g+\bar{g}}_{-1}^2 + \frac{\lambda s}{2}\nrm{\nu+\bar{g}}_2^2 - \iprd{\nu}{f}{}+ \frac{s}{p} \norm{\nabla_h^v \nu}{p}^p +\frac{s\varepsilon^2}{2}\nrm{\Delta_h \nu}_{2}^2.  
	\end{equation*} 
For the discrete sixth order problem, we define a linear operator ${\mathcal L}_h :\mathring{\mathcal{C}}_{\rm per} \to \mathring{\mathcal{C}}_{\rm per}$ via
	\begin{align*}
(\nu,\xi)_{\mathcal{L}_h} = \mathcal{L}_h[\nu](\xi):= & \ s\lambda \iprd{\nu}{\xi}_2+\iprd{\nu}{\xi}_{-1} +s\eipew{D_x \nu}{D_x \xi} +s\eipns{D_y \nu}{D_y \xi} +s\varepsilon^2 (\Delta_h \nu, \Delta_h \xi)_2 
	\nonumber
	\\
= & \ \iprd{s\lambda \nu  -s \Delta_h \nu   +s\varepsilon^2 \Delta_h^2 \nu -\msfT_h\left[-\nu\right]}{\xi}_2,
	\end{align*}
where the second equality may be seen using summation-by-parts \cite{shen2012second,wise09a}. This operator satisfies (L1) -- (L3), and the next result, which we give without proof for the sake of brevity, shows that (L4) -- (L6) are satisfied as well.

	\begin{lem}
	\label{lem:dis-lemma6}
For any $\nu, \xi \in \mathring{\mathcal{C}}_{\rm per}$, the following inequality is valid  
	\begin{equation*}
C_5 \norm{\xi-\nu}{\mathcal{L}_h}^2 \le \left(\delta E_h[\xi] -\delta E_h[\nu]\right)(\xi-\nu) , 
	\label{dis-lem 6-0}
	\end{equation*}
where $C_5 = \min \left( \frac13, \varepsilon^{\frac43} s^{-\frac13} \right)$.
Let $E_{0}$  be given such that  $B:= \left\{\xi \in   \mathring{\mathcal{C}}_{\rm per} \ \middle| \ E_h[\xi]\le E_{0}\right\}$ is non-empty. For any $\nu \in B$, we have 
	\begin{equation*}
\left| \delta^2 E_h[\nu](\xi,\xi) \right| \leq  C_6 \norm{\xi}{\mathcal{L}_h}^2 ,
\label{dis-lem 7-0-2} 
	\end{equation*}
for all $\xi\in \mathring{\mathcal{C}}_{\rm per}$, where 
	\begin{equation*}
C_6= 1+(p-1)\left(\frac{3p}{2}\right)^{-\frac{2}{3p}} \left(\frac{3p}{3p-2}\right)^{\frac{2-3p}{3p}} \varepsilon^{\frac{4-6p}{3p}} s^{\frac{2}{3p}} C_9^2C_{10}^{p-2},
	\end{equation*}
and  $C_{10} = (p E_{h,0})^{\frac{1}{p}}$. We can take $C_8 = C_5$  to satisfy  estimate \eqref{assmp-L6} of assumption (L6).
	\end{lem}
		
	\section{Numerical Experiments}\label{sec:num}

In this section we perform some numerical experiments to support the theoretical results. The finite difference search direction equations and Poisson equations are solved efficiently using the Fast Fourier Transform (FFT).  We would like to point out that the Fourier pseudo-spectral method can be used to discretize space, and, once again, one can utilize the FFT for the inversion of the linear systems. For descriptions of the pseudo-spectral methods, see, for example,~\cite{boyd2001chebyshev, cheng2015fourier, hesthaven2007spectral}.

	\subsection{Thin Film Epitaxy Model with Slope Selection}

In this section we recall the convex splitting numerical scheme in \cite{wang2010unconditionally} for the  thin film  epitaxy model with slope selection. Suppose that $\Omega\subset\mathbb{R}^2$ is a rectangular domain.  The energy of an epitaxial thin film  is given by 
	\begin{equation*}
\mathcal{E}[ u ] = \int_\Omega\left\{\frac{1}{p}  \left| \nabla  u   \right|^p   - \frac12 \left| \nabla  u   \right|^2 
+ \frac{\varepsilon^2}{2} \left| \Delta  u  \right|^2 \right\}d{\bf x},  \quad \forall \ u\in H^2_{\rm per}(\Omega),
	\end{equation*} 
where, $p\ge 4$ is even,  $u :\Omega\rightarrow \mathbb{R}$ is the height film, and $\varepsilon$ is a constant.  The $L^2$ gradient flow is
	\begin{equation}
	\label{dyn-non-conserve}
\partial_t u  = - w  , \quad w  := \delta \mathcal{E} =  - \nabla \cdot 
\left( \left| \nabla  u  \right|^{p-2} \nabla  u  \right) 
+  \Delta  u  + \varepsilon^2 \Delta^2  u ,
	\end{equation}
and $w$ is called the chemical potential. The model predicts the emergence of a faceted thin film, whose facets have slopes of magnitude approximately  one, that coarsens over time. The fully-implicit convex splitting scheme in 2D~\cite{wang2010unconditionally} can be written in operator format as $\mathcal{N}_h[u^{n+1}]=f$, where 
	\begin{equation}
\mathcal{N}_h[\nu] := \nu - s \nabla_h^v \cdot \left( \left| \nabla_h^v  \nu \right|^{p-2} \nabla_h^v  \nu  \right) + \varepsilon^2 s \Delta_h^2  \nu , \quad f= u^n-s{\Delta}^v_h  u^n,
	\label{thin-film-fully-discrete}
	\end{equation}
and $s>0$ is the time step. Hence, the scheme can be reformulated as the fourth-order  problem \eqref{discrete-4th-order} with $f= u ^n-s \Delta^v_h  u^{n}$ and $p\ge 4$  and even. 

In way of summary, to solve $\mathcal{N}_h[u]=f$, suppose that iterate $u^k\in\mathcal{C}_{\rm per}$ is given. (Note that $k$ is the PSD solver iteration index, not the time step index, the latter of which we usually denote by $n$.) We first compute the search direction $d^k\in\mathcal{C}_{\rm per}$ via \eqref{search-direction-pb}: 
	\begin{align*}
\mathcal{L}_{h}[d^k] = d^k  -s \Delta_h d^k   +s\varepsilon^2 \Delta_h^2 d^k = & \ -\delta E_h[u^k]
	\\
= & \ -\left(u^k-f - s\nabla_h^v \cdot \left( \left| \nabla_h^v u^k\right|^{p-2} \nabla_h^v u^k \right) + s\varepsilon^2\Delta_h^2 u^k\right)
	\\
= & \ f- \mathcal{N}_h[u^k] ,
	\end{align*}
where $E_h$ is as defined in \eqref{eqn:diseng4}. This equation is efficiently solved using FFT.  Once $d^k$ is found, we perform a line-search according to \eqref{eqn-search}: find $\alpha_k\in\mathbb{R}$ such that $q(\alpha_k) = 0$, where
	\begin{align*}
q(\alpha) := & \ \delta E_h[u^k + \alpha d^k](d^k)
	\\
= & \  \left(u^k + \alpha d^k -f - s\nabla_h^v \cdot \left( \left| \nabla_h^v (u^k + \alpha d^k)\right|^{p-2} \nabla_h^v (u^k + \alpha d^k) \right) + s\varepsilon^2\Delta_h^2 (u^k + \alpha d^k), d^k \right)_2
	\\
= & \ \iprd{\mathcal{N}_h[u^k + \alpha d^k] - f}{d^k}_2.
	\end{align*}
The approximation sequence is then updated via $u^{k+1} = u^k+\alpha_kd^k$. When  $p=4$ ($p=6$), a short calculation shows that $q$ is a cubic (quintic) polynomial whose coefficients can be easily obtained. Moreover, the theory predicts that there is a unique global root for $q$.

	\subsubsection{Convergence and complexity of the PSD solver}
	
In this subsection we demonstrate the accuracy and efficiency of the PSD solver by using the epitaxial thin film model with slope selection. We present the results of some convergence tests and perform some sample computations to demonstrate the convergence and near optimal complexity with respect to the grid size $h$.

\begin{table}[!htb]
	\begin{center}
		\caption{Errors, convergence rates, average iteration numbers and average CPU time for each time step. Parameters are given in the
		text, and the initial data are defined in \eqref{eqn:init1}. The refinement path is $s=0.1h^2$. } \label{tab:cov}
		\begin{tabular}{cccccccccc}
			\hline &&\multicolumn{3}{c}{$p=4$}&\multicolumn{5}{c}{$p=6$}\\
			\hline $h_c$&$h_{f}$&$\nrm{\delta_u}_{2}$ & Rate&$\#_{iter}$ &$T_{cpu}(h_f)$& $\nrm{\delta_u}_{2}$ & Rate&$\#_{iter}$ &$T_{cpu}(h_f)$\\
			\hline $\frac{3.2}{16}$&$\frac{3.2}{32}$& $6.2192\times 10^{-3}$&- & 4  & 0.0007&
			$ 9.3074\times 10^{-3}$&- & 5& 0.0009
			\\$\frac{3.2}{32}$&$\frac{3.2}{64}$& $1.2685\times 10^{-3}$ &2.29& 2 &0.0024 & 
			$ 1.6392\times 10^{-3}$ &2.51 &3 &0.0032
			\\ $\frac{3.2}{64}$ &$\frac{3.2}{128}$& $2.6046\times 10^{-4}$&2.28 & 2 &0.0114& 
			$2.9046\times 10^{-4}$&2.50 & 2 & 0.0141
			\\ $\frac{3.2}{128}$ & $\frac{3.2}{256}$ &$5.9639\times 10^{-5}$&2.13 &2 &0.0475&
			$6.5325\times 10^{-5}$&2.15 & 2 & 0.0616
			\\ $\frac{3.2}{256}$ & $\frac{3.2}{512}$&$1.4526\times 10^{-5}$&2.04&2 &0.3560&
			$1.5886\times 10^{-5}$&2.04 & 2& 0.4636\\
			\hline
		\end{tabular}
	\end{center}
\end{table}

To simultaneously demonstrate the spatial accuracy and the efficiency of the solver, we perform a typical time-space convergence test for the fully discrete scheme \eqref{thin-film-fully-discrete} for the slope selection model.
As in \cite{shen2012second, wang2010unconditionally},  we perform the Cauchy-type convergence test using the following periodic initial data~\cite{shen2012second}:
	\begin{eqnarray}
	\label{eqn:init1}
 u (x,y,0)&=&0.1\sin^2\left(\frac{2\pi x}{L}\right)\cdot \sin\left(\frac{4\pi (y-1.4)}{L}\right)
 	\nonumber
	\\
&& - 0.1\cos\left(\frac{2\pi (x-2.0)}{L}\right)\cdot\sin\left(\frac{2\pi y}{L}\right),
	\end{eqnarray}
where $\Omega = (0,3.2)^2$. In this test, we compute the Cauchy difference, $\delta_ u : = u _{h_f}(T)-\mathcal{I}_c^f( u _{h_c}(T))$, where $h_c = 2h_f$, and $\mathcal{I}_c^f$ is a bilinear interpolation operator that maps the coarse grid  approximation $u _{h_c}$ onto the fine grid. We take a quadratic refinement path, \emph{i.e.},  $s=h^2/10$, to equalize the spatial and temporal error contributions.  At the final time, $T=0.32$, we expect the global error to be $\mathcal{O}(s)+\mathcal{O}(h^2)=\mathcal{O}(h^2)$ in the discrete $\|\cdot\|_2$ and $\|\cdot\|_\infty$ norms, as $h, s\to 0$. The other parameter is given by  $\varepsilon=0.1$. The norms of Cauchy difference, the convergence rates, average iteration number and average CPU time can be found on Table \ref{tab:cov}. Second-order convergence is observed. At the same time, the average iteration count for the solver remains at around 2. Since we are using a quadratic refinement path, increasing the grid size by a factor of two (decreasing the grid spacing by 2) means increasing the number of time-space degrees of freedom by a factor of 16. But the CPU time increases at a much slower rate.  The complexity can be offset, of course, by the fact the starting guesses for the solver at each independent time level are better for smaller time step sizes.

	\begin{figure}[h]
	\begin{center}
	\begin{subfigure}{0.48\textwidth}
\includegraphics[width=\textwidth]{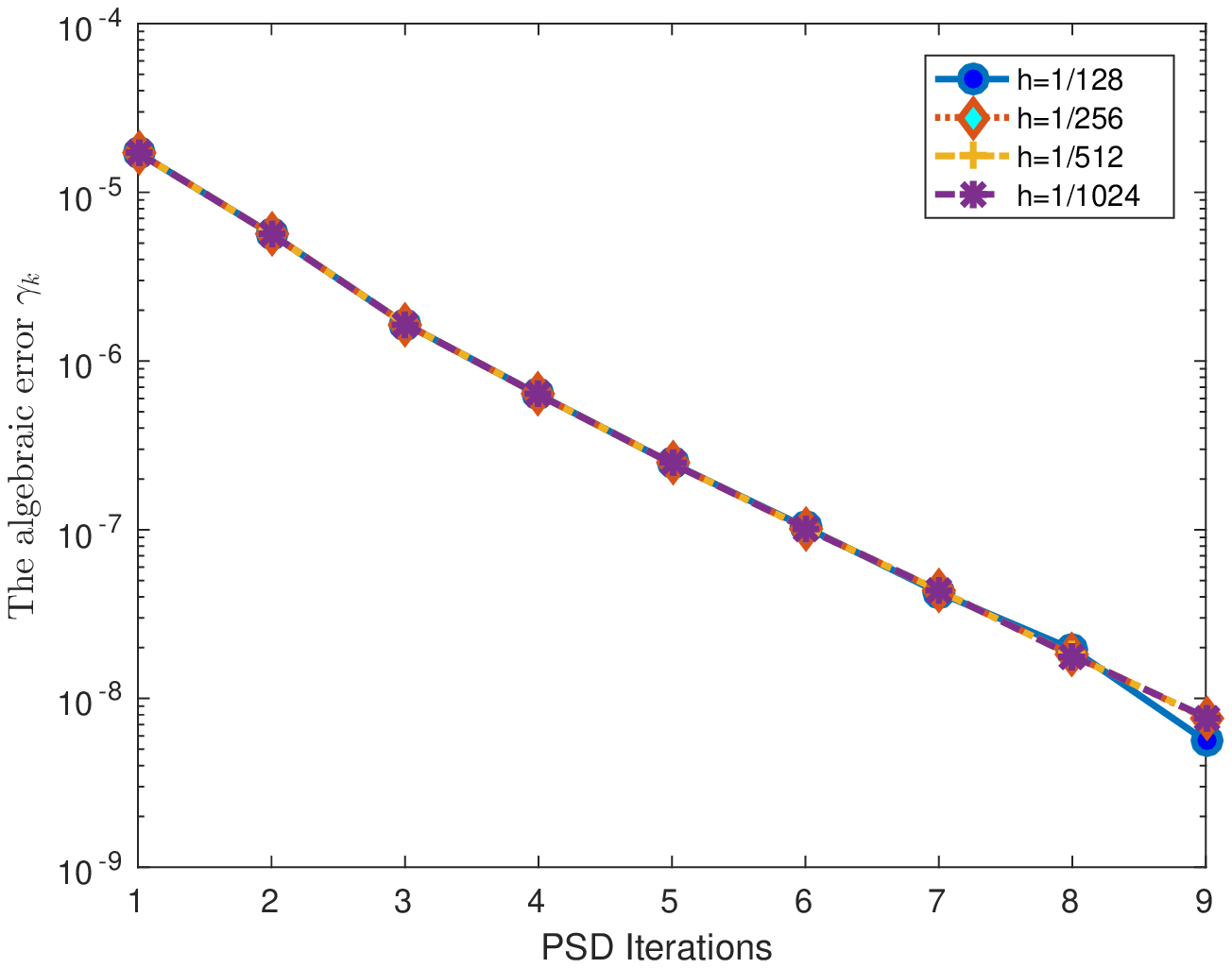} 
\caption{$h$-independence: $p=4$, $s=0.01$ and $\varepsilon=0.03$.}
	\end{subfigure}
	\begin{subfigure}{0.48\textwidth}
\includegraphics[width=\textwidth]{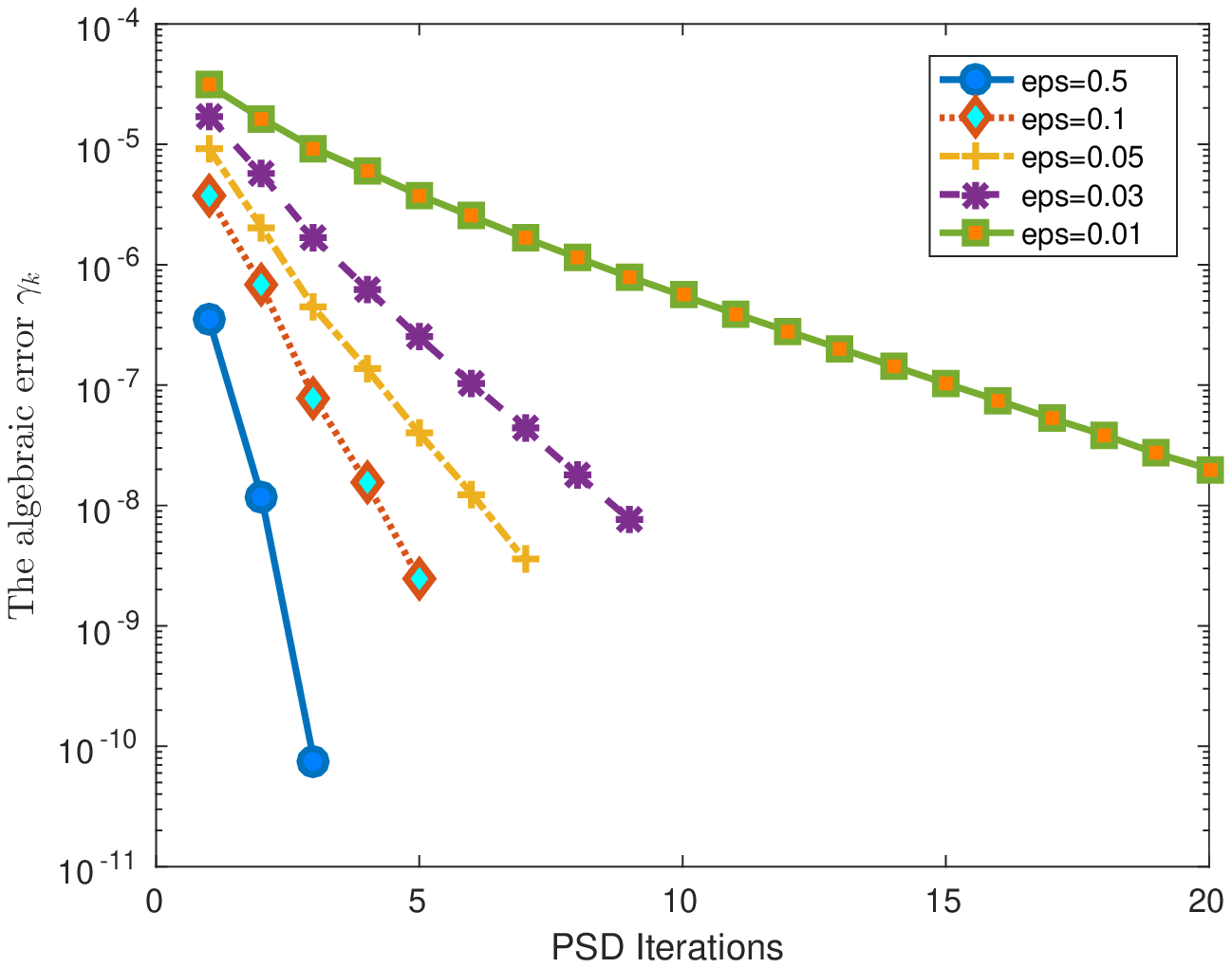}
\caption{$\varepsilon$-dependence:  $p=4$, $s=0.01$ and $h=1/512$.}
	\end{subfigure}
		
	\begin{subfigure}{0.48\textwidth}
\includegraphics[width=\textwidth]{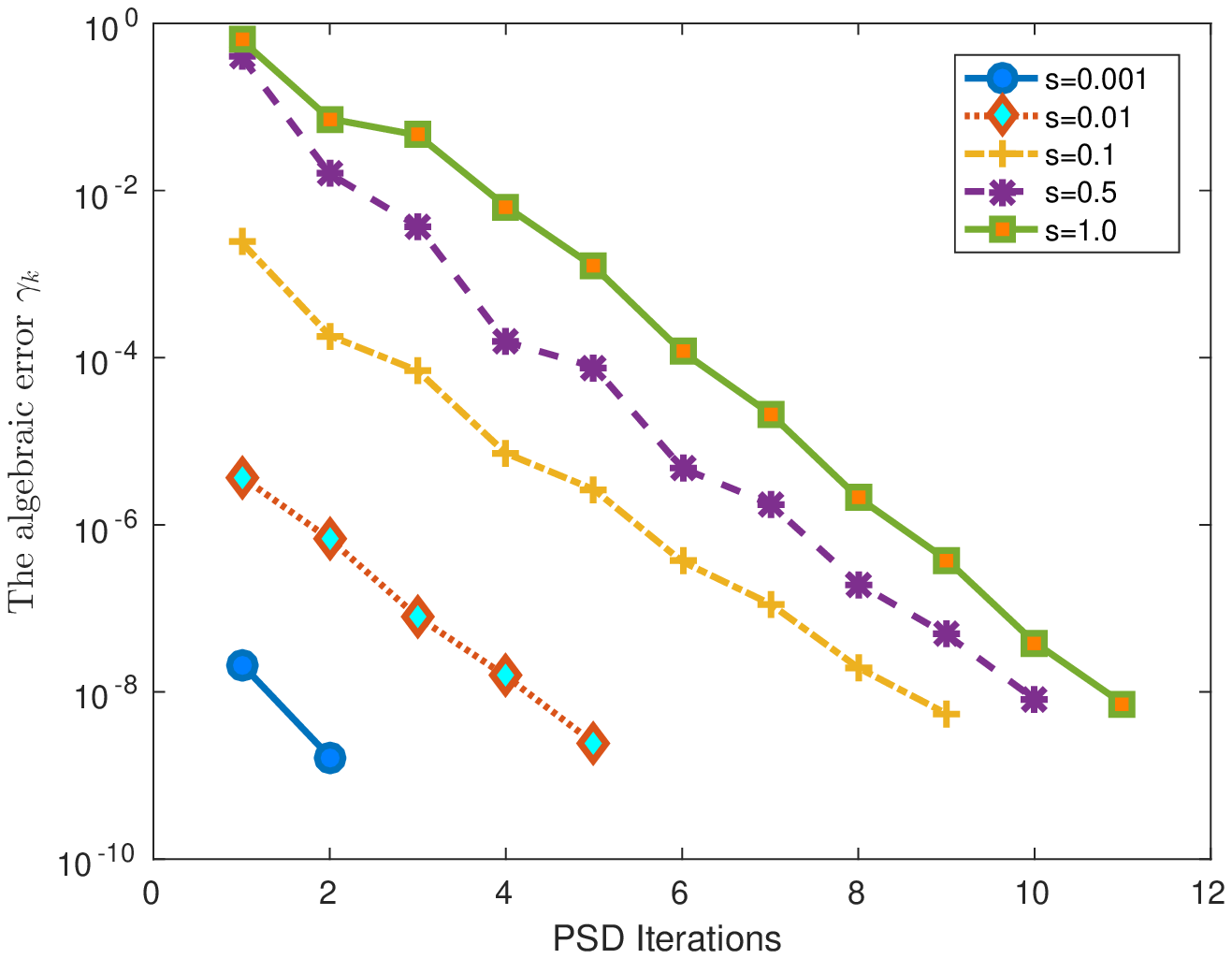} 
\caption{$s$-dependence: $p=4$, $h=\nicefrac{1}{512}$ and $\varepsilon=0.03$.}
	\end{subfigure}
	\begin{subfigure}{0.48\textwidth}
\includegraphics[width=\textwidth]{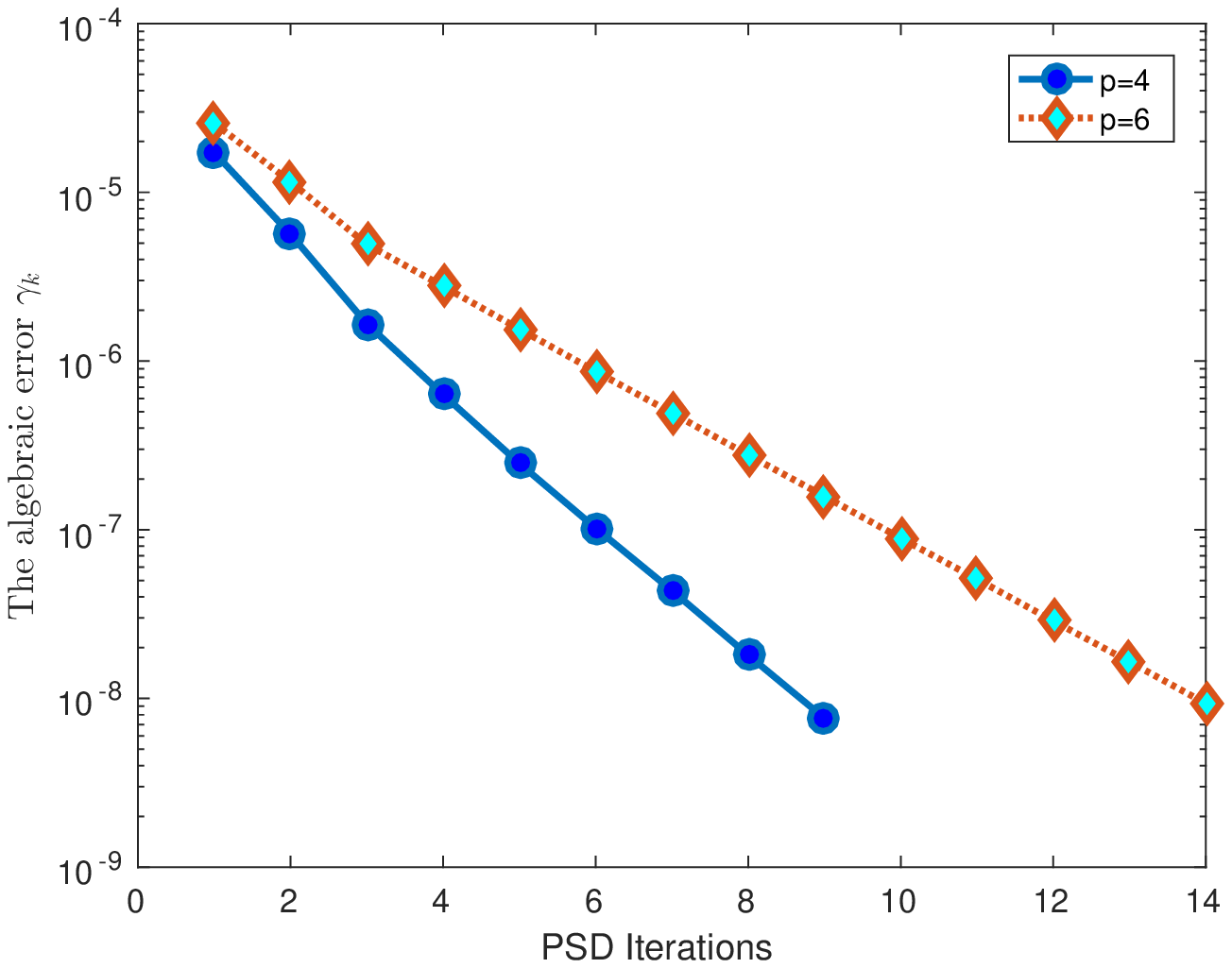}
\caption{$p$-dependence: $h=\nicefrac{1}{512}$,  $s=0.01$ and $\varepsilon=0.03$.}
	\end{subfigure}

\caption{Complexity tests showing the solver performance for changing values of $h$, $\varepsilon$, $s$ and $p$. Parameters are given in the text.}
	\label{fig:complexity}
	\end{center}
	\end{figure}

To more directly investigate the complexity of the PSD solver we perform another series of tests to determine the dependences of the convergence rates on $\varepsilon$, $h$, $s$, and $p$, in particular. Consider the following spatially periodic function parametrized by $s$: 
	\begin{eqnarray}
	\label{eqn:exact}
\tilde{u}(x,y,s) = \frac{1}{2\pi}\sin\big({2\pi x}\big)\cos\big({2\pi y}\big)\cos(s).
	\end{eqnarray}
First we calculate $f:=\mathcal{N}_h\left[ \mathcal{I}_h \left(\tilde{u}(\, \cdot\, ,\,  \cdot \, , s )\right) \right] \in \mathcal{C}_{\rm per}$, where $\mathcal{I}_h: {C}^0_{\rm per}(\Omega)\to \mathcal{C}_{\rm per}$ is the canonical grid projection operator. Then we compute the sequence $\left\{u^k\right\}_{k=0}^\infty$ via the PSD algorithm, with the initialization
	\[
u^0_{i,j}  = \tilde{u}(p_i,p_j ,0)+ s^2\sin\big({4\pi p_i}\big)\sin\big({6\pi p_j}\big),
	\]
hence $u^k \to \mathcal{I}_h \left(\tilde{u}(\, \cdot\, ,\,  \cdot \, , s )\right)$, as $k\to \infty$. Define $\gamma_k:=\| u^k -\mathcal{I}_h \left(\tilde{u}(\, \cdot\, ,\,  \cdot \, , s )\right) \|_\infty$. We stop the PSD algorithm when $\gamma_k \le \tau := 1\times 10^{-8}$.

In Figure~\ref{fig:complexity} we plot $\gamma_k$ versus $k$, on a semi-log scale, for various choices of $h$, $\varepsilon$, $s$ and $p$. In Figure~\ref{fig:complexity}(a) $p=4$, $s=0.01$ and $\varepsilon=0.03$; in Figure \ref{fig:complexity}(b) $p=4$, $s=0.01$ and $h=1/512$; in Figure \ref{fig:complexity}(c) $p=4$, $h=1/512$ and $\varepsilon=0.03$; in Figure \ref{fig:complexity}(d): $h=1/512$, $s=0.01$ and $\varepsilon=0.03$. As can be seen in Figure \ref{fig:complexity}(a),  the convergence rate (as gleaned from the error reduction) is nearly uniform and nearly independent of $h$. Figures \ref{fig:complexity} (b) and (c) indicate that more PSD iterations are required for smaller values of $\varepsilon$ and larger values of $s$, respectively. Figure \ref{fig:complexity}(d) shows that the number of PSD iterations increases with the value of $p$. These general trends are expected form the theory.

\begin{figure}[h]
	\begin{center}
		\begin{subfigure}{0.45\textwidth}
			\includegraphics[width=\textwidth]{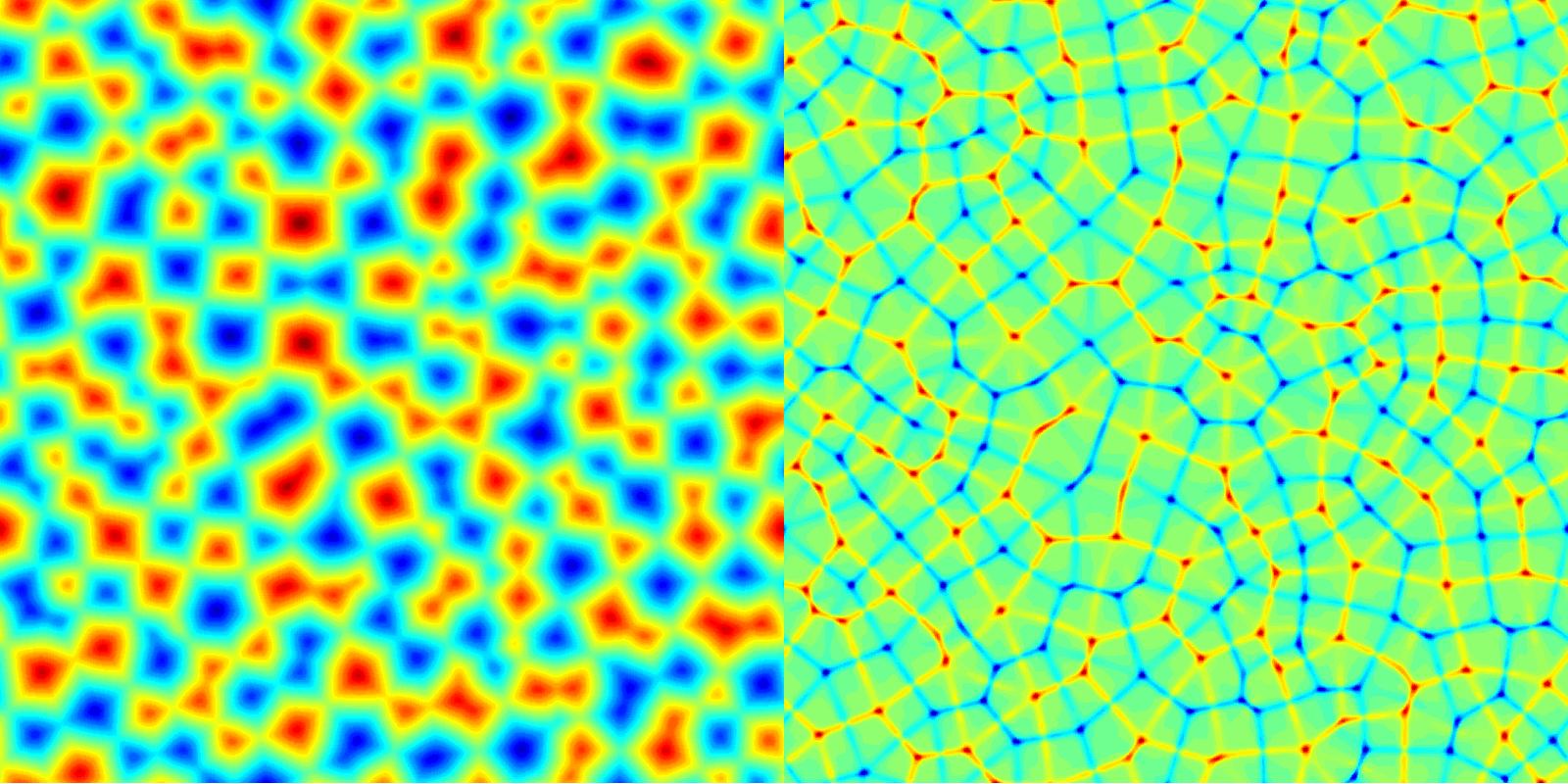} 
			\caption*{$t=10$}
		\end{subfigure}
		\begin{subfigure}{0.45\textwidth}
			\includegraphics[width=\textwidth]{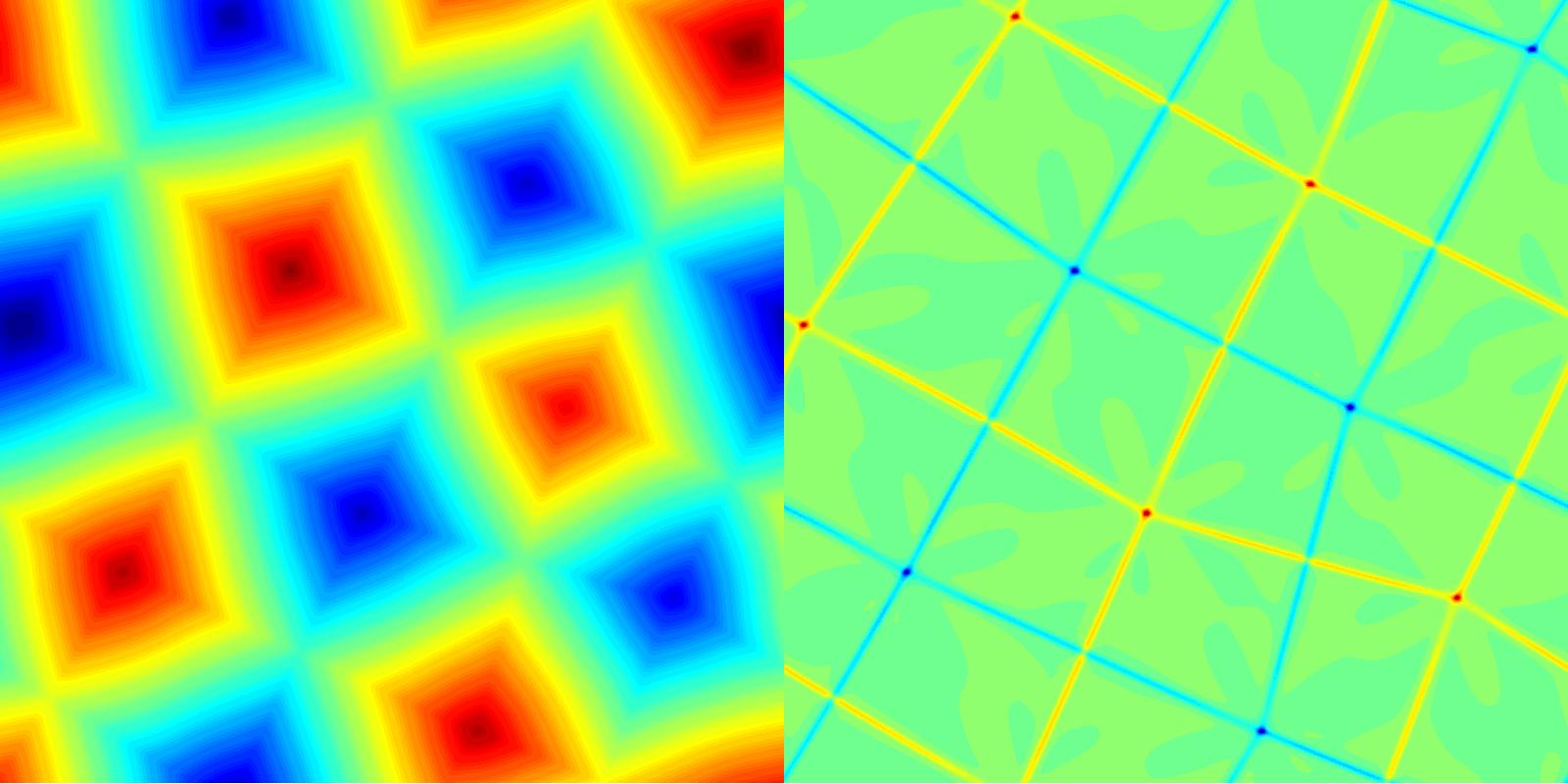}
			\caption*{$t=1000$}
		\end{subfigure}
		\begin{subfigure}{0.45\textwidth}
			\includegraphics[width=\textwidth]{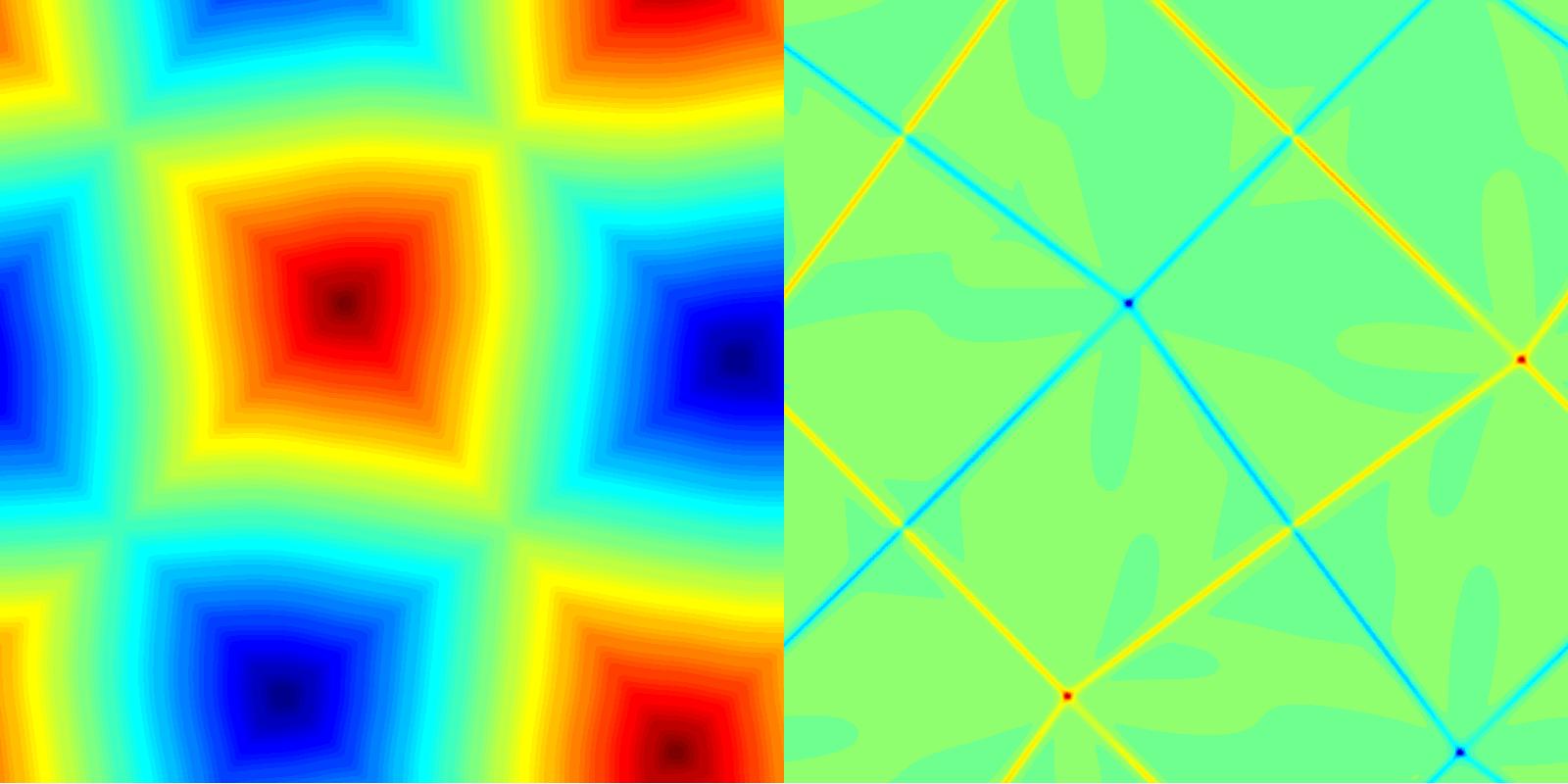} 
			\caption*{$t=3000$}
		\end{subfigure}
		\begin{subfigure}{0.45\textwidth}
			\includegraphics[width=\textwidth]{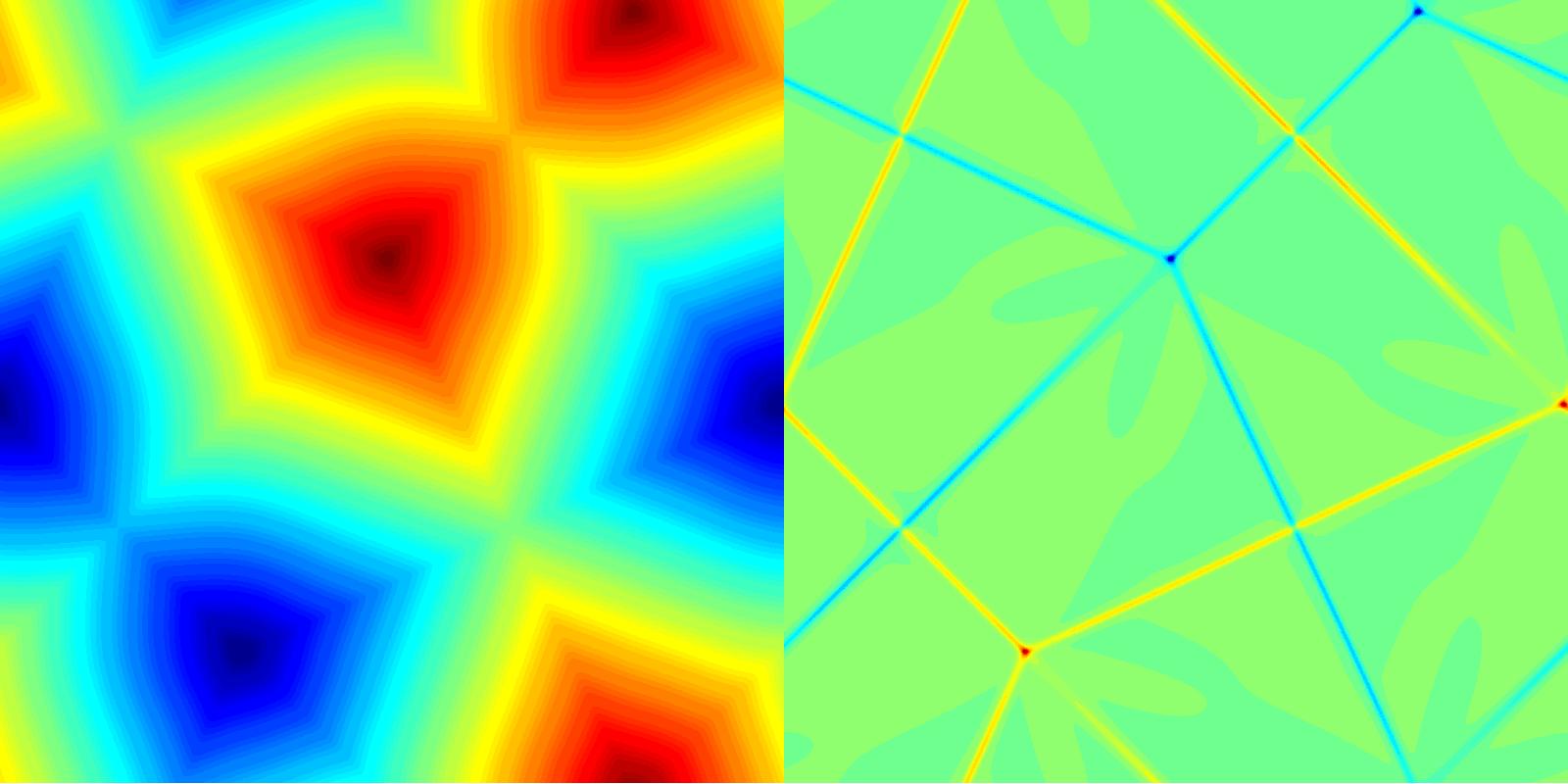}
			\caption*{$t=6000$}
		\end{subfigure}
		\begin{subfigure}{0.45\textwidth}
			\includegraphics[width=\textwidth]{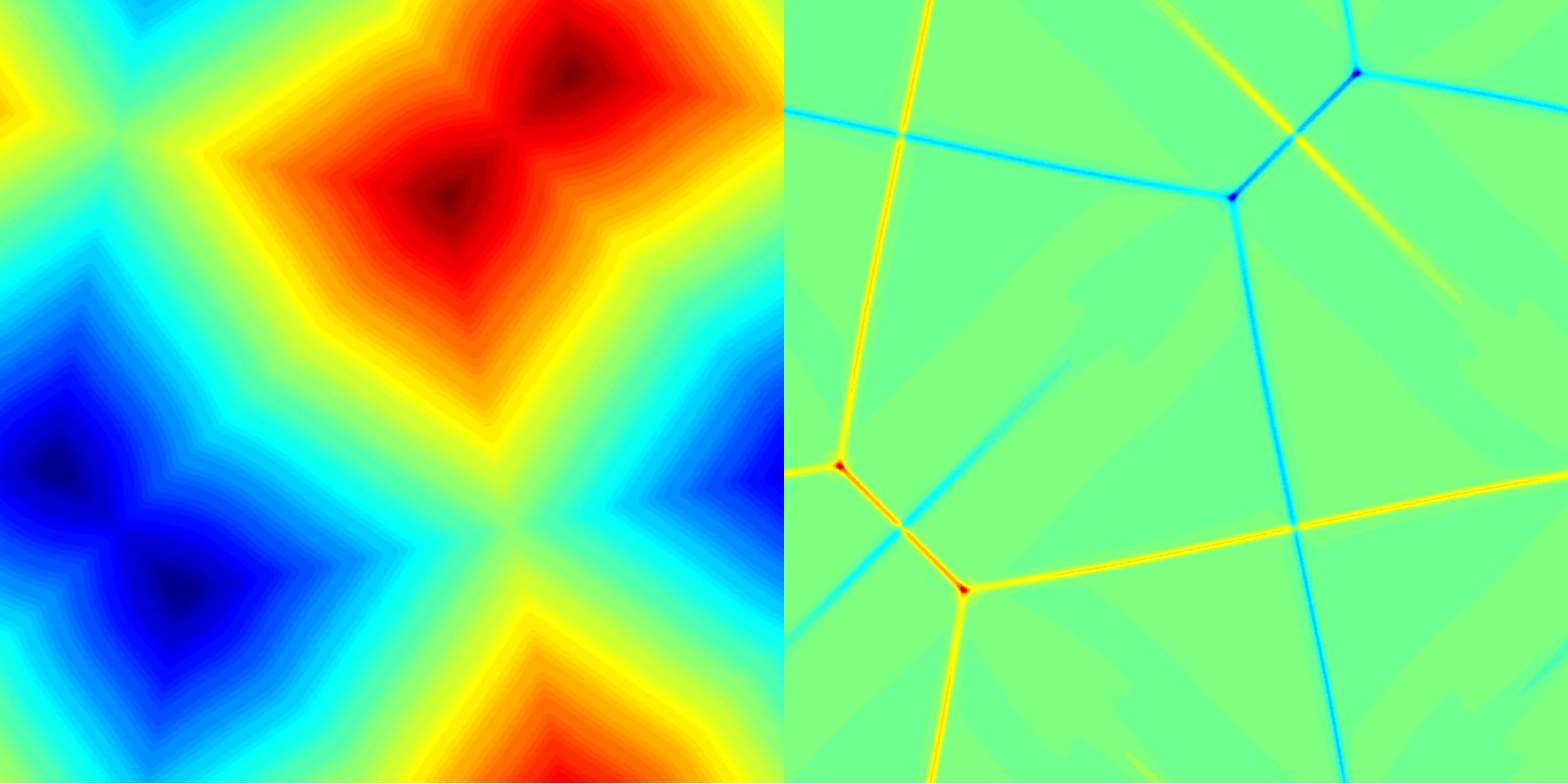} 
			\caption*{$t=8000$}
		\end{subfigure}
		\begin{subfigure}{0.45\textwidth}
			\includegraphics[width=\textwidth]{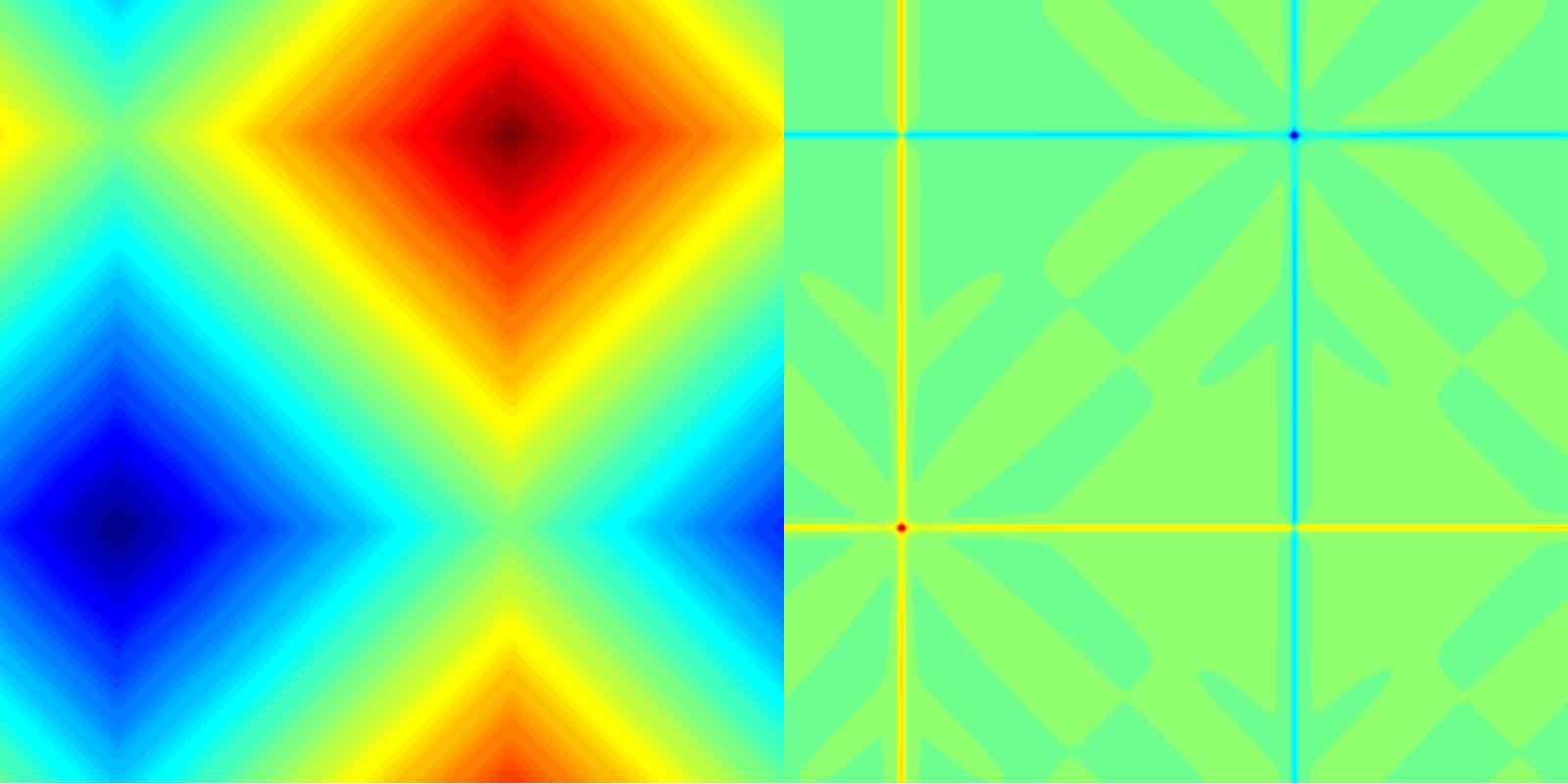}
			\caption*{$t=10000$}
		\end{subfigure}
		\caption{Time snapshots of the evolution with PSD solver for the epitaxial thin film growth model with $p=4$ at $t=10,1000, 3000, 6000, 8000~ \text{and}~ 10000$. Left: contour plot of $ u $, Right: contour plot of $\Delta u $. The parameters are
			$\varepsilon = 0.03, \Omega=[12.8]^2, s=0.01$.
			These simulation results are consistent with earlier work on this topic in \cite{shen2012second, wang2010unconditionally,xu2006stability}.}
		\label{fig:long-time-psd-l2}
	\end{center}
\end{figure}

\begin{figure}[ht]
	\begin{center}
		\begin{subfigure}{0.45\textwidth}
			\includegraphics[width=\textwidth]{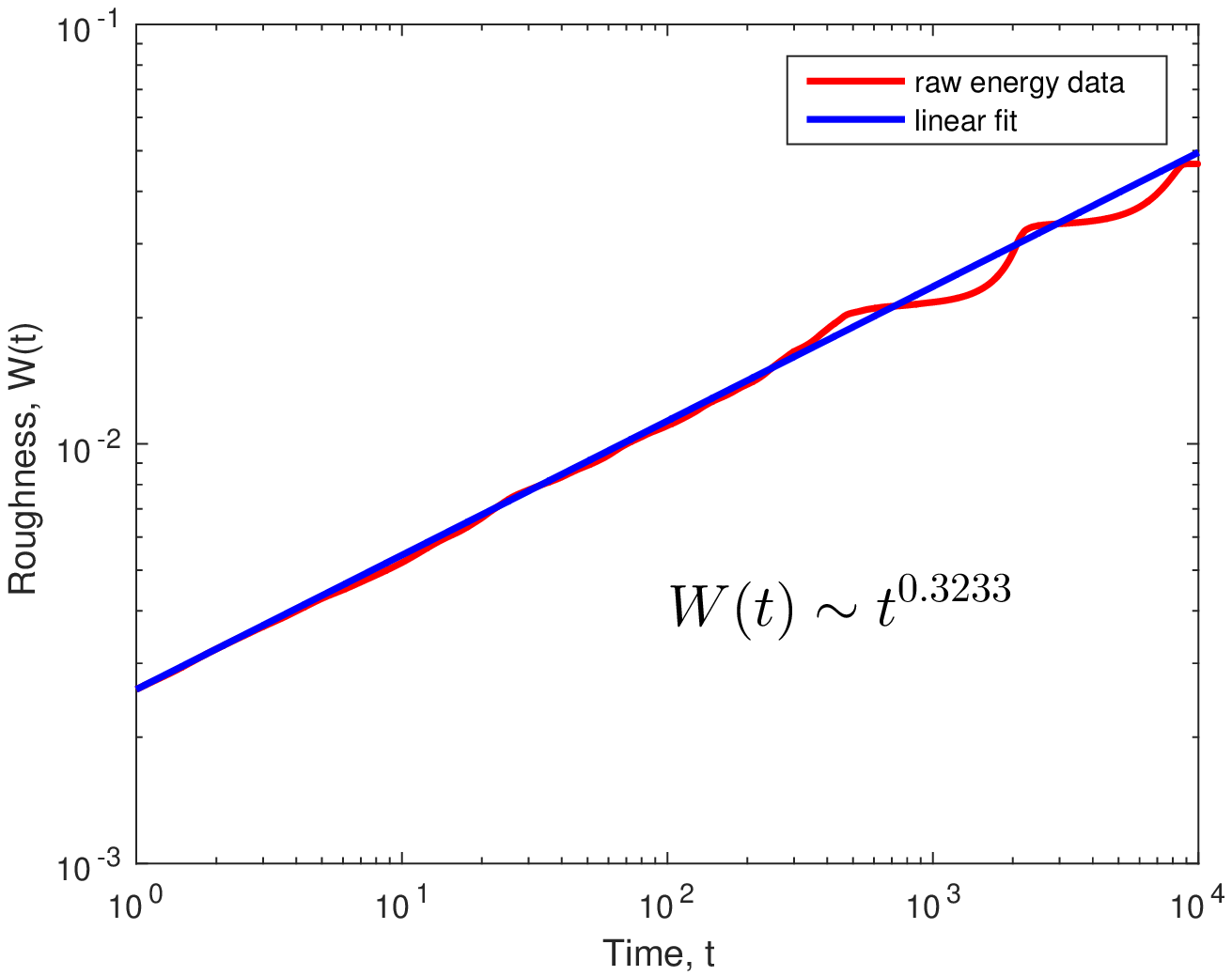} 
			\caption{Roughness evolution}
		\end{subfigure}
		\begin{subfigure}{0.45\textwidth}
			\includegraphics[width=\textwidth]{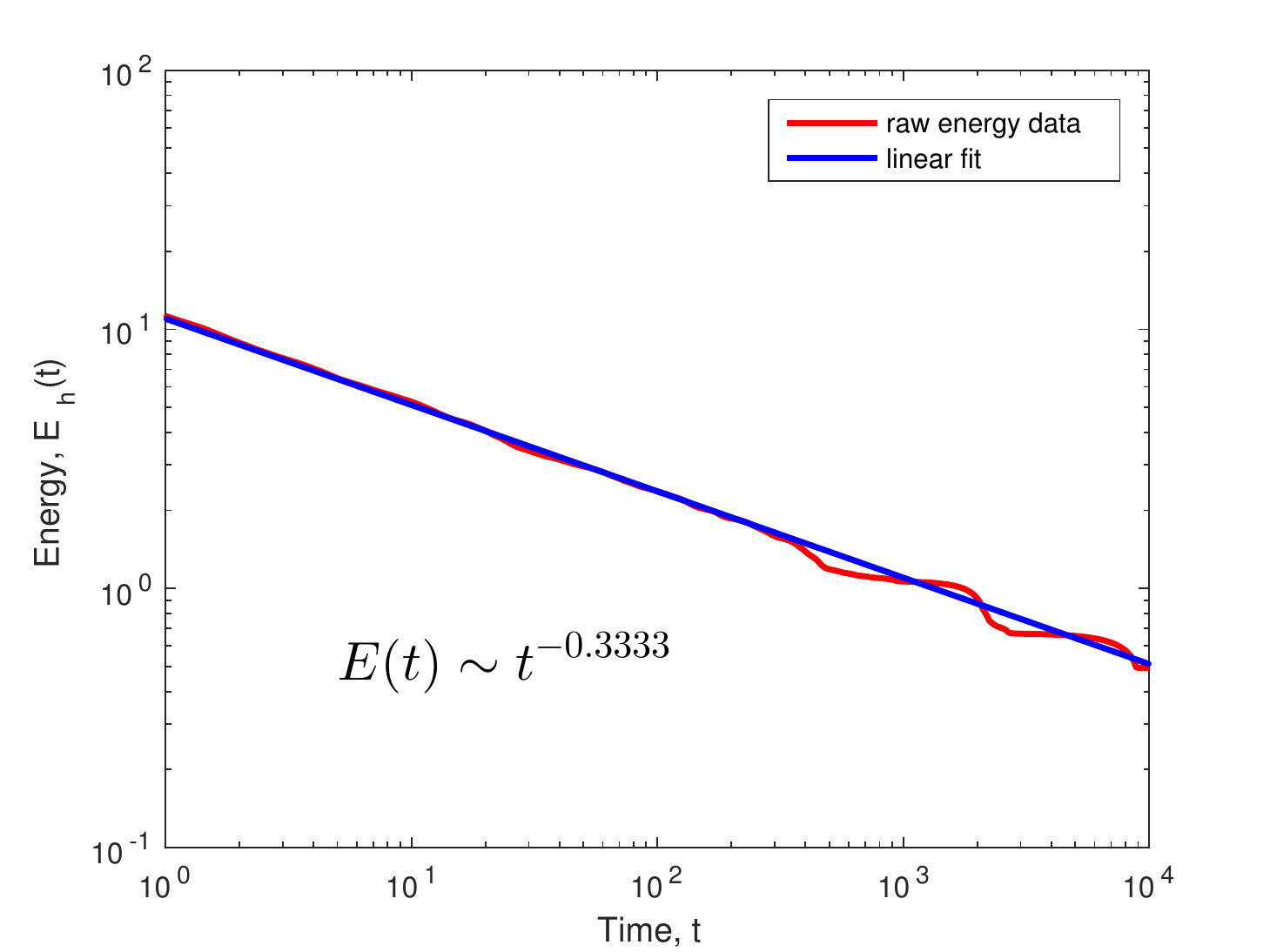}
			\caption{Energy evolution}
		\end{subfigure}
	\end{center}
	\caption{Log-log plot of Roughness and energy evolution for the simulation depicted in Figure~\ref{fig:long-time-psd-l2}. }
	\label{fig:one-third-psd-l2}
\end{figure}

\begin{figure}[h]
	\begin{center}
		\begin{subfigure}{0.45\textwidth}
			\includegraphics[width=\textwidth]{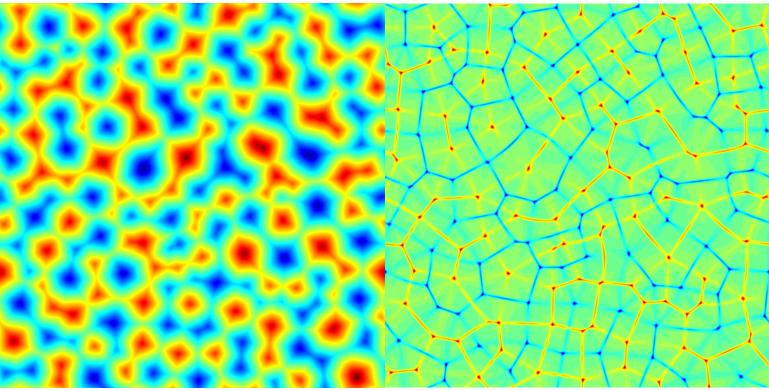} 
			\caption*{$t=10$}
		\end{subfigure}
		\begin{subfigure}{0.45\textwidth}
			\includegraphics[width=\textwidth]{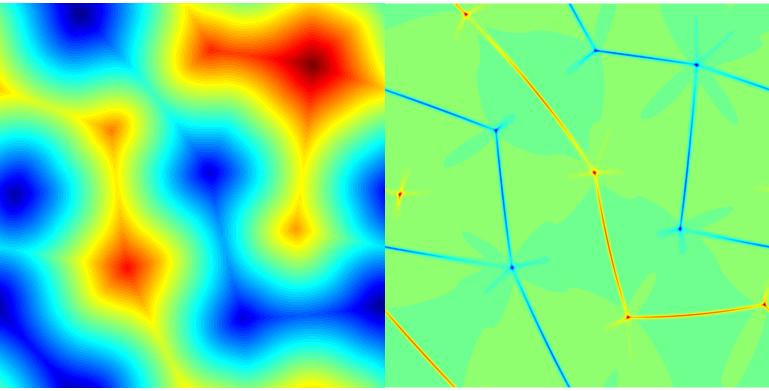}
			\caption*{$t=1000$}
		\end{subfigure}
		\begin{subfigure}{0.45\textwidth}
			\includegraphics[width=\textwidth]{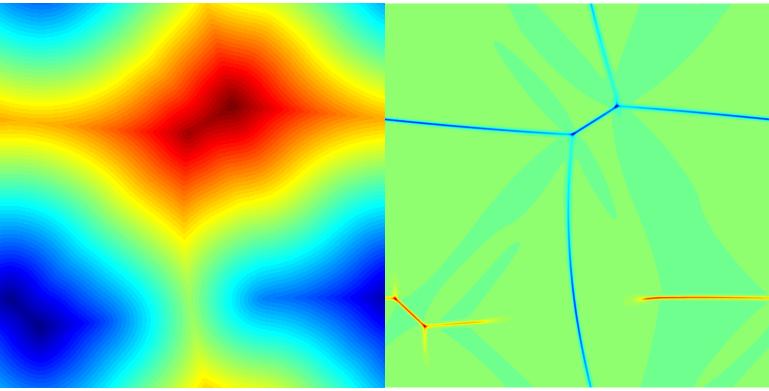} 
			\caption*{$t=3000$}
		\end{subfigure}
		\begin{subfigure}{0.45\textwidth}
			\includegraphics[width=\textwidth]{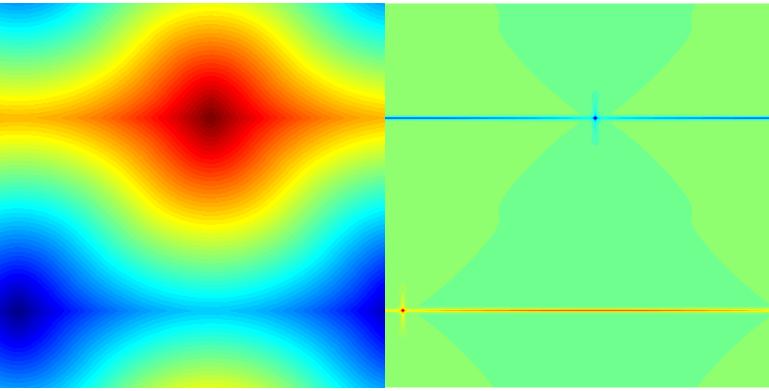}
			\caption*{$t=6000$}
		\end{subfigure}
		\caption{Time snapshots of the evolution with PSD solver for the epitaxial thin film growth model with $p=6$ at $t=10,1000, 3000~ \text{and}~ 6000$. Left: contour plot of $ u $, Right: contour plot of $\Delta u $. The parameters are
			$\epsilon = 3.0\times 10^{-2}, \Omega=[12.8]^2, s=0.01$.}
		\label{fig:long-time-psd-6p}
	\end{center}
\end{figure}

	\subsubsection{Long-time coarsening behavior for the thin film model with $p=4,6$}

Coarsening processes in thin film systems can take place on very long time scales~\cite{kohn2006energy}.  In this subsection, we perform (now standard) long time behavior tests for $p=4, 6$.   Such test, which have been performed in many places, will  confirm the expected coarsening rates and serve as benchmarks for our solver. See, for example, \cite{shen2012second, wang2010unconditionally}. The initial data for the simulations are taken as essentially random:
	\begin{equation}
	\label{eqn:init2}
u^0_{i,j}=0.05\cdot(2r_{i,j}-1),
	\end{equation}
where the $r_{i,j}$ are uniformly distributed random numbers in [0, 1]. Time snapshots of the evolution for the epitaxial thin film growth model with $p=4$ can be found in Figure \ref{fig:long-time-psd-l2}. The coarsening rates for the $p=4$ case are given in Figure~\ref{fig:one-third-psd-l2}. These simulation results are consistent with earlier work on this topic in \cite{shen2012second, wang2010unconditionally,xu2006stability}, showing the surface roughness, $W$, grows like $t^{\frac{1}{3}}$ and the energy, $E$, decays like $t^{-\frac{1}{3}}$. We also present the numerical simulations for the epitaxial thin film growth model with $p=6$ in Figure \ref{fig:long-time-psd-6p}. Notice in Figure \ref{fig:long-time-psd-6p} that the evolution process is significantly different from the process depicted in Figure \ref{fig:long-time-psd-l2}.

	\subsection{Square Phase Field Crystal Model}

Suppose that $\Omega\subset \mathbb{R}^d$, $d=2,3$ is a rectangular domain. The energy of square phase field crystal (SPFC) model is given by \cite{elder04,golovin03,medina2014formation,lloyd2008localized}:
	\begin{equation*}
	\label{energy-spfc}
\mathcal{E}[ u ] = \int_\Omega\left\{\frac{\gamma_0 }{2}   u ^2 
- \frac{\gamma_1}2 \left| \nabla  u   \right|^2 
+ \frac{\varepsilon^2}{2} \left| \Delta  u  \right|^2 + \frac{1}{4}  \left| \nabla  u   \right|^4\right\}d{\bf x}, 
	\end{equation*} 
where $u :\Omega\rightarrow \mathbb{R}$ corresponds to the number density  field of the atoms, and $\varepsilon >0$, $\gamma_0,\gamma_1\ge 0$ are parameters. The SPFC model is the  $H^{-1}$ gradient flow of this energy and is given by 
	\begin{eqnarray*}
\partial_t  u  = \Delta  w  ,  \quad   w  := \delta \mathcal{E} =  \gamma_0   u + \gamma_1\Delta u  + \varepsilon^2 \Delta^2  u  - \nabla \cdot \left( \left| \nabla  u  \right|^2 \nabla  u  \right) . 
	\end{eqnarray*}
We propose the following fully-implicit, nonlinear convex-splitting scheme
	\begin{equation}\label{scheme-spfc}
u^{n+1}-\Delta_h  w^{n+1} =  g  ,\quad  s\gamma_0   u^{n+1}  - s\nabla_h^v \cdot \left( \left| \nabla_h^v  u ^{n+1} \right|^2 \nabla_h^v  u^{n+1} \right) + s\varepsilon^2 \Delta_h^2  u^{n+1} - w^{n+1} = f ,
	\end{equation}
where $g = u^n$ and $f =  -\gamma_1\Delta_h u^n$.  Using the techniques of~\cite{wang2010unconditionally,wise09a}, we can prove that this scheme is unconditionally energy stable. The fully discrete scheme can also be rewritten in operator format as  $\mathcal{N}_h[u^{n+1}] = f$, where
	\begin{equation*}
\mathcal{N}_h[\nu] := s\gamma_0 \nu  + s \varepsilon^2 \Delta_h^2 \nu  - s\nabla_h^v \cdot \left( \left| \nabla_h^v \nu  \right|^2 \nabla_h^v \nu  \right) - T_h[-\nu + g].
	\end{equation*}
We can shift the scheme from the affine space of solutions -- whose elements $\nu$ satisfy $\iprd{\nu-\overline{g}}{ 1}_2 = 0$ -- to the mean zero space, but this is not necessary for practical implementation.  Otherwise, this scheme is in the scope of our theory, and, according to the prescription in Section~\ref{subsec-sixth-discrete},  the pre-conditioner should be
	\begin{equation*}
\mathcal{L}_h[\nu] := s\gamma_0 \nu -s\Delta_h \nu + s \varepsilon^2 \Delta_h^2 \nu  - T_h[-\nu].
	\end{equation*}
Given $u^k\in \mathcal{C}_{\rm per}$, with $\iprd{u^k-\overline{g}}{1}_2 = 0$, we compute the search direction $d^k\in\mathring{\mathcal{C}}_{\rm per}$ by solving the sixth order linear problem $\mathcal{L}_{h}[d^k] = f- \mathcal{N}_h[u^k]$ using FFT.  Once $d^k$ is found, we perform the line-search: find $\alpha_k\in\mathbb{R}$ such that $q(\alpha_k) = 0$, where
	\[
q(\alpha) =   \iprd{\mathcal{N}_h[u^k + \alpha d^k] - f}{d^k}_2.
	\]
After this, we update the approximation via $u^{k+1} = u^k+\alpha_k d^k$. As before, $q$ is a cubic polynomial (since $p=4$) whose coefficients can be precomputed. But this time, two of the coefficients involve the $\mathsf{T}_h = -\Delta_h^{-1}$ operator. Specifically, for $q(\alpha)$ we need to compute
	\begin{align*}
\iprd{\mathsf{T}_h\left[u^k-f +\alpha d^k\right]}{d^k}_2 = & \ \iprd{\mathsf{T}_h\left[u^k-f\right]}{d^k}_2 +  \alpha\iprd{\mathsf{T}_h\left[d^k\right]}{d^k}_2
	\\
= & \ \iprd{u^k-f}{\mathsf{T}_h\left[d^k\right]}_2 +  \alpha\iprd{d^k}{\mathsf{T}_h\left[d^k\right]}_2,
	\end{align*}
where we have use the linearity and symmetry properties of the $\msfT_h$ operator. These terms have only to be calculated once per line search, and can be efficiently computed using FFT. In fact, observe that we only need to compute $\mathsf{T}_h\left[d^k\right]$, at the cost of a single FFT, per line search!

The $4$-Laplacian term in \eqref{scheme-spfc} gives preference to rotationally invariant patterns with square symmetry. We perform a simple test showing the emergence of these patterns in this subsection. The initial data for those simulations are similar to \eqref{eqn:init2}, but we add nucleation sites at specific locations in the domain. The rest of the parameters are given by $\varepsilon = 1.0$; $\lambda = \gamma_0 = 0.5$; $\gamma_1=2.0$; $\Omega =(0,100)^2$; and $s=0.01$. The time snapshots of the evolution by using the given parameters are presented in Figures~\ref{fig:long-time-spfc-one} (one nucleation site) and \ref{fig:long-time-spfc-four} (four nucleation sites). These tests confirm the emergence of the rotationally invariant square-symmetry patterns in the density field $u$.

\begin{figure}[h]
	\begin{center}
		\begin{subfigure}{0.48\textwidth}
			\includegraphics[height=0.48\textwidth,width=0.48\textwidth]{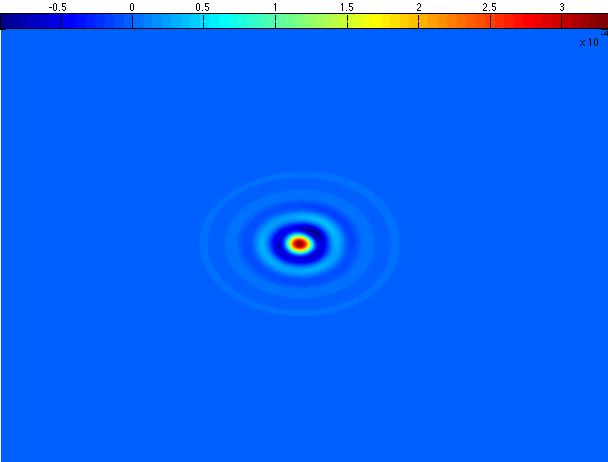} 
			\includegraphics[height=0.48\textwidth,width=0.48\textwidth]{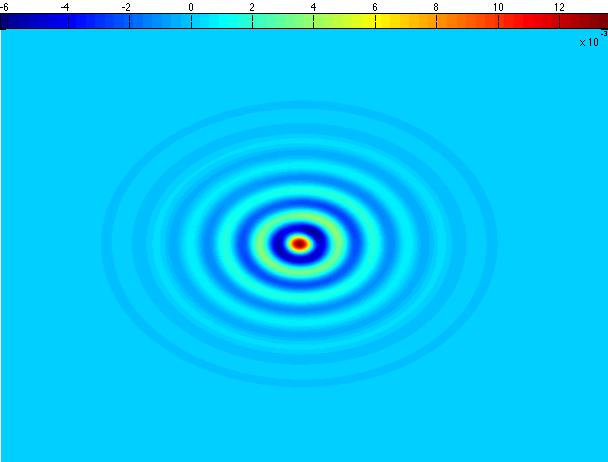} 
			\caption*{$t=1, 10$}
		\end{subfigure}
		\begin{subfigure}{0.48\textwidth}
			\includegraphics[height=0.48\textwidth,width=0.48\textwidth]{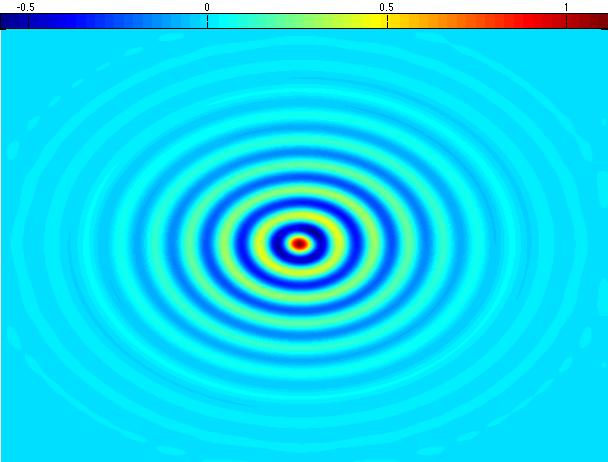} 
			\includegraphics[height=0.48\textwidth,width=0.48\textwidth]{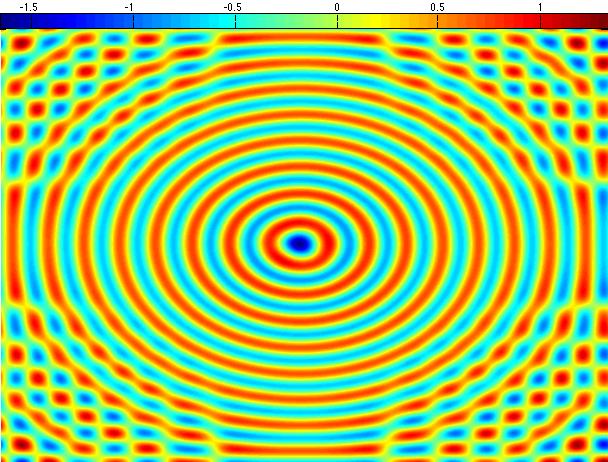}
			\caption*{$t=20, 40$}
		\end{subfigure}
		\begin{subfigure}{0.48\textwidth}
			\includegraphics[height=0.48\textwidth,width=0.48\textwidth]{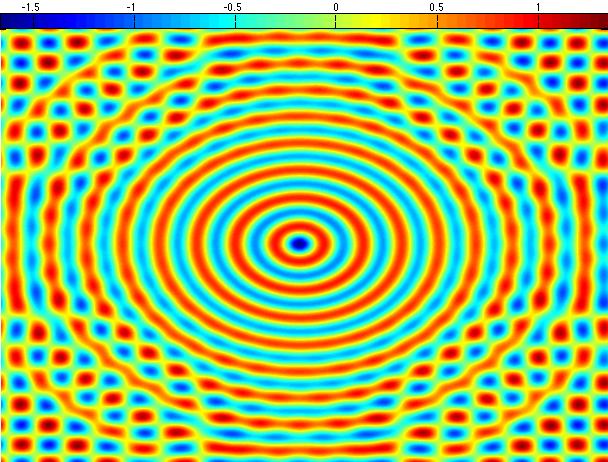} 
			\includegraphics[height=0.48\textwidth,width=0.48\textwidth]{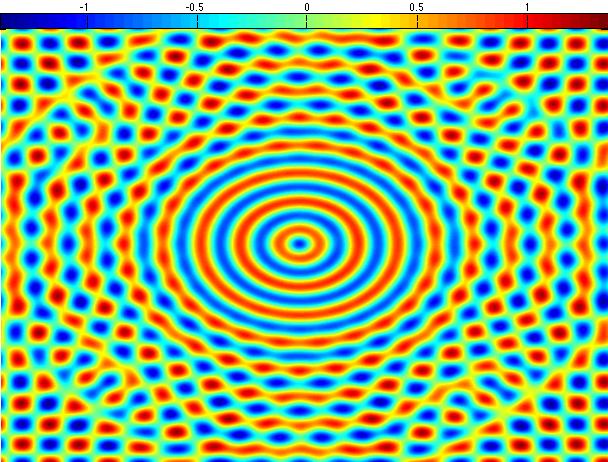}
			\caption*{$t=60, 80$}
		\end{subfigure}
		\begin{subfigure}{0.48\textwidth}
			\includegraphics[height=0.48\textwidth,width=0.48\textwidth]{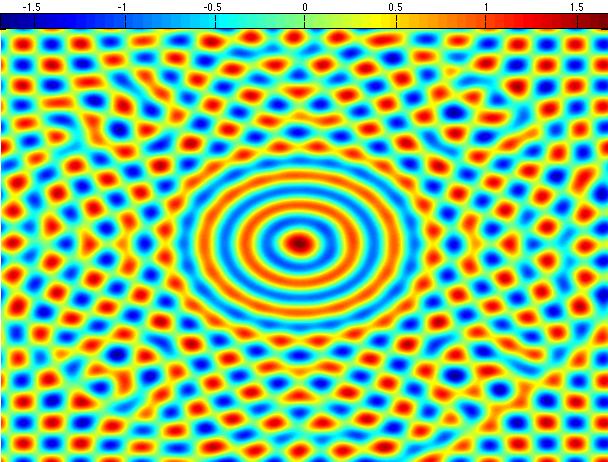}
			\includegraphics[height=0.48\textwidth,width=0.48\textwidth]{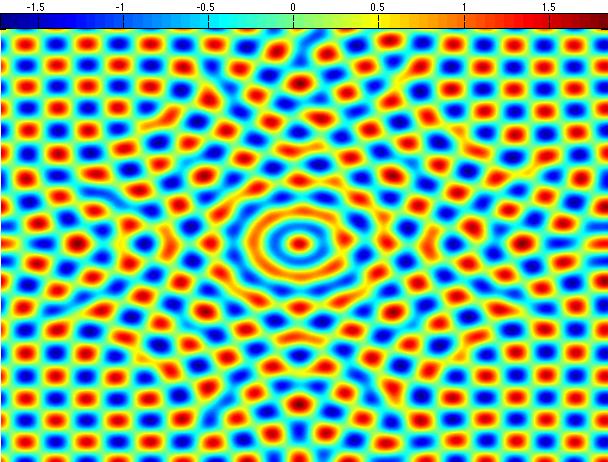}
			\caption*{$t=100,200$}
		\end{subfigure}
		\begin{subfigure}{0.48\textwidth}
			\includegraphics[height=0.48\textwidth,width=0.48\textwidth]{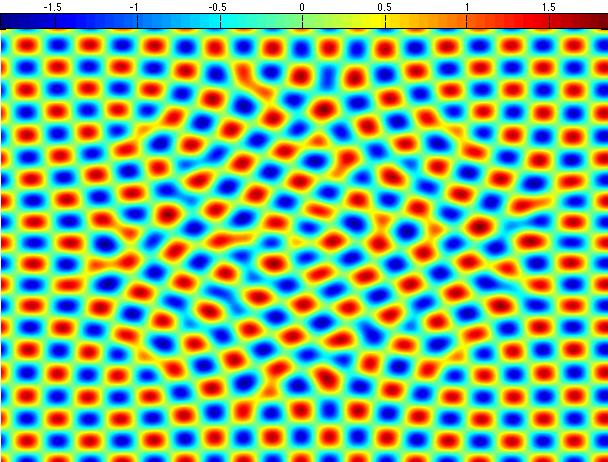} 
			\includegraphics[height=0.48\textwidth,width=0.48\textwidth]{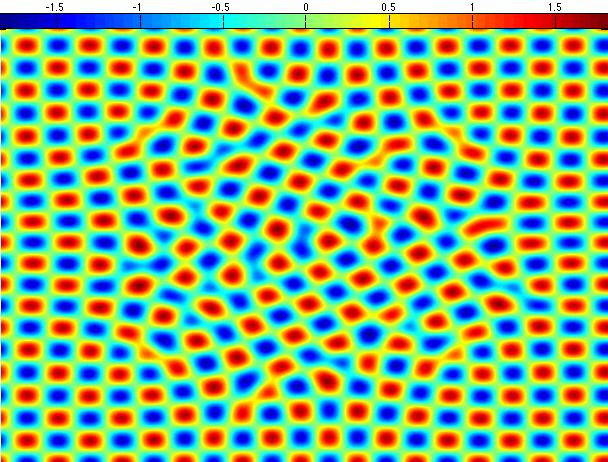}
			\caption*{$t=500, 1000$}
		\end{subfigure}
		\begin{subfigure}{0.48\textwidth}
			\includegraphics[height=0.48\textwidth,width=0.48\textwidth]{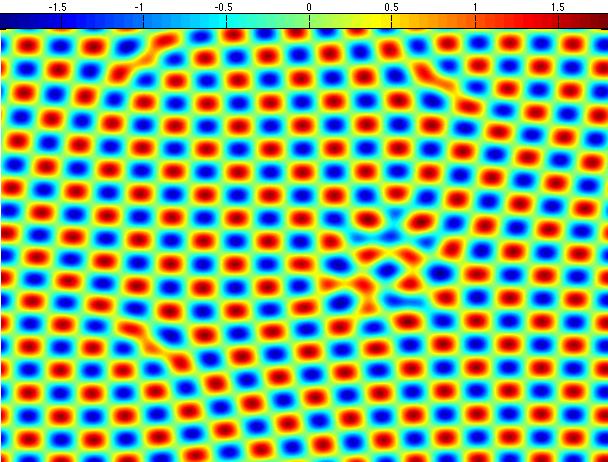} 
			\includegraphics[height=0.48\textwidth,width=0.48\textwidth]{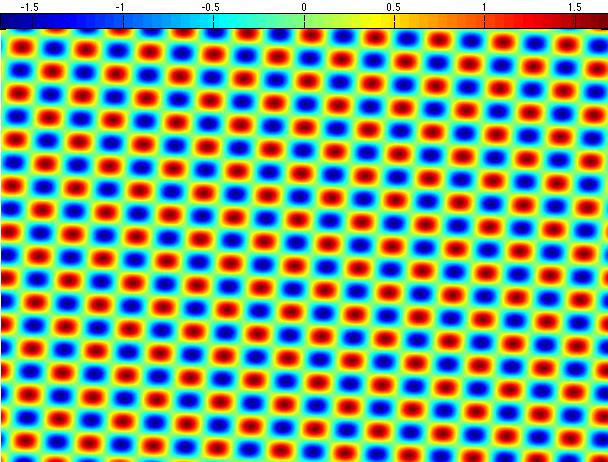}
			\caption*{$t=5000, 9000$}
		\end{subfigure}
		\caption{Time snapshots of the evolution with PSD solver for squared phase field crystal model at $t=1,10, 20, 40, 60, 80, 100, 200, 500, 1000, 5000~ \text{and}~ 9000$. The parameters are
			$\epsilon = 1.0, \lambda =0.5, \gamma_1= 2.0, \Omega=[100]^2$ and  $s=0.01$.}
		\label{fig:long-time-spfc-one}
	\end{center}
\end{figure}

\begin{figure}[h]
	\begin{center}
		\begin{subfigure}{0.48\textwidth}
			\includegraphics[height=0.48\textwidth,width=0.48\textwidth]{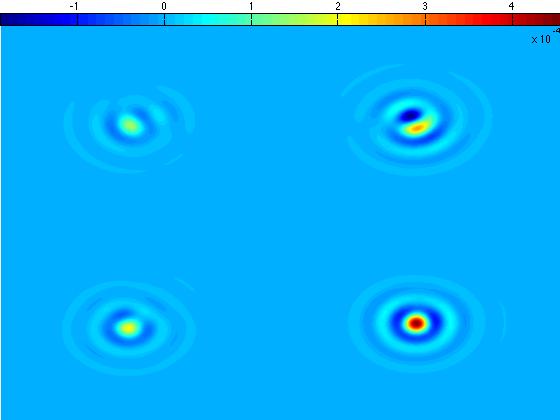} 
			\includegraphics[height=0.48\textwidth,width=0.48\textwidth]{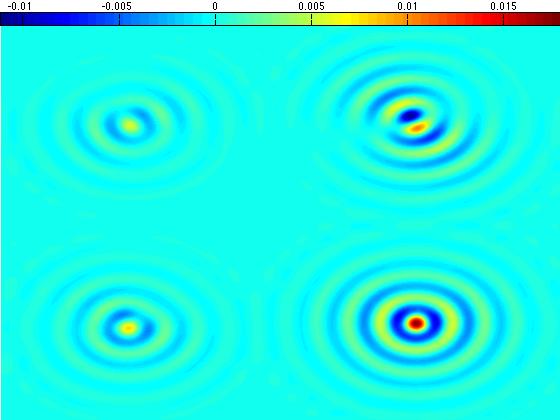} 
			\caption*{$t=1, 10$}
		\end{subfigure}
		\begin{subfigure}{0.48\textwidth}
			\includegraphics[height=0.48\textwidth,width=0.48\textwidth]{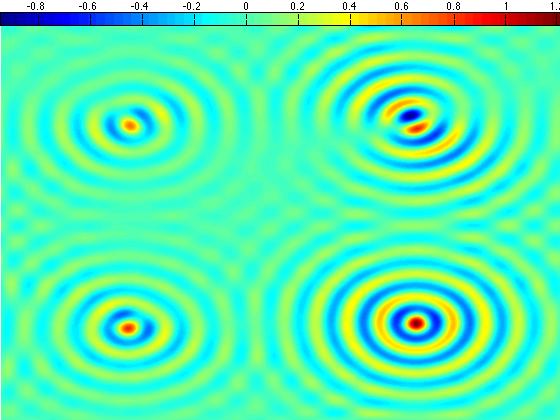} 
			\includegraphics[height=0.48\textwidth,width=0.48\textwidth]{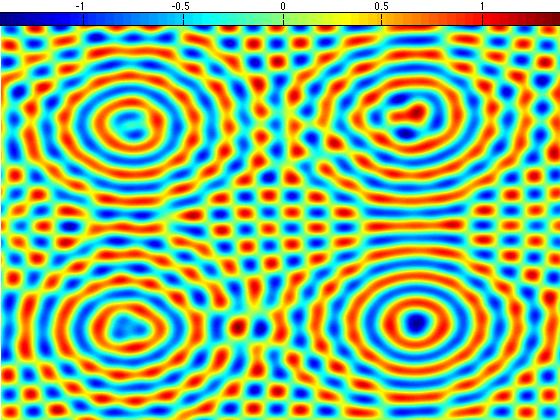}
			\caption*{$t=20, 40$}
		\end{subfigure}
		\begin{subfigure}{0.48\textwidth}
			\includegraphics[height=0.48\textwidth,width=0.48\textwidth]{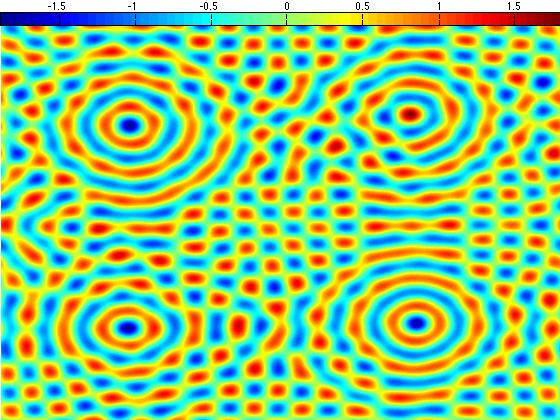} 
			\includegraphics[height=0.48\textwidth,width=0.48\textwidth]{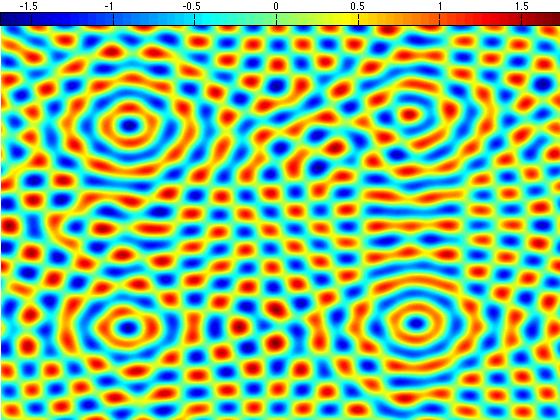}
			\caption*{$t=60, 80$}
		\end{subfigure}
		\begin{subfigure}{0.48\textwidth}
			\includegraphics[height=0.48\textwidth,width=0.48\textwidth]{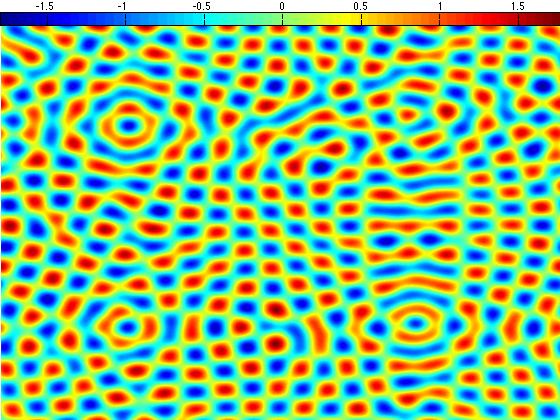}
			\includegraphics[height=0.48\textwidth,width=0.48\textwidth]{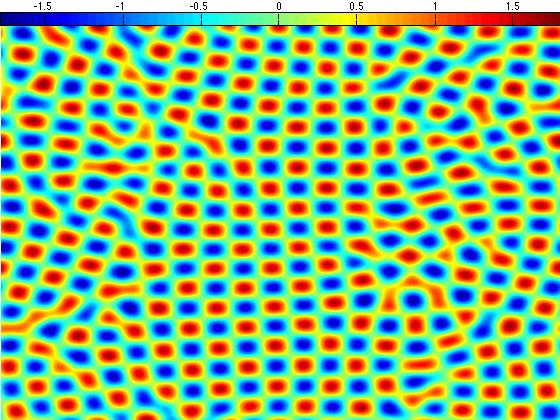}
			\caption*{$t=100,200$}
		\end{subfigure}
		\begin{subfigure}{0.48\textwidth}
			\includegraphics[height=0.48\textwidth,width=0.48\textwidth]{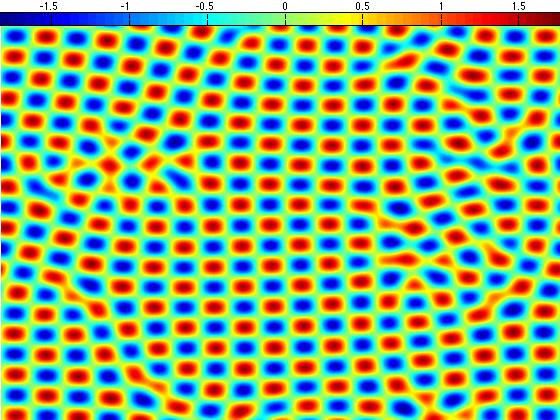} 
			\includegraphics[height=0.48\textwidth,width=0.48\textwidth]{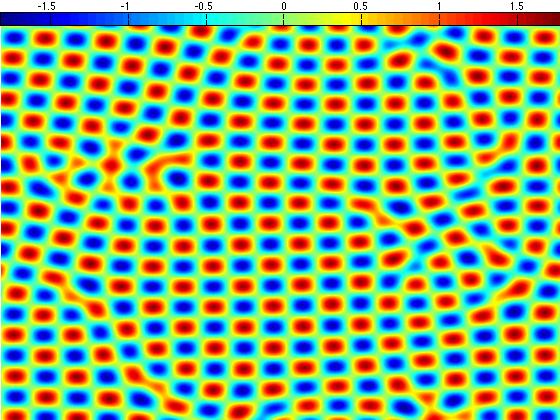}
			\caption*{$t=600, 800$}
		\end{subfigure}
		\begin{subfigure}{0.48\textwidth}
			\includegraphics[height=0.48\textwidth,width=0.48\textwidth]{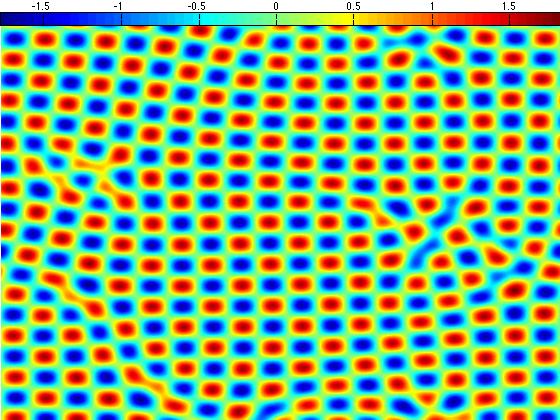} 
			\includegraphics[height=0.48\textwidth,width=0.48\textwidth]{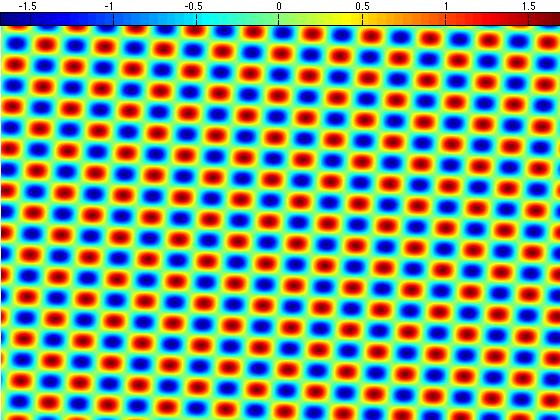}
			\caption*{$t=1000, 3000$}
		\end{subfigure}
		\caption{Time snapshots of the evolution with PSD solver for squared phase field crystal model at $t=1,10, 20, 40, 60, 80, 100, 200, 600, 800, 1000~ \text{and}~ 3000$. The parameters are
			$\epsilon = 1.0, \lambda =0.5, \gamma_1= 2.0, \Omega=[100]^2$ and  $s=0.01$.}
		\label{fig:long-time-spfc-four}
	\end{center}
\end{figure}

	\section{Summary}
	\label{sec:sum}
A preconditioned steepest descent (PSD) solver is proposed and analyzed for fourth and sixth-order regularized p-Laplacian equations. Solution of the highly nonlinear equations are equivalent to the minimizations of the associated convex energies. The energy dissipation property of the PSD solver leads to a bound for the numerical solution at each iteration stage. This fact, coupled with an upper-bound for the second derivative of the energy with respect to the metric induced by the pre-conditioner, leads to a geometric convergence rate for our (PSD) solver, which is proved rigorously for both the continuous and discrete space cases. Various numerical results are presented in this article, including a convergence test and a complexity analysis for the PSD solver, as well as long-time simulation results for the thin film epitaxy model with slope selection (both $p=4$ and $p=6$) and the square phase field crystal model.

	\section*{Acknowledgments}
This work is supported in part by the grants NSF DMS-1418689 (C. Wang),  NSF DMS-1418692 (S.M. Wise) and NSF DMS-1418784 (A.J. Salgado).
	
\appendix
\section{Proof of discrete Sobolev inequality}

Herein we only present the proof of (\ref{dis-sob-inq}) in Lemma \ref{lem:dis-lemma3} with $d=2$ and $p=4$. The other cases can be handled in the same way. Without loss of generality, we assume that $m=N=2K+1$ is odd and $L_x = L_y = L$, so that $h = \frac{L}{N}= \frac{L}{2K +1}$. We use $N$, rather than $m$, for the mesh size, as it is more standard.

For simplicity of presentation, we are focused on the estimate of $\| D_x  u  \|_4$, and we aim to establish the following estimate:
  \begin{equation*}
  \nrm{D_x  u }_4 \le C_0^{(1)} \nrm{ u }_2^{\frac14} \cdot \nrm{\Delta_h  u }_2^{\frac34}  , \quad \forall \, u\in \mathring{\mathcal C}_{\rm per}
	\end{equation*}
where $C_0^{(1)}>0$ depends upon $L$, but is independent of $h$ and $u$.  Due to the periodic boundary conditions for $ u $ and its cell-centered representation, it has a corresponding discrete Fourier transformation:
  	\begin{eqnarray*}
   u _{i,j} &=& \sum^{K}_{\ell,m=-K}
  \hat{ u }^N_{\ell,m} {\rm e}^{2 \pi i ( \ell x_{i} + m y_{j} )/ L }  ,
  	\end{eqnarray*}
  where $x_{i} = (i - \frac12 ) h$, $y_{j} = ( j - \frac12) h$, and $\hat{ u }^N_{\ell,m}$ 
  are discrete Fourier coefficients. Then we make its extension to a space-continuous function:
  	\begin{equation*}
   u _{{\bf F}}(x,y) = \sum^{K}_{\ell,m=-K} \hat{ u }^N_{\ell,m} {\rm e}^{2 \pi i ( \ell x + m y)/ L }  .
  	\end{equation*}
  	
Similarly, we denote the grid function $f := D_x  u \in \mathcal{E}_{\rm per}^{\rm ew}$.  The periodic boundary conditions for $f$ and its (east-west-edge-centered) mesh location indicates the following discrete Fourier transformation:
  	\begin{eqnarray*}
  f_{i+1/2,j} &=& \sum^{K}_{\ell,m=-K}
  \hat{f}^N_{\ell,m} {\rm e}^{2 \pi i ( \ell x_{i+1/2} + m y_{j} )/ L } , 
  	\end{eqnarray*}
with $\hat{f}^N_{\ell,m}$ the discrete Fourier coefficients. Its extension to a space-continuous function is given by 
  	\begin{equation*}
  f_{{\bf F}}(x,y) = \sum^{K}_{\ell,m=-K} \hat{f}^N_{\ell,m} {\rm e}^{2 \pi i ( \ell x + m y)/ L }  .
  	\end{equation*}
  	
Meanwhile, we  observe that $\hat{ u }^N_{0,0}=0$ and $\hat{f}^N_{0,0}=0$. The first identity comes from the fact that $\overline{ u }=0$, while the second one is due to the fact that $\overline{f} = \overline{D_x  u } =0$, for any periodic grid function $u$. 
  	
The following preliminary estimates will play a very important role in the later analysis. 
  	
  	\begin{lem} \label{lem:lemma 2.1}
  	We have 
  	\begin{eqnarray} 
  	  &&
  	  \|  u  \|_2 = \|  u _{\bf F} \| ,  
  	  \label{lemma 2.1-1} 
  	\\
  	  && 
  	  \frac{4}{\pi^2} \| \Delta  u _{\bf F} \| \le \nrm{\Delta_h  u }_2 \le \| \Delta  u _{\bf F} \| ,  
  	  \label{lemma 2.1-2} 
  	\\
  	  && 
  	  \nrm{\partial_x f_{\bf F} } \le \nrm{\partial_x^2  u _{\bf F} } ,  \quad 
  	  \nrm{\partial_y f_{\bf F} } \le \nrm{\partial_x \partial_y  u _{\bf F} } ,  
  	  \label{lemma 2.1-3} 
  	\\
  	  &&
  	   \nrm{f_{\bf F} }_{\mathring{H}^{-1}_{\rm per}} \le \nrm{ u _{\bf F} } .  
  	   \label{lemma 2.1-4}
  	\end{eqnarray}    
  	\end{lem}

  		\begin{proof}
  	Parseval's identity (at both the discrete and continuous levels) implies that
  	\begin{eqnarray}
  	   \sum^{N-1}_{i,j=0}| u _{i,j}|^2 &=&  N^2 \sum^{K}_{\ell,m=-K} 
  	|\hat{ u }^N_{\ell,m}|^2, \nonumber
  	\\
  	  \nrm{ u _{\bf F}}^2 &=& L^2  \sum^{K}_{\ell,m=-K}
  	|\hat{ u }^N_{\ell,m}|^2. 
  	\label{lemma 2.1-0-2} 
  	\end{eqnarray}
  	Based on the fact that $h N = L$, this in turn results in
  	\begin{equation*}
  	\nrm{ u }^2_{2} = \nrm{ u _{{\bf F}}}^2 =  L^2 \sum^{K}_{\ell,m=-K}|\hat{ u }^N_{\ell,m}|^2 ,  
  	\end{equation*}
  	so that \eqref{lemma 2.1-1} is proven. 
  	
  	For the comparison between $f = D_x  u $ and $\partial_x  u _{\bf F}$, we look at the following Fourier expansions:
  		\begin{eqnarray*}
  	f_{i+1/2,j} = (D_x  u )_{i+1/2,j} &=& \frac{ u _{i+1,j}- u _{i,j}}{h}
  		\nonumber
  		\\
  	&=& \sum^{K}_{\ell,m=-K}  w _{\ell} \hat{ u }^N_{\ell,m}  {\rm e}^{2 \pi i  ( \ell x_{i+1/2} + m y_{j} )/ L }  ,  
  		\\
  	f_{\bf F} (x,y) &=& \sum^{K}_{\ell,m=-K}  w _{\ell} \hat{ u }^N_{\ell,m}  {\rm e}^{2 \pi i  ( \ell x + m y )/ L }  ,  
  	\\	
  	  \partial_x  u _{{\bf F}} (x,y) &=& \sum^{K}_{\ell,m=-K}  
  	  \nu_{\ell} \hat{ u }^N_{\ell,m} {\rm e}^{2 \pi i ( \ell x + m y ) / L } , 
  		\end{eqnarray*}
  	with
  		\begin{equation*}
  	 w _{\ell} = -\frac{2 i \sin{\frac{\ell\pi h}{L}}}{h}, \quad
  	\nu_{\ell} = -\frac{2 \ell\pi i}{L}. 
  	\end{equation*}
  	A comparison of Fourier eigenvalues between $| w _{\ell}|$ and $|\nu_{\ell}|$ shows that
  	\begin{equation}
  	\frac{2}{\pi} |\nu_{\ell}| \le | w _{\ell}| \le |\nu_{\ell}|, 
  	\quad \rm{for}  \quad -K \le {\ell} \le K . 
  	\label{lemma 2.1-10}
  	\end{equation}
  	
  	For the estimate \eqref{lemma 2.1-2}, we look at similar Fourier expansions: 
  		\begin{eqnarray*}
  	(\Delta_h  u )_{i,j} &=& \sum^{K}_{\ell,m=-K} \left(  w _{\ell}^2 +  w _m^2 \right) \hat{ u }^N_{\ell,m}  {\rm e}^{2 \pi i  ( \ell x_{i} + m y_{j} )/ L }  ,  
  		\\	
  	  \Delta  u _{{\bf F}} (x,y) &=& \sum^{K}_{\ell,m=-K}  
  	  \left( \nu_{\ell}^2 + \nu_m^2 \right) \hat{ u }^N_{\ell,m} {\rm e}^{2 \pi i ( \ell x + m y ) / L } .  
  		\end{eqnarray*}
  	In turn, an application of Parseval's identity yields
  	\begin{eqnarray}
  	\nrm{\Delta_h  u }^2_2 = L^2 \sum^{K}_{\ell,m=-K}\left|  w _{\ell}^2 +  w _m^2 \right|^2 |\hat{ u }^N_{\ell,m}|^2,  
  	\label{lemma 2.1-13-1} 
  	 \\
  	\nrm{\Delta  u _{\bf F}}^2 =
  	 L^2 \sum^{K}_{\ell,m=-K} \left| \nu_{\ell}^2 + \nu_m^2 \right|^2
  	  |\hat{ u }^N_{\ell,m}|^2. 
  	  \label{lemma 2.1-13-2} 
  	\end{eqnarray}
  	The eigenvalue comparison estimate \eqref{lemma 2.1-10} implies the following inequality: 
  	\begin{equation}
  	\frac{4}{\pi^2} \left| \nu_{\ell}^2 + \nu_m^2 \right| \le \left|  w _{\ell}^2 +  w _m^2 \right| \le \left| \nu_{\ell}^2 + \nu_m^2 \right|, 
  	\quad \rm{for}  \quad -K \le \ell, m  \le K . 
  	\label{lemma 2.1-14} 
  	\end{equation}
  	As a result, inequality \eqref{lemma 2.1-2} comes from a combination of \eqref{lemma 2.1-13-1}, \eqref{lemma 2.1-13-2} and \eqref{lemma 2.1-14}. 
  	
  	For the estimate \eqref{lemma 2.1-3}, we observe the following Fourier expansions: 
  		\begin{eqnarray*}
  	 \partial_x f_{\bf F} (x,y) &=& \sum^{K}_{\ell,m=-K} \nu_{\ell}  w _{\ell} \hat{ u }^N_{\ell,m}  {\rm e}^{2 \pi i  ( \ell x + m y )/ L }  ,  
  		\\
  	\partial_x^2  u _{\bf F} (x,y) &=& \sum^{K}_{\ell,m=-K}  
  	  \nu_{\ell}^2 \hat{ u }^N_{\ell,m} {\rm e}^{2 \pi i ( \ell x + m y ) / L } , 
  		\end{eqnarray*}
  	which in turn leads to (with an application of Parseval's identity) 
  	\begin{eqnarray}
  	\nrm{\partial_x f_{\bf F}}^2 = L^2 \sum^{K}_{\ell,m=-K}\left| \nu_\ell  w _{\ell} \right|^2 |\hat{ u }^N_{\ell,m}|^2,  
  	\label{lemma 2.1-17-1}  
  	\\
  	\nrm{\partial_x^2  u _{\bf F}}^2 =
  	 L^2 \sum^{K}_{\ell,m=-K} | \nu_{\ell} |^4 |\hat{ u }^N_{\ell,m}|^2. 
  	 \label{lemma 2.1-17-2} 
  	\end{eqnarray}
  	Similarly, the following inequality could be derived, based on the eigenvalue comparison estimate \eqref{lemma 2.1-10}: 
  	\begin{equation}
  	  \left| \nu_\ell  w _{\ell} \right|^2 \le | \nu_\ell |^4 ,   
  	\quad \rm{for}  \quad -K \le \ell, m  \le K . 
  	\label{lemma 2.1-18} 
  	\end{equation}
  	Consequently, a combination of \eqref{lemma 2.1-17-1}, \eqref{lemma 2.1-17-2} and \eqref{lemma 2.1-18} leads to the first inequality in \eqref{lemma 2.1-3}. The second inequality, $\nrm{\partial_y f_{\bf F} } \le \nrm{\partial_x \partial_y  u _{\bf F} }$, could be derived in the same manner. 
  	
  	For the last estimate \eqref{lemma 2.1-4}, we observe that 
  	\begin{eqnarray*} 
  	  \nrm{f_{\bf F}}_{\mathring{H}^{-1}_{\rm per}}^2 = L^2 \sum^{K}_{(\ell,m) \ne \0, \ell,m=-K}  \frac{1}{| \nu_\ell^2 + \nu_m^2 |}  \cdot |  w _{\ell} |^2 |\hat{ u }^N_{\ell,m}|^2 .  
  	\end{eqnarray*}
  	Meanwhile, the derivation of the following inequality is straight forward: 
  	\begin{eqnarray*} 
  	  \frac{1}{| \nu_\ell^2 + \nu_m^2 |}   \cdot |  w _{\ell} |^2 
  	  = \frac{ |  w _{\ell} |^2 }{ |  \nu_\ell^2 + \nu_m^2 |} 
  	  \le \frac{ | \nu_\ell |^2 }{ |  \nu_\ell |^2 } \le 1 ,  \quad \forall (\ell,m) \ne \0 , 
  	\end{eqnarray*}
  	in which the eigenvalue estimate \eqref{lemma 2.1-10} was used again in the second step. In comparison with \eqref{lemma 2.1-0-2}, we arrive at \eqref{lemma 2.1-4}. The proof of Lemma \ref{lem:lemma 2.1} is complete.   
  	\end{proof}

   The following lemma gives a bound of the discrete $\|\cdot\|_4$ norm of the grid function $f$, in terms of the continuous $L^4$ norm of its continuous version $f_{\bf F}$. 
   
   \begin{lem} \label{lem:lemma 2.2} 
     We have 
   \begin{eqnarray} 
     \| f \|_4 \le \sqrt{2} \| f_{\bf F} \|_{L^4} .  
     \label{lemma 2.2-0} 
   \end{eqnarray} 
   \end{lem}

   \begin{proof} 
   We denote the following grid function 
   \begin{equation*} 
     g_{i+1/2,j} = \left( f_{i+1/2,j} \right)^2 .
   \end{equation*} 
   A direct calculation shows that 
   \begin{equation} 
     \nrm{ f }_4 = \left( \nrm{ g }_2 \right)^{\frac12} . 
     \label{lemma 2.2-2} 
   \end{equation}  
   Note that both norms are discrete in the above identity. Moreover, we assume the grid function $g$ has a discrete Fourier expansion as 
   \begin{equation*}
     g_{i+1/2,j} = \sum_{\ell,m=-K}^{K}
      (\hat{g}^N_c)_{\ell,m} \mathrm{e}^{2\pi {\rm i}  (\ell x_{i+1/2} + m y_{j} )} ,  
   \end{equation*}
   and denote its continuous version as 
   \begin{equation*}
     G (x,y) = \sum_{\ell,m=-K}^{K}
      (\hat{g}^N_c)_{\ell,m} \mathrm{e}^{2\pi {\rm i} (\ell x + m y )}  
    \in {\cal P}_{K} = \textrm{span}\left\{ \mathrm{e}^{2\pi {\rm i} (\ell x + m y )}: \ell,m = -K, \ldots, K  \right\}.  
   \end{equation*} 
   With an application of the Parseval equality at both the discrete and continuous levels, we have 
   \begin{equation} 
     \nrm{ g }_2^2 = \nrm{ G }^2 
     = \sum_{\ell,m=-K}^{K}  \left|  (\hat{g}^N_c)_{\ell,m}  \right|^2  .  
     \label{lemma 2.2-5}
   \end{equation}
   
     On the other hand, we also denote 
   \begin{equation*} 
     H (x,y) = \left( f_{\bf F} (x,y) \right)^2 
     = \sum_{\ell,m=-2K}^{2K}
      (\hat{h}^N)_{\ell,m} \mathrm{e}^{2\pi {\rm i} (\ell x + m y )}  
    \in {\cal P}_{2K} .  
    \label{lemma 2.2-6}
   \end{equation*}
   The reason for $H \in {\cal P}_{2K}$ is because $f_{\bf F} \in {\cal P}_{K}$. We note that $H \ne G$, since $H \in {\cal P}_{2K}$, while $G \in {\cal P}_{K}$, although $H$ and $G$ have the same interpolation values on at the numerical grid points $(x_{i}, y_{j+1/2})$. In other words, $g$ is the interpolation of $H$ onto the numerical grid point and $G$ is the continuous version of $g$ in ${\cal P}_{K}$. As a result, collocation coefficients $\hat{g}_c^N$ for $G$ are not equal to $\hat{h}^N$ for $H$, due to the aliasing error. In more detail, for $- K \le \ell, m \le K$, we have the following representations: 
   \begin{eqnarray*} 
     ( \hat{g}_c^N )_{\ell,m} =  \left\{  \begin{array}{l} 
         (\hat{h}^N)_{\ell,m} + (\hat{h}^N)_{\ell+N,m} 
     + (\hat{h}^N)_{\ell,m+N} + (\hat{h}^N)_{\ell+N,m+N} ,  \, \,  
      \ell < 0 , m < 0 , 
   \\
         (\hat{h}^N)_{\ell,m} + (\hat{h}^N)_{\ell+N,m} ,  \, \,  
      \ell < 0 , m = 0 ,  
   \\
      (\hat{h}^N)_{\ell,m} + (\hat{h}^N)_{\ell+N,m} 
     + (\hat{h}^N)_{\ell,m-N} + (\hat{h}^N)_{\ell+N,m-N} ,  \, \,  
     \ell < 0 , m > 0   , 
   \\
      (\hat{h}^N)_{\ell,m} + (\hat{h}^N)_{\ell -N,m} 
     + (\hat{h}^N)_{\ell,m-N} + (\hat{h}^N)_{\ell-N,m-N} ,  \, \,  
     \ell > 0 , m > 0 , 
   \\
      (\hat{h}^N)_{\ell,m} + (\hat{h}^N)_{\ell-N,m} ,  \, \,  
     \ell > 0 , m = 0 , 
   \\
      (\hat{h}^N)_{\ell,m} + (\hat{h}^N)_{\ell-N,m} 
     + (\hat{h}^N)_{\ell,m+N} + (\hat{h}^N)_{\ell-N,m+N} ,  \, \,  
     \ell > 0 , m < 0 , 
   \\ 
         (\hat{h}^N)_{\ell,m} + (\hat{h}^N)_{\ell,m+N} ,  \, \,  
      \ell = 0 , m < 0 , 
   \\
         (\hat{h}^N)_{\ell,m}  ,  \, \,  
      \ell = 0 , m = 0 ,  
   \\
      (\hat{h}^N)_{\ell,m} + (\hat{h}^N)_{\ell,m-N}  ,  \, \,  
     \ell = 0 , m > 0  . 
   \end{array}  \right. 
   \end{eqnarray*}
   With an application of Cauchy inequality, it is clear that 
   \begin{equation} 
     \sum_{\ell,m=-K}^{K}  
     \left|  (\hat{g}^N_c)_{\ell,m}  \right|^2  
     \le 4 \left| \sum_{\ell,m=-2K}^{2K}
      (\hat{h}^N)_{\ell,m}  \right|^2 .  
      \label{lemma 2.2-8}
   \end{equation}
   
     Meanwhile, an application of Parseval's identity to the Fourier expansion (\ref{lemma 2.2-6}) gives 
   \begin{equation*} 
     \nrm{ H }^2 = \left| \sum_{\ell,m=-2K}^{2K}
      (\hat{h}^N)_{\ell,m}  \right|^2 .  
   \end{equation*}
   Its comparison with (\ref{lemma 2.2-5}) indicates that 
   \begin{equation} 
     \nrm{ g }_2^2 = \nrm{ G }^2 
      \le 4 \nrm{ H }^2  ,  \quad \mbox{i.e.} \, \, 
     \nrm{ g }_2 \le 2 \nrm{ H } ,  
     \label{lemma 2.2-10}
   \end{equation}
   with the estimate (\ref{lemma 2.2-8}) applied. Meanwhile, since $H (x,y) = \left( f_{\bf F} (x,y) \right)^2$, we have 
   \begin{equation} 
     \nrm{ f_{\bf F} }_{L^4} = \left( \nrm{ H }_{} \right)^{\frac12} . 
     \label{lemma 2.2-11}
   \end{equation}
   Therefore, a combination of (\ref{lemma 2.2-2}), (\ref{lemma 2.2-10}) and (\ref{lemma 2.2-11}) results in 
   \begin{equation*} 
     \nrm{ f }_4 = \left( \nrm{ g }_2 \right)^{\frac12} 
     \le \left( 2 \nrm{ H }_{} \right)^{\frac12} 
     \le \sqrt{2} \nrm{ f_{\bf F} }_{L^4}  . 
   \end{equation*}
   This finishes the proof of (\ref{lemma 2.2-0}). 
   \end{proof} 
   
   Now we proceed into the proof of Proposition~\ref{lem:dis-lemma3}. 
   
   \begin{proof} 
     We begin with an application of (\ref{lemma 2.2-0}) in Lemma~\ref{lem:lemma 2.2}: 
   \begin{eqnarray} 
     \| D_x  u  \|_4 = \| f \|_4 \le  \sqrt{2} \| f_{\bf F} \|_{L^4} .  
     \label{prop 1.1-1} 
   \end{eqnarray}  
   Meanwhile, using the fact that $\overline{ f _{\bf F} }=0$, we apply the 2-D Sobolev inequality and get 
   \begin{eqnarray} 
     \| f_{\bf F} \|_{L^4} \le C \| f_{\bf F} \|_{H^{\frac{1}{2}}} \le C \| f_{\bf F} \|_{\mathring{H}^{-1}_{\rm per}}^{\frac{1}{4}} \cdot \| \nabla f_{\bf F} \|^{\frac{3}{4}}   .  
     \label{prop 1.1-2} 
   \end{eqnarray} 
   Moreover, the estimates (\ref{lemma 2.1-1}) -- (\ref{lemma 2.1-4}) (in Lemma~\ref{lem:lemma 2.1}) indicate that 
	\begin{eqnarray} 
	&& 
\| f_{\bf F} \|_{\mathring{H}^{-1}_{\rm per}} \le \|  u _{\bf F} \| = \|  u  \|_2 ,
	\label{prop 1.1-3-1} 
	\\
	&& 
\nrm{\partial_x f_{\bf F} } \le \nrm{\partial_x^2  u _{\bf F} } \le M_0 \nrm{\Delta  u _{\bf F} }  \le  \frac{\pi^2 M_0}{4} \| \Delta_h  u  \|_2 , 
	\nonumber
	\\
 	&& 
\nrm{\partial_y f_{\bf F} } \le \nrm{\partial_x \partial_y  u _{\bf F} } \le M_0 \nrm{\Delta  u _{\bf F} }  \le  \frac{\pi^2 M_0}{4} \| \Delta_h  u  \|_2 , 
	\nonumber
	\end{eqnarray}
so that
	\begin{equation}
\nrm{\nabla f_{\bf F} }  \le  \frac{\sqrt{2} \pi^2 M_0}{4} \| \Delta_h  u  \|_2 ,  
\label{prop 1.1-3-4} 
   \end{equation} 
where the following elliptic regularity estimate is applied:
	\[ 
\nrm{\partial_x^2  u _{\bf F} }  ,  \nrm{\partial_x \partial_y  u _{\bf F} } \le M_0 \nrm{\Delta  u _{\bf F} } .
	\]
Therefore, a substitution of (\ref{prop 1.1-3-1}), (\ref{prop 1.1-3-4}) and (\ref{prop 1.1-2}) into (\ref{prop 1.1-1}) results in 
	\[
\| D_x  u  \|_4  \le C_0^{(1)} \|  u  \|_2^{\frac{1}{4}} \cdot \| \Delta_h  u  \|_2^{\frac{3}{4}} ,  
      \quad \mbox{with} \, \, \,  C_0^{(1)} = 2^{-5/8} M_0^{\frac{3}{4}} \pi^{3/2} .
	\]
   
   The estimate for $\| D_y  u  \|_4$ could be derived in the same fashion. The result is stated below; its proof is skipped for the sake of brevity. 
	\[ 
\| D_y  u  \|_4  \le C_0^{(1)} \|  u  \|_2^{\frac{1}{4}} \cdot \| \Delta_h  u  \|_2^{\frac{3}{4}} .
	\] 
   
Moreover, by the definition of $\mD_x  u $ and $\mD_y  u $ we get  
	\[
\| \mD_x  u  \|_4 = \| A_y(D_x u ) \|_4 \le \| D_x  u  \|_4 ,  \quad  \| \mD_y  u  \|_4 = \| A_x(D_y u ) \|_4 \le \| D_y  u  \|_4 .
	\]
As a consequence, the first case of (\ref{dis-sob-inq}) (with $d=2$, $p=4$) is valid, by setting $C_0 = \sqrt{2} C_0^{(1)}$. The other cases could be analyzed in the same way. This finishes the proof of Proposition~\ref{lem:dis-lemma3}. 
	\end{proof} 
   
	\bibliographystyle{plain}
	\bibliography{PSD}
	\end{document}